\documentclass[letterpaper, 11pt,  reqno]{amsart}

\usepackage{amsmath,amssymb,amscd,amsthm,amsxtra, esint}

\usepackage[implicit=true]{hyperref}

\allowdisplaybreaks[2]

\usepackage[left=32mm, right=32mm, 
bottom=27mm]{geometry}

\usepackage{color}
\definecolor{gr}{rgb}   {0.,   0.69,   0.23 }
\definecolor{bl}{rgb}   {0.,   0.5,   1. }
\definecolor{mg}{rgb}   {0.85,  0.,    0.85}
\definecolor{yl}{rgb}   {0.8,  0.7,   0.}
\definecolor{or}{rgb}  {0.7,0.2,0.2}

\newtheorem{oldtheorem}{Theorem}

\newtheorem{theorem}{Theorem} [section]

\newtheorem{lemma}[theorem]{Lemma}
\newtheorem{proposition}[theorem]{Proposition}
\newtheorem{remark}[theorem]{Remark}

\newtheorem{definition}[theorem]{Definition}
\newtheorem{corollary}[theorem]{Corollary}

\DeclareMathOperator*{\intt}{\int}

\DeclareMathOperator*{\supp}{supp}

\newcommand{\1}{\hspace{0.5mm}\text{I}\hspace{0.5mm}}
\newcommand{\II}{\text{I \hspace{-2.8mm} I} }
\newcommand{\III}{\text{I \hspace{-2.9mm} I \hspace{-2.9mm} I}}

\newcommand{\noi}{\noindent}
\newcommand{\Z}{\mathbb{Z}}
\newcommand{\R}{\mathbb{R}}

\newcommand{\T}{\mathbb{T}}

\let\Re=\undefined\DeclareMathOperator*{\Re}{Re}
\let\Im=\undefined\DeclareMathOperator*{\Im}{Im}

\let\P= \undefined
\newcommand{\P}{\mathbf{P}}

\newcommand{\E}{\mathbb{E}}

\newcommand{\RR}{\mathcal{R}}

\newcommand{\F}{\mathcal{F}}

\newcommand{\al}{\alpha}

\newcommand{\dl}{\delta}

\newcommand{\nb}{\nabla}

\newcommand{\Dl}{\Delta}
\newcommand{\eps}{\varepsilon}

\newcommand{\g}{\gamma}
\newcommand{\G}{\Gamma}
\newcommand{\ld}{\lambda}
\newcommand{\Ld}{\Lambda}
\newcommand{\s}{\sigma}
\newcommand{\Si}{\Sigma}
\newcommand{\ft}{\widehat}

\newcommand{\wt}{\widetilde}
\newcommand{\cj}{\overline}
\newcommand{\dx}{\partial_x}
\newcommand{\dt}{\partial_t}
\newcommand{\dd}{\partial}

\newcommand{\embeds}{\hookrightarrow}

\newcommand{\ta}{\theta}

\renewcommand{\l}{\ell}
\renewcommand{\o}{\omega}
\renewcommand{\O}{\Omega}

\newcommand{\les}{\lesssim}
\newcommand{\ges}{\gtrsim}

\newcommand{\jb}[1]
{\langle #1 \rangle}

\newcommand{\ind}{\mathbf 1}

\renewcommand{\S}{\mathcal{S}}

\newcommand{\M}{\mathcal{M}}

\newcommand{\N}{\mathbb{N}}

\newcommand{\NN}{\mathcal{N}}

\newcommand{\Pk}{P^{(k+1)}_2}

\renewcommand{\H}{\mathcal{H}}

\newcommand{\I}{\mathcal{I}}

\newcommand{\TT}{\mathcal{T}}

\newcommand{\BT}{{\bf T}}

\newcommand{\too}{\longrightarrow}
\newcommand{\IP}{\mathfrak{I}}

\usepackage{tikz}

\usetikzlibrary{shapes.misc}
\usetikzlibrary{shapes.symbols}
\usetikzlibrary{shapes.geometric}
\tikzset{
	dot/.style={circle,fill=black,draw=black,inner sep=0pt,minimum size=0.5mm},
	>=stealth,
	}
\tikzset{
	ddot/.style={circle,fill=white,draw=black,inner sep=0pt,minimum size=0.8mm},
	>=stealth,
	}

\tikzset{decision/.style={ 
        draw,
        diamond,
        aspect=1.5
    }}

\tikzset{dia2/.style
={diamond,fill=white,draw=black,inner sep=0pt,minimum size=1mm},
	>=stealth,
	}

\tikzset{dia/.style
={star,fill=black,draw=black,inner sep=0pt,minimum size=1mm},
	>=stealth,
	}

\makeatletter
\def\DeclareSymbol#1#2#3{\expandafter\gdef\csname MH@symb@#1\endcsname{\tikz[baseline=#2,scale=0.15]{#3}}}
\def\<#1>{\csname MH@symb@#1\endcsname}
\makeatother

\DeclareSymbol{X}{-2.4}{\node[dot] {};}
\DeclareSymbol{1}{0}{\draw[white] (-.4,0) -- (.4,0); \draw (0,0)  -- (0,1.2) node[dot] {};}
\DeclareSymbol{2}{0}{\draw (-0.5,1.2) node[dot] {} -- (0,0) -- (0.5,1.2) node[dot] {};}

\DeclareSymbol{1}{-2.7}
 {
  \draw (0,0) node[dot]{};}

\DeclareSymbol{1'}{-2.7}
 {\draw (0,0) node[ddot]{};}

\DeclareSymbol{1''}{-2.7}
 {\draw (0,0) node[dia]{};}

\DeclareSymbol{3}{0}
 {\draw (0,0) node[dot]{} -- (0,1.2) node[ddot] {}; 
 \draw (-1.2,0) node[dot] {} -- (0,1.2)node[ddot] {} -- (1.2,0) node[dot] {};}

\DeclareSymbol{3'}{0}
 {\draw (0,0) node[dia]{} -- (0,1.2) node[ddot] {}; 
 \draw (-1.2,0) node[dia] {} -- (0,1.2)node[ddot] {} -- (1.2,0) node[dia] {};}

\DeclareSymbol{31}{-3}{
\draw (0,-1)node[dot] {} -- (0,0) node[ddot] {}-- (0.9, -1) node[dot] {}; 
\draw (0,-1)node[dot] {} -- (0,0) node[ddot] {}-- (-0.9, -1) node[dot] {}; 
\draw (0,0)node[ddot] {} -- (1.3,1) node[ddot] {}-- (1.3, 0) node[dot] {}; 
\draw  (1.3,1) node[ddot] {}-- (2.6, 0) node[dot] {}; 
}

\DeclareSymbol{32}{-3}{
\draw (0,-1)node[dot] {} -- (0,0) node[ddot] {}-- (0.9, -1) node[dot] {}; 
\draw (0,-1)node[dot] {} -- (0,0) node[ddot] {}-- (-0.9, -1) node[dot] {}; 
\draw (0,0)node[ddot] {} -- (0,1) node[ddot] {}-- (0.9, 0) node[dot] {}; 
\draw (0,0)node[ddot] {} -- (0,1) node[ddot] {}-- (-0.9, 0) node[dot] {}; 
}

\DeclareSymbol{33}{-3}{
\draw (2.6,-1)node[dot] {} -- (2.6,0) node[ddot] {}-- (3.5, -1) node[dot] {}; 
\draw (2.6,-1)node[dot] {} -- (2.6,0) node[ddot] {}-- (1.7, -1) node[dot] {}; 
\draw (0,0)node[dot] {} -- (1.3,1) node[ddot] {}-- (1.3, 0) node[dot] {}; 
\draw  (1.3,1) node[ddot] {}-- (2.6, 0) node[ddot] {}; 
}

\DeclareSymbol{3''}{0}
 {\draw (0,0) node[dia]{} -- (0,1.2) node[ddot] {}; 
 \draw (-1.2,0) node[dot] {} -- (0,1.2)node[ddot] {} -- (1.2,0) node[dia] {};}

\DeclareSymbol{3'''}{0}
 {\draw (0,0) node[dot]{} -- (0,1.2) node[ddot] {}; 
 \draw (-1.2,0) node[dot] {} -- (0,1.2)node[ddot] {} -- (1.2,0) node[dia] {};}

\DeclareSymbol{31'}{-3}{
\draw (0,-1.2)node[dia] {} -- (0,0) node[ddot] {}-- (0.9, -1.2) node[dia] {}; 
\draw (0,-1.2)node[dia] {} -- (0,0) node[ddot] {}-- (-0.9, -1.2) node[dia] {}; 
\draw (0,0)node[ddot] {} -- (1.3,1.2) node[ddot] {}-- (1.3, 0) node[dia] {}; 
\draw  (1.3,1.2) node[ddot] {}-- (2.6, 0) node[dia] {}; 
}

\DeclareSymbol{32'}{-3}{
\draw (0,-1.2)node[dia] {} -- (0,0) node[ddot] {}-- (0.9, -1.2) node[dia] {}; 
\draw (0,-1.2)node[dia] {} -- (0,0) node[ddot] {}-- (-0.9, -1.2) node[dia] {}; 
\draw (0,0)node[ddot] {} -- (0,1.2) node[ddot] {}-- (0.9, 0) node[dia] {}; 
\draw (0,0)node[ddot] {} -- (0,1.2) node[ddot] {}-- (-0.9, 0) node[dot] {}; 
}

\DeclareSymbol{33'}{-3}{
\draw (2.6,-1.2)node[dia] {} -- (2.6,0) node[ddot] {}-- (3.5, -1.2) node[dia] {}; 
\draw (2.6,-1.2)node[dia] {} -- (2.6,0) node[ddot] {}-- (1.7, -1.2) node[dia] {}; 
\draw (0,0)node[dot] {} -- (1.3,1.2) node[ddot] {}-- (1.3, 0) node[dot] {}; 
\draw  (1.3,1.2) node[ddot] {}-- (2.6, 0) node[ddot] {}; 
}

\newtheorem*{ackno}{Acknowledgments}

\numberwithin{equation}{section}
\numberwithin{theorem}{section}

\begin{document}
\baselineskip = 14pt

\title[
On  the two-dimensional Wick ordered cubic NLW]
{Uniqueness and non-uniqueness of the Gaussian free field evolution 
under the two-dimensional  Wick ordered cubic wave equation}

\author[T.~Oh, M.~Okamoto, and   N.~Tzvetkov]
{Tadahiro Oh, Mamoru Okamoto,  and Nikolay Tzvetkov}

\address{
Tadahiro Oh, School of Mathematics\\
The University of Edinburgh\\
and The Maxwell Institute for the Mathematical Sciences\\
James Clerk Maxwell Building\\
The King's Buildings\\
Peter Guthrie Tait Road\\
Edinburgh\\ 
EH9 3FD\\
 United Kingdom}

\email{hiro.oh@ed.ac.uk}

\address{
Mamoru Okamoto\\
Department of Mathematics\\
 Graduate School of Science\\ Osaka University\\
Toyonaka\\ Osaka\\ 560-0043\\ Japan}
\email{okamoto@math.sci.osaka-u.ac.jp}

\address{
Nikolay Tzvetkov\\
Ecole Normale Sup\'erieure de Lyon, UMPA, UMR CNRS-ENSL 5669, 46, all\'ee d'Italie, 69364-Lyon Cedex 07, France.}
%

\email{nikolay.tzvetkov@ens-lyon.fr}



\begin{abstract}
We study the nonlinear wave equation (NLW) on the two-dimensional torus~$\T^2$
with Gaussian random initial data on $H^s(\T^2) \times H^{s-1}(\T^2)$, $ s < 0$,
distributed according 
to the base Gaussian free field $\mu$
associated with  the invariant Gibbs measure 
studied 
by Thomann and the first author (2020).
In particular, we investigate the approximation property
of the corresponding solution by smooth (random) solutions.
Our main results in this paper are two-fold.
(i) We show that the  solution map for the renormalized cubic NLW
defined on the Gaussian free field $\mu$ 
is the unique extension of the solution map defined 
for  smoothed Gaussian initial data obtained by  mollification, 
independent of mollification kernels.
(ii)~We also show that there is a regularization
of the Gaussian initial data so that the corresponding smooth solutions
almost surely 
have no limit in the natural topology.
This second result in particular states that 
one can not use arbitrary smooth approximation
for the renormalized cubic NLW dynamics.

As a preliminary step for proving (ii), 
we  establish a (deterministic) norm inflation result 
at general initial data for the (unrenormalized) cubic NLW on $\T^d$ 
and $\R^d$ in negative Sobolev spaces, extending the norm inflation result by 
Christ, Colliander, and Tao (2003).

\end{abstract}


\maketitle

\tableofcontents

\section{Introduction}

\subsection{Nonlinear wave equations} 
We consider 
the  defocusing  nonlinear wave equation (NLW) on $\T^2 = (\R/\Z)^2$: 
\begin{align}
\begin{cases}
\dt^2 u  + (1 -  \Dl)  u    +  u^k  = 0\\
(u, \dt u) |_{t = 0} = (u_0, u_1), 
\end{cases}
\qquad (x, t) \in \T^2 \times \R,
\label{NLW0}
\end{align}

\noi
where  $k \geq 3$ is an odd integer
and the unknown function $u$ is real-valued.\footnote
{The equation \eqref{NLW0}
is also referred to as the nonlinear Klein-Gordon equation.
We, however, simply refer to \eqref{NLW0} as NLW in the following.
Moreover, we only consider real-valued functions in the following.
The modifications required to handle the complex-valued case are straightforward.
See \cite{OTh1}.}
In particular, 
we study  the  Cauchy problem\footnote{More precisely, we
study a renormalized version of \eqref{NLW0}.  See the Wick ordered NLW \eqref{WNLW0} 
below.}~\eqref{NLW0}
with Gaussian random initial data $(u_0^\o, u_1^\o)$ distributed
according to the massive  Gaussian free field\footnote{In fact, 
 $\mu$ is a measure on a vector $(u_0. u_1)$,  given as the tensor product
 of the mass Gaussian free fields on the $u_0$ component
 and the white noise measure on the $u_1$ component.
 For simplicity, however, we refer to $\mu$ as the (massive) Gaussian free field 
 in the following.}
$\mu$
on $\H^s(\T^2)  \stackrel{\text{def}}{=} H^s(\T^2) \times H^{s-1}(\T^2)$, $s < 0$, 
with  the covariance operator $(\text{Id}-\Dl)^{-1+s}$, 
whose density is formally given by\footnote{Henceforth, we use  
$Z$ to denote various  normalization constants
so that the corresponding measures are probability measures when appropriate.} 
\begin{align}
 d\mu =  Z^{-1} e^{-\frac 12 \int_{\T^2}(  u^2 + |\nb u |^2) dx} du\otimes e^{-\frac 12 \int_{\T^2} v^2dx} dv.
\label{G3}
\end{align}

\noi
This problem naturally appears in the study of invariant Gibbs measures
for \eqref{NLW0};
see the next subsection.
In particular, the (renormalized) NLW on $\T^2$
is known to be  almost surely globally well-posed
with respect to the massive Gaussian free field $\mu$ (see Theorem~\ref{THM:OTh} below).

Our main goal in this paper is to study the approximation
property of the  (random) solution 
to the renormalized NLW with\footnote{Given a random variable $X$, 
we use $\mathcal{L}(X)$ to denote the law (= distribution) of $X$.}
 $\mathcal{L}(u_0^\o, u_1^\o) = \mu$ 
(constructed in Theorem \ref{THM:OTh})
by smooth (random) solutions.
In other words, we are interested in understanding
the following question:
``In what sense is the solution map$: (u_0^\o, u_1^\o)\mapsto (u, \dt u)$
to the (renormalized) NLW
with $\mathcal{L}(u_0^\o, u_1^\o) = \mu$
  an extension of the solution map, a priori defined on smooth (random) initial data?''
A  natural way to study this question
 is to approximate the rough initial data by regular functions 
 and see whether the obtained sequence of smooth solutions converges to a unique limit (independent of the choice of the regularization).
   This is the strongest form of  uniqueness and it basically  holds 
   when the problem is  deterministically locally well-posed, 
allowing us to conclude that {\it any} approximation  
would give a good approximating sequence of smooth solutions, 
tending to the unique limit.  It turns out that for our problem  at hand with  $\mathcal{L}(u_0^\o, u_1^\o) = \mu$, 
 this strongest form of uniqueness does  not hold because of the low regularity of the initial data.
 See (ii) below.
This  gives rise to the ``non-uniqueness'' part in the title of this paper. 
On the other hand, if we restrict our attention to 
 regularization by convolution (which is a very particular way
 of  approximating the rough initial data), then 
 the sequence converges to a unique limit, 
 justifying the ``uniqueness'' part of the title.

In this paper, we will establish the following two claims:

\begin{itemize}
\item [(i)]
We show that the solution map$: (u_0^\o, u_1^\o)\mapsto (u, \dt u)$
to the (renormalized) NLW
with $\mathcal{L}(u_0^\o, u_1^\o) = \mu$
is the unique extension of the solution map defined 
for smoothed Gaussian initial data obtained by  mollification
(Theorem \ref{THM:uniq}).
Here, the uniqueness refers to the fact that 
the whole sequence of regularized solutions converges.
Note that convergence of a subsequence typically follows from weak solution
(= compactness) arguments; see \cite{BTT, BTT2, OTh1}.
Moreover, 
the limiting solution map 
is independent of mollification kernels.
See Theorem \ref{THM:uniq}.

\smallskip

\item [(ii)]
We  show that 
 there exists a regularization
of the Gaussian initial data $(u_0^\o, u_1^\o)$
with $\mathcal{L}(u_0^\o, u_1^\o) = \mu$
such that the corresponding smooth solutions
 almost surely have no limit in the natural topology;
 see Theorem~\ref{THM:ill0}.
We prove  this second result by establishing
{\it almost sure norm inflation} for the renormalized NLW
(Proposition~\ref{PROP:ill3}).

\end{itemize}

\noi
As a preliminary step for (ii), 
we prove (deterministic) norm inflation for  NLW
in negative Sobolev spaces
(Theorem \ref{THM:illposed})
by following the argument in \cite{Oh1}.
See Subsection~\ref{SUBSEC:illposed}.

\subsection{Invariant Gibbs measures}
With $v = \dt u$, 
we can write the equation~\eqref{NLW0}  in the following Hamiltonian formulation:
\begin{equation*}
 \dt 
 \begin{pmatrix}
 u \\ v 
 \end{pmatrix}
 =  
 \begin{pmatrix}
 0& 1 \\ -1 & 0
 \end{pmatrix}
\frac{\dd H}{\dd(u, v )}, 
\end{equation*}

\noi
where $H = H(u, v)$ is  the Hamiltonian given by 
\begin{align}
H(u, v) = \frac{1}{2}\int_{\T^2}\,\big( u^2 +  |\nb u|^2\big)\, dx
+ 
\frac{1}{2}\int_{\T^2} v^2dx
+ \frac1{k+1} \int_{\T^2} u^{k+1} dx.
\label{Hamil}
\end{align}

\noi
By drawing an analogy to the finite dimensional setting,
the Hamiltonian structure of the equation and the conservation of the Hamiltonian
suggest that the Gibbs measure $\Pk$ of the form:
\begin{align}
 \text{``}d\Pk = Z^{-1} \exp(- H(u,v ))du\otimes dv \text{''}
\label{G1}
 \end{align}

\noi
is  invariant under the dynamics of~\eqref{NLW0}.
By substituting  \eqref{Hamil} for $H(u, v)$ in the exponent, we can rewrite the formal expression~\eqref{G1}
as 
\begin{align}
d\Pk
& = Z^{-1} 
e^{-\frac 1{k+1} \int_{\T^2} u^{k+1} dx} 
e^{-\frac 12 \int_{\T^2} ( u^2 + |\nb u |^2) dx} du\otimes e^{-\frac 12 \int_{\T^2} v^2 dx} dv \notag\\
& \sim e^{-\frac 1{k+1} \int_{\T^2} u^{k+1} dx} d\mu , 
\label{G2}
 \end{align}

\noi
where  $\mu$ is   the  massive  Gaussian free field defined in \eqref{G3}.

Recall that the Gaussian measure $\mu$ in \eqref{G3} is the induced probability measure
under the map: 
\[\o \in \O \longmapsto (u_0^\o, u_1^\o), \]

\noi
where $(u_0^\o, u_1^\o)$ is given by
the following random Fourier series:\footnote{We drop the harmless factor $2\pi$ in the following.} 
\begin{align}
(u_0^\o,  u_1^\o)
=\bigg( \sum_{n \in \Z^2} \frac{g_{0, n}(\o)}{\jb{n}}e^{in\cdot x},
\sum_{n \in \Z^2} g_{1, n}(\o)e^{in\cdot x}\bigg).
\label{Gauss1}
\end{align}

\noi
Here, 
$\jb{n} = \sqrt{1+|n|^2}$ and 
$\{g_{0, n}, g_{1, n}\}_{n \in \Z^2}$
is a sequence of independent standard complex-valued Gaussian
random variables on a probability space $(\O, \F, P)$
conditioned that ${g_{j, -n} = \cj{g_{j, n}}}$, $n \in \Z^2$, $j = 0, 1$.
It is easy to check that 
$(u_0^\o,  u_1^\o)$ belongs to $\H^s(\T^2)\setminus \H^0(\T^2)$, $s < 0$,
almost surely. 
In particular, for an odd integer $k\geq 3$, we have $\int_{\T^2} u^{k+1} dx = \infty$ almost surely with respect to~$\mu$
and thus 
 the density 
$e^{-\frac 1{k+1} \int_{\T^2} u^{k+1} dx}$ in \eqref{G2} vanishes almost surely.
As a result, 
 the expression in~\eqref{G2} does not make sense as a probability measure.
This forces us to renormalize the potential part of the Hamiltonian, 
which 
enables us to define the Gibbs measure $\Pk$
corresponding to the renormalized Hamiltonian
as a probability measure (absolutely continuous with respect to the Gaussian free field $\mu$).
See \cite{Simon, GJ, DPT1, OTh1} for details.
As a consequence, one is led to  study  the renormalized NLW dynamics (see \eqref{WNLW0} below)
associated with the renormalized Hamiltonian. 
\subsection{Wick ordered NLW}
In this subsection, we go over a derivation of the renormalized NLW by directly introducing a renormalization at the level of the equation. By writing~\eqref{NLW0}  in the Duhamel formulation with the random initial data $(u_0^\o, u_1^\o)$ in \eqref{Gauss1}, we have
\begin{align}
u(t) = S(t)(u_0^\o, u_1^\o) - \int_0^t \frac{\sin ((t - t') \jb{\nb})}{\jb{\nb}}u^k(t') dt', 
\label{NLW1}
\end{align}

\noi
where $\jb{\nb}= \sqrt{1 -\Dl}$
and $S(t)$ denotes the linear wave propagator
given by 
\begin{align}
S(t) (f, g) = \cos (t \jb{\nb}) f+ \frac{\sin (t \jb{\nb})}{\jb{\nb}} g.
\notag
\end{align}

\noi
Let $ z$ denote the random linear solution given by 
\begin{align}
z = z^\o= S(t) (u_0^\o,  u_1^\o).
\label{lin1a}
\end{align}

\noi
Recalling that 
$(u_0^\o,  u_1^\o) \in \H^s(\T^2)\setminus \H^0(\T^2)$, $s < 0$,
almost surely, 
we see that $z(t)$ is merely a Schwartz distribution.
Hence, there is an issue 
 in making sense of the power $z^k(t)$
and thus the full nonlinearity $u^k(t)$ appearing in \eqref{NLW1}.
In fact, 
by following the argument in~\cite{OPTz, OOR}, 
 a phenomenon of triviality
 may be shown for \eqref{NLW0} without renormalization (at least when $k = 3$).
Namely, by considering smooth solutions $u_N$ to 
 \eqref{NLW0} with regularized random initial data, 
we may show that, 
as the regularization is removed, 
$u_N$ converges to  a trivial solution $u \equiv 0$.
This shows the necessity of  a proper renormalization 
at the level of the equation.

%
%
%
%

With \eqref{lin1a},  we easily see that, for any $t \in \R$,  the distribution of $z(t)$ is once again given by  
the massive Gaussian free field $\mu$ in \eqref{G3}.
Namely, $\mu$ is invariant under the linear wave dynamics. Indeed, we have 
\begin{align}
(z(t), \dt z(t) )   =\bigg( \sum_{n \in \Z^2} \frac{g_{0, n}^t}{\jb{n}}e^{in\cdot x},\sum_{n \in \Z^2}  g_{1, n}^t e^{in\cdot x}\bigg),
\label{lin2}
\end{align}

\noi
where
\begin{align}
\begin{split}
g_{0, n}^t & \stackrel{\text{def}}{=}  \cos (t \jb{n}) g_{0, n}
+ \sin (t \jb{n}) g_{1, n},\\
g_{1, n}^t & \stackrel{\text{def}}{=}  - \sin (t \jb{n}) g_{0, n}
+ \cos (t \jb{n}) g_{1, n}.
\end{split}
\label{lin3}
\end{align}

\noi
It is easy to check that 
$\{g_{0, n}^t, g_{1, n}^t\}_{n \in \Z^2}$
forms a  sequence of independent standard complex-valued Gaussian
random variables conditioned that 
\begin{align}
g_{j, -n}^t = \cj{g_{j, n}^t}
\label{B1}
\end{align}
\noi
for any $n \in \Z^2$ and $j = 0, 1$.
This shows that the massive Gaussian free field $\mu$ in \eqref{G3} is invariant under the linear wave dynamics.

Let $\P_N$ denote the frequency projection
onto the spatial frequencies $\{|n|\leq N\}$
and set $z_N = \P_N z$.
Then,   for each $(x, t) \in \T^2\times \R$,  
 $ z_N(x, t)$ 
 is a mean-zero real-valued Gaussian random variable with variance\footnote{While it may be
common  to denote the variance by $\s_N^2$, we chose to use $\s_N$
 to denote the variance in~\eqref{G6}
 so that it is consistent with the notation $H(x; \s)$ for the Hermite polynomial
 with a parameter $\s$,
 which is used for the Wick renormalization \eqref{Wick1};
 see \eqref{H1} and \eqref{H1a}.  See also Kuo's book \cite[Chapter 9]{Kuo}.}
 \begin{align}
\s_N \stackrel{\text{def}}{=} 
\text{Var}(z_N(x, t))
= \E [  z_N^2(x, t)]
= \sum_{|n|\leq N}\frac{1}{\jb{n}^2} \sim \log N.
\label{G6}
\end{align}

\noi
Note that $\s_N$ is independent of $(x, t) \in \T^2\times \R$,
reflecting the translation-invariant nature of the problem.
We then define  the Wick powers
$:\!  z_N^\l  \!:$, $\l \in \N \cup\{0\}$,
by setting
\begin{align}
:\!  z_N^\l (x, t)\!:  \, \stackrel{\text{def}}{=} H_\l( z_N(x, t); \s_N)
\label{Wick1}
\end{align}

\noi
in a pointwise manner, 
where  $H_\l(x; \s)$ denotes the Hermite polynomial of degree $\l$
with a parameter $\s> 0$.
See Section \ref{SEC:2} for more on the Hermite polynomials.
We now recall the following proposition from \cite{OTh2, GKO}.

\begin{proposition}\label{PROP:Z1}
Let $\l \in \N\cup \{0\}$.
Then, 
for any $p < \infty$, $T > 0$,  and $\eps>  0$, 
the sequence
 $\{ : \!z_N^\l \!: \}_{N \in \N}$ is  Cauchy 
 in $L^p(\O; C([-T, T]; W^{-\eps, \infty}(\T^2)) )$.
Denoting the limit by 
\begin{align}
:\! z^\l\!:\
= \ :\! z_\infty^\l\!:\, \stackrel{\textup{def}}{=} \lim_{N \to \infty}
:\! z_N^\l\!:, 
\label{Wick4}
\end{align}

\noi
we have
$:\! z^\l\!:\, \in C([-T, T]; W^{-\eps, \infty}(\T^2)) $, almost surely.

\end{proposition}

In \cite{OTh2}, the convergence was shown only 
in  $L^p(\O; L^q([-T, T]; W^{-\eps, r}(\T^2)) )$
for $q, r < \infty$.
By repeating the argument in \cite[Proposition 2.1]{GKO}, 
however, we can easily upgrade this to the claimed regularity result
in Proposition \ref{PROP:Z1}.
One may also apply 
Proposition \ref{PROP:reg}
below and directly verify Proposition \ref{PROP:Z1}.
See Subsection \ref{SUBSEC:sto1}. See also \cite{GKO2, GKOT, OOcomp}.

\medskip

Given $N \in \N$,  consider the  following truncated  NLW: 
\begin{align}
\dt^2 u_N + (1 -  \Dl)  u_N   +  \P_N\big[ (\P_N u_N)^{k} \big] = 0 
\notag
\end{align}

\noi
with the random initial data $ ( u_0^\o,  u_1^\o)$ in \eqref{Gauss1}.
In view of the Duhamel formula, it is natural to decompose $ u_N$
as 
\begin{align}
 u_N = z + v_N
\label{u1}
\end{align}

\noi
with $v_N = \P_N v_N$.
Then, by the binomial theorem, 
we have
\begin{align}
 (\P_N u_N)^{k}
 = (z_N + v_N)^k 
=  \sum_{\l = 0}^{k}
\begin{pmatrix}
k \\ \l
\end{pmatrix}
 z_N^\l \cdot  \,  v_N^{k - \l}
\label{Herm0}
\end{align}

\noi
and thus we see that there is an issue in taking 
a limit as $N \to \infty$,  since the limit of $z_N^\l$ does not exist.
By  recalling 
the following  identities for the Hermite polynomials:
\begin{align}
H_k(x+y) & = \sum_{\l = 0}^k
\begin{pmatrix}
k \\ \l
\end{pmatrix}
H_\l(y)\cdot x^{k - \l} 
\qquad \text{and}\qquad 
H_k(x; \s )  = \s^\frac{k}{2} H_k(\s^{-\frac{1}{2}} x), 
\notag
\end{align}

\noi
we define the renormalized nonlinearity $:\! (\P_N u_N)^{k}\!\!:   $ by setting
\begin{align}
\begin{split}
:\! (\P_N u_N)^{k}\!\!:  \, 
& = \NN^k_{(\P_N u_0^\o, \P_N u_1^\o)}(u_N) \\
& \stackrel{\text{def}}{=} H_{k}( z_N +  v_N; \s_N) 
= \sum_{\l = 0}^{k}
\begin{pmatrix}
k \\ \l
\end{pmatrix}
 H_\l(z_N; \s_N)\cdot   v_N^{k - \l}\\
& = \sum_{\l = 0}^{k}
\begin{pmatrix}
k \\ \l
\end{pmatrix}
:\! z_N^\l\!:\cdot  \,  v_N^{k - \l}.
\end{split}
\label{Herm2}
\end{align}

\noi
Namely, we replaced 
$z_N^\l$ in \eqref{Herm0} by 
the Wick power $:\! z_N^\l\!:$.
In view of Proposition~\ref{PROP:Z1}, 
we can take a limit of \eqref{Herm2} as $N \to \infty$.
This leads to the following Wick ordered NLW:
\begin{align}
\begin{cases}
\dt^2 u + (1 -  \Dl)  u\,   +  :\! u^{k} \!:\, = 0 \\
(u, \dt u)|_{t = 0} = (u_0^\o, u_1^\o),
\end{cases}
\label{WNLW0}
\end{align}

\noi
where  the Wick ordered nonlinearity 
$:\! u^{k} \!:$ is defined by 
\begin{align}
:\!  u^{k}\!\!:  \, 
= \NN^k_{(u_0^\o, u_1^\o)}(u)
& \stackrel{\text{def}}{=} \sum_{\l = 0}^{k}
\begin{pmatrix}
k \\ \l
\end{pmatrix}
:\! z^\l\!:\cdot  \,  v^{k - \l}
\label{Herm4}
\end{align}

\noi
for functions $u$ of the form:
\begin{align}
 u = z + v
\label{u2}
\end{align}

\noi
with some sufficiently smooth $v$ such that $v^{k-\l}$ in
\eqref{Herm4} makes sense.
We stress that the Wick ordered nonlinearity 
$:\! u^{k} \!:$ is not defined for  general functions $u$
but is defined only for  functions $u$ of the form \eqref{u2}.

\medskip

In \cite{OTh2}, 
the first author and Thomann studied 
the Wick ordered NLW \eqref{WNLW0}
by considering the following  fixed point
problem for the residual term $v = u - z$:
\begin{align}
\begin{cases}
\dt^2 v +(1 -  \Dl)  v  \, +  :\! (v+z)^{k} \!:\, = 0 \\
(v, \dt v) |_{t = 0} = (0, 0).
\end{cases}
\label{WNLW1}
\end{align}

\noi
A result of interest to us reads as follows:

\begin{oldtheorem}[\cite{OTh2}]\label{THM:OTh}
The Wick ordered NLW \eqref{WNLW0} is almost surely globally well-posed with respect to the 
massive Gaussian free field $\mu$ in \eqref{G3}.
Moreover, the solution $(u, \dt u)$ to \eqref{WNLW0} almost surely lies in the class:
\begin{align}
( u , \dt u)  \in ( z, \dt z)  + C(\R; \H^{1-\eps}(\T^2))
  \subset C(\R; \H^{-\eps}(\T^2) ).
\label{class1}
 \end{align}

\noi
for any $\eps > 0$.

\end{oldtheorem}
\begin{remark}\rm
Consider the following truncated  Wick ordered NLW: 
\begin{align}
\begin{cases}
\dt^2 u_N + (1 -  \Dl)  u_N   +  \P_N\big[:\! (\P_N u_N)^{k} \!:\big] = 0 \\
(u_N, \dt u_N) |_{t = 0} = ( u_0^\o,  u_1^\o),
\end{cases}
\label{WNLW2}
\end{align}

\noi
where the truncated Wick ordered nonlinearity is interpreted as in \eqref{Herm2}
for $u_N$ of the form~\eqref{u1}.
Then, it  follows from iterating the local theory in \cite{OTh2}
that, for given $T > 0$,  the solution $u_N$ to~\eqref{WNLW2}
converges almost surely to the solution $u$ to~\eqref{WNLW0}
in $C([-T, T]; H^{-\eps}(\T^2))$, $\eps>  0$
(and the residual part $v_N = u_N - z$ converges
to $v = u - z$  
in $C([-T, T]; H^{1-\eps}(\T^2))$,  almost surely).

\end{remark}

The proof of almost sure local well-posedness of \eqref{WNLW0} follows
from studying the fixed point problem \eqref{WNLW1} for $v$
with  Sobolev's inequality\footnote{While the argument in \cite{OTh2}
used the Strichartz estimates, 
it is possible to prove the local well-posedness part in Theorem \ref{THM:OTh}
by Sobolev's inequality.
See \cite{GKOT}.
} and the space-time control on the stochastic terms (Proposition \ref{PROP:Z1}).
The almost sure global well-posedness follows
from (i)~almost sure global well-posedness of the Wick ordered NLW \eqref{WNLW0}
with respect to the Gibbs measure $\Pk$  (by Bourgain's invariant measure argument
\cite{BO94, BO96, BT2})
and (ii)~the mutual absolute continuity of the Gibbs measure $\Pk$
and the massive Gaussian free field~$\mu$.
Lastly, the second claim \eqref{class1}
follows 
from iterating the local-in-time argument
with Proposition~\ref{PROP:Z1}. 

\medskip

Let $u$ be 
 the random solution  to the Wick ordered NLW \eqref{WNLW0} 
with $\mathcal{L}\big((u, \dt u)|_{t = 0}\big) = \mu$ constructed in Theorem \ref{THM:OTh}.
In the following, 
we study the approximation property
of the random solution $u$ 
to the renormalized NLW \eqref{WNLW0} 
by the smooth solutions corresponding to smooth approximating (random) initial data.
We point out that the renormalized nonlinearity 
$:\!  u^{k}\!\!: $ in \eqref{WNLW0} 
is defined for the specific  random initial data
$(u_0^\o, u_1^\o)$ in \eqref{Gauss1}.
In particular, 
in considering the renormalized dynamics corresponding to smooth random 
initial data, we need to make it clear what we mean by the renormalized nonlinearity
for smooth random initial data.
This is the topic of the next subsection.
\subsection{Renormalized NLW with smooth Gaussian initial data}
In this subsection, we
consider the renormalized NLW with {\it smooth} Gaussian random initial data.
While  there is no need to consider any renormalization 
in studying \eqref{NLW0} with smooth random initial data, 
we introduce a renormalization even for smooth random initial data
so that we can study smooth approximations to the Wick ordered NLW \eqref{WNLW0}
with $\mathcal{L} (u_0^\o, u_1^\o) = \mu$. 
For this purpose,  we introduce the following definition.

\begin{definition}\rm 
\label{DEF}
Let  $(\varphi_0^\o, \varphi_1^\o)$ be an $\H^s(\T^2)$-valued random variable
for some $s \geq 0$.
Set
$$
\sigma(t)\stackrel{\text{def}}{=} \text{Var}\big(S(t)(\varphi_0^\o, \varphi_1^\o)\big)=
\E\big[(S(t)(\varphi_0^\o, \varphi_1^\o))^2\big]-\big(\E[S(t)(\varphi_0^\o, \varphi_1^\o)]\big)^2.
$$
Then,  we define the renormalized nonlinearity $\NN^k_{(\varphi_0^\o, \varphi_1^\o)} (v)$ by
$$
\NN^k_{(\varphi_0^\o, \varphi_1^\o)} (v)
\stackrel{\text{def}}{=}H_{k}\big(S(t)(\varphi_0^\o, \varphi_1^\o)+v ;\sigma(t)\big).
$$
\end{definition}

In view of the previous discussion, we aim to  study the following problem: 
\begin{align}
\begin{cases}
\dt^2 v +(1 -  \Dl)  v  + \NN^k_{(\varphi_0^\o, \varphi_1^\o)} (v) = 0\\
(v, \dt v) |_{t = 0} = (0, 0)
\end{cases}
\label{WNLW_PAK}
\end{align}

\noi
for a sequence of (smoother) random initial data 
$(\varphi_0^\o, \varphi_1^\o) \in \H^0(\T^2)$ approximating 
$(u_0^\o, u_1^\o)$ given in~\eqref{Gauss1}.
Our goal is then 
to  try to understand how much the obtained sequence 
 of (smoother) solutions converges to the solution obtained in  Theorem~\ref{THM:OTh} (modulo the free evolution), 
 i.e.~the solution $v = u - z = u - S(t) (u_0^\o,  u_1^\o)$ to \eqref{WNLW1}.  For this purpose, we will first solve~\eqref{WNLW_PAK} for a large class of $(\varphi_0^\o, \varphi_1^\o)$ in  ${\mathcal H}^s(\T^2)$, $s\geq 0$.

\medskip

Let us now describe the class of data  $(\varphi_0^\o, \varphi_1^\o)$ 
for which we study  \eqref{WNLW_PAK}.
Let $(\phi_0, \phi_1) \in \H^s(\T^2)$, $s \geq 0$,  with the Fourier series expansions
\[
\phi_j = \sum_{n \in \Z^2} \ft \phi_{j}(n) e^{in \cdot x}
\quad \text{with} \quad  \ft \phi_{j}(-n) = \cj{\ft \phi_{j}(n)},
  \quad j=0,1.
\]

\noi
We define the randomization  $(\phi_0^\o, \phi_1^\o)$  of $(\phi_0, \phi_1)$ by
setting
\begin{align}
\phi_j^\o   \stackrel{\text{def}}{=}  \sum_{n \in \Z^2}  g_{j,n} (\o) \ft \phi_j(n) e^{in \cdot x},
\label{u4}
\end{align}

\noi
where $\{g_{0, n}, g_{1, n}\}_{n \in \Z^2}$ is as in \eqref{Gauss1}.
Let $(r_0,r_1)\in {\mathcal H}^{s+1}(\T^2)$. We then study  \eqref{WNLW_PAK}  with $(\varphi_0^\o, \varphi_1^\o)$  given by
\begin{equation}\label{genton}
(\varphi_0^\o, \varphi_1^\o)=(\phi_0^\o,\phi_1^\o)+(r_0,r_1).
\end{equation}

\noi
Note that 
\begin{align}
\label{sigma_pak}
\begin{split}
\sigma(t)& =\text{Var}\big(S(t)(\varphi_0^\o, \varphi_1^\o)\big)=\text{Var}\big(S(t)(\phi_0^\o, \phi_1^\o)\big)
\\
& = \sum_{n \in \Z^2} \bigg( \cos^2 (t \jb{n}) |\ft \phi_{0}(n)|^2 
+ \frac{\sin^2 (t \jb{n})}{\jb{n}^2} |\ft \phi_{1}(n)|^2 \bigg)
\les \| (\phi_0,\phi_1) \|_{\H^0}^2 < \infty,
\end{split}
\end{align}

\noi
which shows that the renormalized nonlinearity 
$\NN^k_{(\varphi_0^\o, \varphi_1^\o)} (v)$
in \eqref{WNLW_PAK} is well defined
for the random data $(\varphi_0^\o, \varphi_1^\o)$
given in \eqref{genton}.
Compare this with the renormalized nonlinearity in~\eqref{WNLW1}
which is defined only via a limiting procedure via Proposition \ref{PROP:Z1}.
We have the following proposition on almost sure global existence of unique solutions
to  the Wick ordered NLW \eqref{WNLW_PAK} with 
the random data $(\varphi_0^\o, \varphi_1^\o)$
given by   \eqref{genton}.

\begin{proposition} \label{PROP:easy}
Let $k \geq 3$ be an odd integer
and let $s \in \R$ satisfy
\[ \text{\textup{(i)} $s > 0$ when $k = 3$ 
\qquad and \qquad \textup{(ii)} $s \geq 1$ when $k \geq 5$.}\]

\noi 
Given $(\phi_0, \phi_1) \in \H^s(\T^2)$ and $(r_0,r_1)\in \H^{s+1}(\T^2)$, 
let $(\phi_0^\o, \phi_1^\o)$ be the randomization of $(\phi_0, \phi_1)$ defined in \eqref{u4}
and 
define $(\varphi_0^\o, \varphi_1^\o)$  as in  \eqref{genton}. 
Then,  there exists almost surely  a unique global solution 
$
 (v , \dt v) \in C(\R; \H^{1}(\T^2))
$
to \eqref{WNLW_PAK}.
\end{proposition}

We present the proof of Proposition \ref{PROP:easy} in Section \ref{SEC:GWP}.
The almost sure local well-posedness for $s \geq 0$ (with any $k$) follows from a standard fixed point argument with the probabilistic Strichartz estimate (Lemma~\ref{LEM:PStr}). 
See, for example, \cite{BT3, Poc}.
As for the almost sure global well-posedness,  we proceed with a Gronwall argument as in  \cite{BT3} when $k = 3$. 
For $k \ge 5$, we also use the integration-by-parts  trick, introduced in  \cite{OP},  to control higher order terms with respect to $v$
\begin{remark}\rm
(i)
Observe that if the data in \eqref{genton} is deterministic, i.e. $\phi_j^\omega=0$, then $\sigma(t)=0$ and the nonlinearity is of pure power type, namely 
$\NN^k_{(\varphi_0^\o, \varphi_1^\o)} (v)$ becomes $(S(t)(r_0,r_1) +v)^k$. 

\smallskip

\noi
(ii) For simplicity of the presentation, 
we chose $(r_0,r_1)\in \H^{s+1}(\T^2)$ in the statement of Proposition \ref{PROP:easy}
such that 
the Cameron-Martin theorem \cite{CM}
allows us to reduce the proof to the case  $r_0=r_1=0$
at the beginning of Section \ref{SEC:GWP}.
In fact, a slight modification of the argument in 
Section \ref{SEC:GWP} shows that 
Proposition \ref{PROP:easy}
also holds for 
$(r_0,r_1)\in \H^{1}(\T^2)$, 
whether
$(r_0,r_1)$ is deterministic or random.
Indeed, given
$(r_0,r_1)\in \H^{1}(\T^2)$, 
by setting $w = v + S(t) (r_0, r_1)$, 
where $v$ is a solution to  \eqref{WNLW_PAK}, 
we see that $w$ satisfies the following Cauchy problem:
\begin{align} \label{WNLWx}
\begin{cases}
\dt^2 w + (1- \Dl) w + H_k ( S(t)(\phi_0^\o, \phi_1^\o) +w(t); \sigma(t)) = 0\\
(w, \dt w)|_{t = 0} = (r_0, r_1) \in \H^{1}(\T^2).
\end{cases}
\end{align}

\noi
Then, by noting that 
Lemma \ref{LEM:LWPreg} on local well-posedness
holds for general $\H^1$-initial data, 
global well-posedness of \eqref{WNLWx} 
follows from proceeding as in 
the proof of Proposition \ref{PROP:easy}
presented in Section \ref{SEC:GWP}, 
which is about controlling the $\H^1$-norm of a solution;
see \eqref{enebd}.
Once we have constructed a unique global-in-time solution $w$
to \eqref{WNLWx}, 
we simply set $v = w- S(t) (r_0, r_1)$, 
which is 
a unique global-in-time solution to \eqref{WNLW_PAK}.
\end{remark}

\subsection{Unique and non-unique extensions of the solution map to the Wick ordered NLW with the Gaussian free field $\mu$ as initial data}
In this subsection, we state our main results in this paper. In the following,  we restrict our attention to the cubic case ($k = 3$).
In the previous subsection, we constructed almost surely well-defined
global-in-time dynamics for the Wick ordered NLW \eqref{WNLW_PAK}
with smooth random initial data (Proposition~\ref{PROP:easy}).
In particular, there exists a solution map, sending
smooth random initial data to smooth random solutions. 
On the other hand, 
Theorem~\ref{THM:OTh} shows 
that  the solution map ``extends'' to 
the (rough) Gaussian random initial data $(u_0^\o, u_1^\o)$ of the form~\eqref{Gauss1}, 
distributed according to the massive Gaussian free field $\mu$ in~\eqref{G3}.
In the following, we investigate in what sense the solution map
constructed in Proposition~\ref{PROP:easy}
extends to that in Theorem~\ref{THM:OTh}.

We first establish a (partial) positive answer. Namely, we show that the solution map constructed in Theorem~\ref{THM:OTh} is the unique extension of the solution map defined on a certain class of smooth random initial data. We say that a smooth function $\rho \in L^1(\R^2)$  is a mollification kernel
if  $\int_{\R^2} \rho(x) dx = 1$
and $\supp \rho \subset (-\frac 12, \frac 12]^2$. 
Given a mollification kernel~$\rho$, 
define $\rho_\dl$ by setting
\[\rho_\dl(x) = \dl^{-2} \rho(\dl^{-1} x)\]

\noi
for $0 < \dl \leq 1$.
Then, $\{\rho_\dl\}_{\dl \in (0, 1]}$ is an approximate identity on $\R^2$.
By noting that 
$\supp \rho_\dl  
\subset (-\frac 12, \frac 12]^2\cong \T^2$  for any $\dl \in (0, 1]$, 
we see that 
$\{\rho_\dl\}_{\dl \in (0, 1]}$ is also  an approximate identity on $\T^2$.

The following theorem shows that 
the solution map constructed in Theorem \ref{THM:OTh}
is the unique extension of the solution map
defined on smooth random initial data, regularized by a mollification.
Here, the uniqueness refers to the  convergence of the whole sequence 
and also to the fact that the extension is independent of  mollification kernels
$\rho$.

\begin{theorem}\label{THM:uniq}
Let  $(u_0^\o, u_1^\o) $ be the Gaussian random initial data defined in \eqref{Gauss1}.
Given 
a mollification kernel  $\rho$, 
define 
 $(u_{0,\dl}^\o, u_{1,\dl}^\o) \in C^{\infty} (\T^2) \times C^{\infty} (\T^2)$, 
 $0< \dl \leq 1$,  
 via the  regularization by mollification:
\begin{align}
u_{0,\dl}^\o =  \rho_\dl * u_0^\o \qquad \text{and}\qquad 
u_{1,\dl}^\o = \rho_\dl * u_1^\o ,
\label{uniq1}
\end{align}

\noi
where $\rho_\dl$ is as above 
\textup{(}of course, 
$
 \lim_{\delta\rightarrow 0}\|(u_{0,\dl}^\o, u_{1,\dl}^\o)-(u_0^\omega,u_1^\omega)\|_{\H^s}=0
 $, almost surely\textup{)}.
Denote by $(v_\dl, \dt v_\dl) $ the solution to the Wick ordered NLW \eqref{WNLW_PAK} with 
$$
(\varphi_0^\o, \varphi_1^\o)=(u_{0,\dl}^\o, u_{1,\dl}^\o)\,,
$$
constructed in Proposition~\ref{PROP:easy}, 
and set $u_{\dl}\stackrel{\textup{def}}{=}S(t)(u_{0,\dl}^\o, u_{1,\dl}^\o)+v_{\dl}$. 
Then, given any $T> 0$ and $s< 0$, 
$(u_\dl,  \dt u_\dl)$
converges in probability to $(u, \dt u)$ in $C([-T,T]; \H^{s}(\T^2))$, 
where  $(u , \dt u)$ is the solution to the Wick ordered NLW \eqref{WNLW0} with 
the initial data $(u_0^\o, u_1^\o)$ constructed in Theorem~\ref{THM:OTh}.
Namely, $u = z+v$, 
where $z$ and $v$ are as in \eqref{lin1a} and \eqref{WNLW1}, respectively.

\end{theorem}

Next, we turn our attention to a negative direction. We prove the following  instability result for the Wick ordered NLW \eqref{WNLW0} with the Gaussian free field $\mu$ in \eqref{G3} as initial data.
\begin{theorem}\label{THM:ill0}
Let  $s < 0$ and 
$(u_0^\o, u_1^\o) $ be as in \eqref{Gauss1}.
Then, 
there exists a set $\Si \subset \O$ with $P(\Si) = 1$ such that 
given $\o \in \Si$, 
there exists  a sequence  $(u_{0,\eps}^\o, u_{1,\eps}^\o) \in C^{\infty} (\T^2) \times C^{\infty} (\T^2)$, 
 $0< \eps \leq 1$,  such that almost surely
 $$
 \lim_{\eps\rightarrow 0}\|(u_{0,\eps}^\o, u_{1,\eps}^\o)-(u_0^\omega,u_1^\omega)\|_{\H^s}=0
 $$
but for every $T>0$,  the solutions $v_{\eps}$ to  \eqref{WNLW_PAK} with 
$$
(\varphi_0^\o, \varphi_1^\o)=(u_{0,\eps}^\o, u_{1,\eps}^\o)
$$
defined in Proposition~\ref{PROP:easy} satisfy almost surely
$$
 \lim_{\eps\rightarrow 0}\|v_{\eps}\|_{L^\infty([-T,T]; H^{s})}=\infty.
 $$
As a consequence,  $u_{\eps}\stackrel{\textup{def}}{=}S(t)(u_{0,\eps}^\o, u_{1,\eps}^\o)+v_{\eps}$ diverges 
almost surely in $C([-T,T]; H^{s}(\T^2))$.
\end{theorem}

Theorems \ref{THM:uniq} and \ref{THM:ill0} together imply that the choice of regularization of the random initial data plays an important role. 
On the one hand, there is a class of ``admissible'' regularizations
yielding the  conclusion of Theorem \ref{THM:uniq}.
On the other hand, there is also a regularization, leading
to a strong instability.
This is a sharp contrast with the smoother regime, 
where (deterministic) local well-posedness theory, in particular continuous dependence, guarantees
{\it any} regularization gives a good approximation.
See Theorems~1.33 and~2.7 in~\cite{Tz1} 
for analogous results in the context of the three-dimensional cubic NLW
(without the need of renormalization).
One main difference between our results (Theorems~\ref{THM:uniq} and \ref{THM:ill0})
and those in \cite{Tz1}
appears in the fact that, in our problem,  the effect of the random initial data
shows up in the equation through the renormalized nonlinearity,
giving further complication to the problem.

\begin{remark}\rm
In a recent work
\cite{STz}, the third author and Sun
established  a certain pathological behavior for NLW 
on the three-dimensional torus $\T^3$
with initial data of super-critical (but positive\footnote{In particular, there is no need for renormalization
in \cite{STz}.}) regularity.
They constructed 
 a dense subset  $S$ of the Sobolev space of super-critical regularity
 such that 
 for any $(u_0, u_1) \in S$, 
 the family of global smooth solutions $u_\dl$, generated by the mollified initial data $(\rho_\dl*u_0, \rho_\dl*u_1)$,   diverges. 
While it is a purely deterministic result, 
this result nicely complements  Theorem \ref{THM:uniq}, 
since it shows that  a mollification does not in general (and in fact on a dense set) lead to a good approximation
in the super-critical regularity.

\end{remark}

Before proceeding to the next subsection,  we briefly discuss a reduction of the proof of Theorem~\ref{THM:ill0} in the following.
Our main strategy is as follows.
Given $\eps > 0$, let   $(u_{0,\dl}^\o, u_{1,\dl}^\o)$
be the mollified random initial data as in \eqref{uniq1} for some small $\dl = \dl (\eps) > 0$.
We then construct the  smooth  solution $v_{\dl, \eps} = v_{\dl, \eps}^\o$ 
to \eqref{WNLW_PAK}:
\begin{align}
\begin{cases}
\dt^2 v_{\dl, \eps} +(1 -  \Dl)  v_{\dl, \eps}  + \NN^3_{(\varphi_{0, \dl,\eps}^\o, \varphi_{1, \dl,\eps}^\o)} (v_{\dl, \eps}) = 0\\
(v_{\dl, \eps}, \dt v_{\dl, \eps}) |_{t = 0} =(0,0),
\end{cases}
\label{WNLW4a}
\end{align}
where 
\begin{align*}
(\varphi_{0, \dl,\eps}^\o, \varphi_{1, \dl,\eps}^\o)=(u_{0, \dl}^\o, u_{1, \dl}^\o)+ (\phi_{0,\eps}, \phi_{1, \eps})
\end{align*}

\noi
for some suitably chosen {\it deterministic} functions $(\phi_{0, \eps}, \phi_{1, \eps}) \in C^\infty(\T^2) \times C^\infty(\T^2)$.
The first observation is that the conclusion of Theorem \ref{THM:uniq} holds true even if we replace 
 $(u_0^\o, u_1^\o) $ (and 
 $(u_{0,\dl}^\o, u_{1,\dl}^\o)$, respectively)
by 
 $(u_0^\o, u_1^\o) + (\phi_{0,\eps}, \phi_{1, \eps})$ 
 (and 
 $(u_{0,\dl}^\o, u_{1,\dl}^\o) + (\phi_{0, \eps}, \phi_{1, \eps})$, 
 respectively) 
 for any 
  $(\phi_{0, \eps}, \phi_{1, \eps}) \in C^\infty(\T^2) \times C^\infty(\T^2)$.
Namely, the smooth solution 
$v_{\dl, \eps}$ to \eqref{WNLW4a}
converges in probability to the solution $v_\eps$ to 
\begin{align}
\begin{cases}
\dt^2 v_{\eps} +(1 -  \Dl)  v_{ \eps}  + \NN^3_{(\varphi_{0,\eps}^\o, \varphi_{1,\eps}^\o)} (v_{ \eps}) = 0\\
(v_{\eps}, \dt v_{ \eps}) |_{t = 0} = (0,0)
\end{cases}
\label{WNLW4b}
\end{align}

\noi
as $\dl \to 0$, 
where 
\begin{align}
(\varphi_{0,\eps}^\o, \varphi_{1,\eps}^\o)=(u_0^\o, u_1^\o)+ (\phi_{0,\eps}, \phi_{1, \eps}).
\label{XX2}
\end{align}

\noi
See Remark \ref{REM:uniq2}.
Note that,  in \eqref{WNLW4b},   the nonlinearity $ \NN^3_{(\varphi_{0,\eps}^\o, \varphi_{1,\eps}^\o)} (v_{ \eps})$ is interpreted 
in the limiting sense  as $\delta\rightarrow 0$. 
This observation allows us to drop the smoothness assumption on data in Theorem~\ref{THM:ill0}.
More precisely, Theorem~\ref{THM:ill0} is a consequence of the following statement.
\begin{proposition}\label{PROP:ill2}
Let  $s < 0$ and  $(u_0^\o, u_1^\o) $ be as in \eqref{Gauss1}. Then,  there exists a set $\Si \subset \O$ with $P(\Si) = 1$ such that  given $\o \in \Si$ and $\eps > 0$,  there exist  a solution $v_\eps^\o$ to \eqref{WNLW4b} on $\T^2$ 
with the random data $(\varphi_{0,\eps}^\o, \varphi_{1,\eps}^\o)$ in \eqref{XX2}
and a random time $t_\eps = t_\eps(\o)  \in (0, \eps) $ such that 
\begin{align*}
 \big\|  (\phi_{0,\eps}, \phi_{1, \eps})\big\|_{\H^s} < \eps \qquad \text{but} 
\qquad \| v_\eps^\o(t_\eps)\|_{H^s} > \eps^{-1}.
\end{align*}
\end{proposition}

In our reduction of  Theorem~\ref{THM:ill0} to Proposition~\ref{PROP:ill2}, 
we 
moved from the smooth setting to the rough setting, contrary to the usual
reduction, where one approximates rough objects by smooth objects.
This reduction, however, helps us since 
the solutions $v$ to \eqref{WNLW1} 
and $v_\eps$ to \eqref{WNLW4b} satisfy the {\it same} equation,
where the renormalization on the nonlinearity is based on $(u_0^\o, u_1^\o)$
defined in \eqref{Gauss1}.

\medskip

We now express \eqref{WNLW4b} in terms of  
$w_\eps = v_\eps + S(t) (\phi_{0, \eps}, \phi_{1, \eps})$.
Then, $w_\eps$ satisfies the following perturbed NLW:
\begin{align}
\begin{cases}
\dt^2 w_\eps  + (1 -  \Dl)  w_\eps  +  w_\eps^3 + \RR(w_\eps, z) = 0\\ 
(w_\eps, \dt w_\eps) |_{t = 0} = (\phi_{0, \eps}, \phi_{1, \eps}), 
\end{cases}
\label{WNLW5}
\end{align}

\noi
where $\RR(w, z)$ is given by 
\begin{align}
\RR(w, z)
& = \, :\!  (z+w)^{3}\!\!:  \,  -\,  w^3 \notag\\
& = 3 z  w^2 + 
3 :\! z^2\!:   w
+ :\! z^3\!:.
\notag
\end{align}

\noi
Then, the proof of Proposition~\ref{PROP:ill2}
is reduced to the following proposition on  almost sure norm inflation 
for the perturbed NLW \eqref{WNLW5}.
\begin{proposition}\label{PROP:ill3}
Let  $s < 0$ and  $z = z^\o$ be as in \eqref{lin1a}. Then, there exists a set $\Si \subset \O$ with $P(\Si) = 1$ such that  given $\o \in \Si$ and $\eps > 0$, 
there exist  a solution $w_\eps^\o$ to \eqref{WNLW5} on $\T^2$ and a random time $t_\eps = t_\eps(\o)  \in (0, \eps) $ such that 
\begin{align}
 \big\| (w_\eps^\o(0), \dt w_\eps^\o(0)) \big\|_{\H^s} < \eps \qquad \text{ and } 
\qquad \| w_\eps^\o(t_\eps)\|_{H^s} > \eps^{-1}.
\notag
\end{align}
\end{proposition}

With $\RR(w_\eps, z) = 0$,  such a norm inflation phenomenon has been studied for the (unrenormalized) NLW \eqref{NLW0}; see \cite{CCT2b, BT1, Xia, Tz1}.
In Proposition~\ref{PROP:ill3},  we establish norm inflation  almost surely  in the presence of  the random perturbation  $\RR(w_\eps, z)$.
We point out that the known result on norm inflation for NLW \eqref{NLW0} on $\T^d$ or $\R^d$ in negative Sobolev spaces only covers a partial range  $s \leq - \frac  d2$
in the general setting; see \cite{CCT2b}.
While there is a norm inflation result
for $s < \frac 16$ by reducing the analysis to the one-dimensional case via the finite speed of propagation
(see \cite[Corollary 7]{CCT2b}), 
this result is not useful
to our problem due to the genuine two-dimensional nature of the random perturbation. Therefore, we first need to extend the deterministic  norm inflation result to cover this missing range  $(-\frac d2, 0)$
{\it without} reducing the analysis to the one-dimensional setting. In fact, this is the goal of the next subsection. More precisely,  we consider the (unrenormalized) NLW and  prove norm inflation (at general initial data)  in negative Sobolev spaces, including the missing range $(-\frac d2, 0)$. This will be a basic building block for the proof of Proposition \ref{PROP:ill3}.

Even with norm inflation for the (unrenormalized) NLW (see Theorem \ref{THM:illposed} below), 
the actual proof of Proposition \ref{PROP:ill3} requires a careful analysis.
The main strategy for proving Proposition \ref{PROP:ill3}
is to establish a good approximation argument
for the perturbed NLW~\eqref{WNLW5}
and the cubic NLW \eqref{NLW0}
and then to invoke the norm inflation for the latter equation.
For this purpose, we need to have local well-posedness of the perturbed NLW~\eqref{WNLW5}
for a sufficiently long time.
In \cite{OTh2}, Thomann and the first author proved almost sure local well-posedness
of \eqref{WNLW5} via the Strichartz estimates and  Lemma \ref{LEM:Z4}.
Due to the use of the space-time estimates, 
such an argument provides a rather short local existence time, which is not sufficient for our purpose.
In order to observe the desired growth for norm inflation, 
we need to {\it maximize} the local existence time by avoiding any use of space-times estimates
such as the Strichartz estimates.
Unfortunately, the local well-posedness argument based on Sobolev's inequality
and the product estimates (Lemma~\ref{LEM:bilin}) within the framework of the $L^2$-based Sobolev spaces
(see \cite{GKOT}) or the Wiener algebra (see Section~\ref{SEC:NI})
does not seem to suffice for our purpose.
We instead establish local well-posedness of the perturbed NLW~\eqref{WNLW5}
in a carefully chosen Fourier-Lebesgue space, 
which provides a sufficiently large time of local existence
{\it and} allows us to implement an approximation argument.
See Section \ref{SEC:NI2} for details.

\subsection{Norm inflation 
for the (unrenormalized) NLW in negative Sobolev spaces}
\label{SUBSEC:illposed}
In this subsection, we change gears and
consider the following (deterministic)  NLW:
\begin{align}
\begin{cases}
\dt^2 u  + (m -  \Dl)  u    +  u^k  = 0\\
(u, \dt u) |_{t = 0} = (u_0, u_1), 
\end{cases}
\qquad (x, t) \in \M \times \R,
\label{NLW3}
\end{align}

\noi
where $m \geq 0$ and  $\M = \T^d$ or $\R^d$. 
When $m= 0$, 
the equation \eqref{NLW3} on $\R^d$
enjoys  the  scaling symmetry,
which induces the so-called scaling critical Sobolev index:
$s_\text{scaling} = \frac{d}{2} - \frac{2}{k-1}$.
On the other hand, NLW also enjoys
the Lorentzian invariance (conformal symmetry), 
which yields its own critical regularity 
$s_\text{conf} = \frac{d+1}{4} - \frac{1}{k-1}$
(at least in the focusing case);
see \cite{LS} and \cite[Exercise 3.67]{TAO}.
We then define the critical regularity $s_\text{crit}$ for a given integer $k \ge 2$ by 
\begin{align} \label{scrit}
s_\text{crit} \stackrel{\text{def}}{=} \max(s_\text{scaling}, s_\text{conf}, 0) 
= \max\bigg(\frac{d}{2} - \frac{2}{k-1}, 
\frac{d+1}{4} - \frac1{k-1}, 0\bigg).
\end{align}


\noi
The Cauchy problem  \eqref{NLW3} has been studied extensively
and it is known that  \eqref{NLW3}
is locally well-posed in $\H^s(\M)$ for $s \geq s_\text{crit}$
in many cases
(possibly under an extra condition); 
see \cite{Kap, LS, KT, Tao1}.

On the other hand, ill-posedness of \eqref{NLW3} below the critical regularity $s_\text{crit}$
has been studied in various papers \cite{LS, CCT2b, BT1, Xia, Tz1}.
In particular, Christ, Colliander, and Tao  \cite{CCT2b} 
proved the following norm inflation phenomenon
for NLW~\eqref{NLW3} on $\R^d$;
given any $\eps > 0$, 
there exist a solution $u_\eps$ to \eqref{NLW3} on $\R^d$
and $t_\eps  \in (0, \eps) $ such that 
\begin{align}
 \| (u_\eps(0), \dt u_\eps(0)) \|_{\H^s(\R^d)} < \eps \qquad \text{but} \qquad \| u_\eps(t_\eps)\|_{H^s(\R^d)} > \eps^{-1}, 
\label{NI1}
 \end{align}

\noi
provided that one of the following conditions holds:
\begin{align}
\text{(a) } 
0 < s < s_\text{scaling} \ \text{ or } \ s <  - \frac 12, 
\quad \text{or}\quad 
\text{(b) } 
 -\frac 12 < s < s_\text{sob} \stackrel{\text{def}}{=} \frac 12 - \frac 1k.
\label{NI1a}
\end{align}

\noi
In particular, when $k = 3$, 
the norm inflation holds
except for $s = -\frac 12$.\footnote{While Theorem 4 in \cite{CCT2b} claims
a norm inflation for 
$s = -\frac 12$ when $ d= 1$, their argument uses 
a scaling and hence seems to break down 
when $s = s_\text{crit} = -\frac 12$, contrary to their claim.}
We point out that, in \cite[Corollary 7]{CCT2b},
 the conditions
(a) and  (b) are  obtained first for $d = 1$ (\cite[Theorem~6]{CCT2b})
and then extended for $d \geq 2$ by reducing the analysis
to the one-dimensional case via the finite speed of propagation.

The norm inflation \eqref{NI1} is a stronger form of instability than 
  discontinuity of 
 the solution map (at the trivial function).
In \cite{Xia},  Xia  proved norm inflation based at general initial data
(see~\eqref{NI2} below)
for NLW on $\T^3$
when $0 < s < s_\text{scaling}$.
See also the lecture note \cite{Tz1} by the third author.
We point out that 
 norm inflation at general initial data can not be reduced
to the one-dimensional setting
and thus the conditions in \eqref{NI1a} should be disregarded in the 
following discussion.
In fact, without reducing the analysis to the one-dimensional setting, 
the argument in \cite{CCT2b} yields norm inflation for
\begin{align*}
\text{(c) } 
d \geq 2: 
0 < s < s_\text{scaling}  \text{ or }  s \leq - \frac d2, 
\quad \text{or}\quad 
\text{(d) }d = 1: 
 s < \frac 16 \text{ and } s \ne - \frac 12, 
\end{align*}

\noi
leaving a gap $-\frac d2 < s \leq 0$ for $d \geq 2$.
See 
\cite[Theorems 4 and 6]{CCT2b}.

In what follows, 
we only consider the cubic case ($k = 3$).
See \cite{F2} for the general case, 
where Forlano and the second author extended our result (Theorem \ref{THM:illposed})
to general $k \geq 2$.
The next theorem establishes norm inflation
at general initial data in negative Sobolev spaces.

\begin{theorem}\label{THM:illposed}
Given $d \in \N$, 
let  $\M = \R^d$ or $\T^d$.
Let  $k = 3$ and $m \geq 0$.
Suppose that  $s \in \R$ satisfies
either 
\textup{(i)} $s \leq  - \frac 12$ 
when $d = 1$
or 
\textup{(ii)} $s <  0$ 
when $d \geq 2$.
%
Fix $(u_0, u_1) \in \H^s(\M)$.
Then, 
given any $\eps > 0$, 
there exist a solution $u_\eps$ to \eqref{NLW3} on $\M$
and $t_\eps  \in (0, \eps) $ such that 
\begin{align}
 \big\| (u_\eps(0), \dt u_\eps(0))  - (u_0, u_1) \big\|_{\H^s(\M)} < \eps \qquad \text{but} 
\qquad \| u_\eps(t_\eps)\|_{H^s(\M)} > \eps^{-1}.
\label{NI2}
\end{align}

\end{theorem}


When $(u_0, u_1)  = 0$, 
Theorem~\ref{THM:illposed} reduces to 
 the usual norm inflation
(based at the zero function) 
stated in \eqref{NI1}.
It follows from   Theorem \ref{THM:illposed}
 that the solution map 
 $\Phi: (u_0, u_1) \in  \H^s(\M) \mapsto (u, \dt u)  \in C([-T, T]; \H^s(\M))$ 
to the cubic NLW is discontinuous
everywhere in $\H^s(\M)$.
Theorem~\ref{THM:illposed}
fills the regularity gap $s \ne - \frac 12$ left open in \cite{CCT2b}
for the usual norm inflation
 in  the case of the cubic nonlinearity ($k = 3$). 
Furthermore, 
our argument exploits a
 more robust 
high-to-low energy transfer mechanism
than that in \cite{CCT2b}
and yields
a norm inflation {\it without}
 reducing the analysis to the one-dimensional setting, 
 which is crucial for proving 
norm inflation at general initial data.

The proof of Theorem~\ref{THM:illposed}
is a basic building block for proving
Proposition~\ref{PROP:ill3} on almost sure norm inflation
for the perturbed NLW \eqref{WNLW5}.
While  the argument in \cite{CCT2b, BT1, Xia}
is based on the (dispersionless)
ODE approach
and an approximation argument, 
%
%
 we adapt the Fourier analytic approach employed in~\cite{Oh1},
where the first author proved an analogous 
norm inflation at general initial data for the cubic nonlinear Schr\"odinger equation
on $\R^d$ and $\T^d$
in negative Sobolev spaces.
The main idea is to 
exploit high-to-low energy transfer
in the Picard second iterate.
We refer readers to the previous works
\cite{BT, IO, Kishimoto, CP, Ok}, 
where a similar approach has been taken.
We also mention 
the work  \cite{BonaT, F1} which  exploits high-to-low energy transfer.

\medskip

Let us briefly describe the idea of the proof of Theorem~\ref{THM:illposed}.
By a density argument, we may assume that $(u_0, u_1) \in \S(\M)\times \S(\M)$,
where  $\S(\M)$ denotes the class of Schwartz functions if $\M = \R^d$
 and the class of $C^\infty$-functions if $\M = \T^d$.
See Proposition \ref{PROP:NI3} below.
Then, 
the main goal is to  construct
a pair 
$(\phi_{0, \eps}, \phi_{1, \eps})\in C^\infty(\M)\times C^\infty(\M)$,
$\eps > 0$, 
such that a solution $u_\eps$ to \eqref{NLW3}
with initial data
$(u_{0, \eps},  u_{1, \eps}) = (u_0, u_1) + (\phi_{0, \eps}, \phi_{1, \eps})$
satisfies the conclusion of Theorem \ref{THM:illposed}.

By expressing $u_\eps$   in the Duhamel formulation (with $m = 1$), we have
\begin{align}
u_\eps(t)
& = S(t) (u_{0,\eps}, u_{1, \eps})
- \int_{0}^t  \frac{\sin ((t-t')\jb{\nb})}{\jb{\nb}}  
u_\eps^3(t') dt'.
\notag 
\end{align}

\noi
As in \cite{Oh1}, 
the main ingredient is to express a  smooth solution $u_\eps$
in the  following power series expansion:
\begin{align*}
 u_\eps & 
 = \sum_{j = 0}^\infty \Xi_j (u_{0, \eps}, u_{1, \eps}),
 \end{align*}

\noi
where $\Xi_j(u_{0, \eps}, u_{1, \eps})$ denotes 
homogeneous multilinear terms  of 
degree $2j+1$ (in the linear solution $S(t) (u_{0, \eps}, u_{1, \eps})$).
We then construct $(\phi_{0, \eps}, \phi_{1, \eps})$ such that, as $\eps \to 0$,  

\smallskip

\begin{itemize}
\item[(i)] $(\phi_{0, \eps}, \phi_{1, \eps})$ tends to 0 in $\H^s(\M)$, 

\smallskip

\item[(ii)] the second order term $\Xi_1(u_{0, \eps}, u_{1, \eps})(t_\eps)$
tends to $\infty$ for some $t_\eps \to 0$,

\smallskip

\item[(iii)]  the sum of the higher ordered terms  $\Xi_j(u_{0, \eps}, u_{1, \eps})(t_\eps)$,
$j \geq 2$, is of smaller order than the second order term
 $\Xi_1(u_{0, \eps}, u_{1, \eps})(t_\eps)$.  

\end{itemize}

\smallskip

\noi
This yields the conclusion of Theorem \ref{THM:illposed}.
We remark that, in \cite{IO, Kishimoto, CP, Ok}, 
 $\Xi_j$ was defined in a recursive manner
and the (scaled) modulation space $M_{2, 1} (\M)$
and its algebra property played an important role.
In the following, however, we follow a  simplified approach presented in \cite{Oh1}
and directly define $\Xi_j$ 
via the power series expansion indexed by trees
and use the Wiener algebra $\F L^1(\M)$ instead of the modulation space. 
This latter approach is  more suitable for proving norm inflation at general initial data.

%

\subsection{Remarks and comments}

We conclude this introduction by several remarks.

\smallskip

\noi
(i) 
In the main results (Theorems \ref{THM:uniq} and \ref{THM:ill0}), 
we only considered the cubic case.
It is easy to see that 
Theorem \ref{THM:uniq} is readily extendable to the case $k \geq 5$. Our method for proving
Theorem~\ref{THM:illposed}
on the norm inflation at general initial data
is elementary
and can be applied to other power-type
nonlinearities.
Following this paper, 
the second author and Forlano
\cite{F2} recently established an analogous norm inflation result
for \eqref{NLW3}
with a general power-type nonlinearity $u^k$. 
It is likely that 
 Theorem \ref{THM:ill0} can also be extended for  $k \geq 5$.
We point out, however, 
that a careful analysis 
(beyond establishing
norm inflation at general initial data)
 is needed in proving an analogue of Proposition \ref{PROP:ill3}.
See Section \ref{SEC:NI2}.

\smallskip

\noi
(ii) 
The defocusing nature of the equation is needed only 
in obtaining global solutions (Theorems \ref{THM:OTh}
and Proposition~\ref{PROP:easy}).

\smallskip

\noi
(iii) 
Consider the cubic nonlinear Schr\"odinger equation (NLS) on $\T^d$:
\begin{align}
i \dt u - \Dl u + |u|^2u = 0.
\label{NLS1}
\end{align}

\noi
In this case, we can introduce a renormalization
 in a deterministic manner:
\begin{align}\textstyle
i \dt u - \Dl u + \big( |u|^2 - 2 \int |u|^2 dx \big) u = 0
\label{NLS2}
\end{align}

\noi
to study the dynamics with either random or deterministic initial data of low regularity. 
See \cite{BO96, CO, OS, GO, OW2}. 
Thanks to the $L^2$-conservation, 
the equations \eqref{NLS1} and \eqref{NLS2} are equivalent, at least for smooth solutions, 
via the invertible gauge transform: $u \mapsto e^{2i t \int |u|^2 dx } u$. 
Furthermore, in the case of Gaussian random initial data
(under some regularity restriction), 
the renormalized equation \eqref{NLS2} is equivalent to the
renormalized equation via the Wick renormalization
(as in \eqref{Herm4} but in the complex-valued setting);
see \cite{BO96, OS, OTh1}.
We point out that a deterministic renormalization as in \eqref{NLS2}
has also been used to study  the fractional NLS; see \cite{OW, OTzW}.

In case of the cubic NLW, 
it is tempting to consider a deterministic renormalization
analogous to \eqref{NLS2}:
\begin{align}
\textstyle 
\dt^2 u + (1- \Dl) u +  \big( u^2- 3 \int u^2dx \big)u = 0.
\label{GNLW1}
\end{align}

\noi
Denoting the nonlinearity in \eqref{GNLW1} by 
 $f(u)$, 
 its spatial Fourier transform is written as 
\begin{align}
 \ft{f(u)} (n)
= \sum_{\substack{n = n_1 + n_2 + n_3\\(n_1+ n_2) (n_2 + n_3)(n_3 + n_1)\ne 0}}
\prod_{j = 1}^3 \ft u(n_j)
- 3 |\ft u(n)|^2 \ft u(n)  + \ind_{n = 0} (\ft u(0))^3.
\label{GNLW2}
\end{align}

\noi
This renormalization 
cancels certain resonant interactions ($n_j + n_k = 0$ for $j \ne k$), 
which allows us to make sense of $f(u)$
for 
$u$ of the form \eqref{u2} with  
 $z$  as in \eqref{lin1a} and smoother $v$.
Indeed,  the problematic terms $z^3$ and $3z^2v$ in $(z+v)^3$
are now modified into 
$z^3 - 3 \int z^2dx \cdot z$ and $3 \big( z^2 - \int z^2) v$, 
each of which has a well-defined meaning.

There are, however, two issues in using the renormalized model \eqref{GNLW2}.
Unlike the cubic NLS, the renormalized model \eqref{GNLW2}
is not naturally associated with 
the unrenormalized model \eqref{NLW0} with $k = 3$
in the sense that it is not equivalent to  
the unrenormalized model
even for smooth solutions, 
in particular, due to the lack of the $L^2$-conservation for NLW. 
The second point is that the renormalized model \eqref{GNLW2}
possesses finite-time blowup solutions,\footnote{For a function $u$ independent
of the spatial variable, the defocusing
``renormalized''
nonlinearity in~\eqref{GNLW1}
becomes the focusing (unrenormalized) nonlinearity:
$  \big( u^2- 3 \int u^2dx \big)u = - 2 u^3$, 
showing that there exists a finite time blowup solution
$u(t) \sim \sqrt{2}(T_*- t)^{-1}$ in the sense of asymptotic equality as $t - T_*-$.
}
whereas the Wick ordered NLW \eqref{WNLW0} is almost surely globally well-posed;
see Theorem~\ref{THM:OTh} and Proposition~\ref{PROP:easy}.
See also \cite{GKOT}.

\smallskip

\noi
(iv)  
The main results of this paper
are readily applicable to the two-dimensional stochastic NLW with space-time white noise forcing studied
in \cite{GKO, GKOT, OOcomp}.
Moreover,  
our work provides   a natural framework for obtaining similar non-uniqueness results for singular stochastic PDEs. 
For instance, it would be interesting to establish an analogue of  Theorem~\ref{THM:ill0} 
in the context of the stochastic wave equations in higher dimensions
\cite{OOT1, OOT2, Bring, BDNY} and the stochastic heat equations
\cite{DPD, Hairer14, CC, MW1}. 
We mention a recent work \cite{HZZ}
on the stochastic Navier-Stokes equations.

\medskip

This remaining part of the paper is organized as follows.
In Section~\ref{SEC:2}, we collect some deterministic and stochastic lemmas.
In Section~\ref{SEC:GWP}, we prove Proposition \ref{PROP:easy}.
In Section~\ref{SEC:uniq}, we show the convergence and uniqueness of Wick powers and then present the proof of Theorem~\ref{THM:uniq}.
In Section~\ref{SEC:NI}, we prove norm inflation at general initial data
for the deterministic cubic NLW \eqref{NLW3} with $k = 3$
(Theorem \ref{THM:illposed}).
In Section~\ref{SEC:NI2}, we
first establish local well-posedness of the perturbed NLW \eqref{WNLW5} in 
a carefully chosen  Fourier-Lebesgue space 
(Lemma \ref{LEM:LWP1}) and an approximation lemma (Lemma \ref{LEM:approx}), which implies Proposition~\ref{PROP:ill3}.
In Appendix~\ref{SEC:A}, we present the proof of the almost sure convergence of stochastic objects (Proposition~\ref{PROP:reg}).

\section{Deterministic and stochastic lemmas}
\label{SEC:2}
\subsection{Hermite polynomials
and white noise functional}
First, we recall the Hermite polynomials $H_k(x; \s)$
defined via the generating function:
\begin{align}
F(t, x; \s) \stackrel{\text{def}}{=}  e^{tx - \frac{1}{2}\s t^2} = \sum_{k = 0}^\infty \frac{t^k}{k!} H_k(x;\s).
\label{H1}
 \end{align}
	
\noi
For simplicity, we set 
$F(t, x) \stackrel{\text{def}}{=} F(t, x; 1)$ and
$H_k(x) \stackrel{\text{def}}{=} H_k(x; 1)$.
Note that $H_k (x;\s) = \s^{\frac k2} H_k (\s^{-\frac 12} x)$ holds.
In the following, we list the first few Hermite polynomials
for readers' convenience:
\begin{align}
 H_0(x;\s) = 1, 
\quad 
H_1(x;\s) = x, 
\quad
H_2(x;\s) = x^2 - \s, 
\quad  H_3(x;\s) = x^3 - 3 \s x.
\label{H1a}
\end{align}
	
\noi
For the derivative, the following properties hold:
\begin{align}
\dx H_k(x) = k H_{k-1} (x) \quad \text{and}\quad 
H_k (x) = x H_{k-1}(x) - \dx H_{k-1}(x).
\label{H2}
\end{align}

Next, we define the white noise functional.
Let  $\xi(x;\o)$ be the (real-valued) mean-zero Gaussian white noise on $\T^2$
defined by
\[ \xi(x;\o) = \sum_{n\in \Z^2} g_n(\o) e^{in\cdot x},\]

\noi
where 
$\{g_{ n} \}_{n \in \Z^2}$
is a sequence of independent standard complex-valued Gaussian
random variables
conditioned that $g_{ -n} = \cj{g_{n}}$, $n \in \Z^2$.
It is easy to see that $\xi \in \H^s(\T^2) \setminus \H^{-1}(\T^2)$, $s < -1$, almost surely.
In particular, $\xi$ is a distribution, acting  on smooth functions.
In fact, the action of $\xi$ can be defined on $L^2(\T^2)$.
We define the white noise functional  $W_{(\cdot)}: L^2(\T^2) \to L^2(\O)$
by 
\begin{align}
 W_f (\o) = \jb{f, \xi(\o)}_{L^2} = \sum_{n \in \Z^2} \ft f(n) \cj{g_n}(\o)
\label{W0}
 \end{align}

\noi
for a real-valued function $f \in L^2(\T^2)$.
Note that $W_f=\xi(f)$ is basically the Wiener integral of $f$.
In particular, 
$W_f$ is a real-valued Gaussian random variable
with mean 0 and variance $\|f\|_{L^2}^2$.
Moreover, 
$W_{(\cdot)}$ is unitary:
\begin{align}
\E\big[ W_f W_h ] = \jb{f, h}_{L^2}
 \label{W0a}
\end{align}

\noi
for $f, h \in L^2(\T^2)$.
In general, we have the following lemma.
See \cite[Lemma 1.1.1]{Nu}.

\begin{lemma}\label{LEM:W1}
\textup{(i)}
Let $g_1$ and $g_2$ be mean-zero real-valued jointly Gaussian random variables with variances $\s_1$
and $\s_2$.
Then, we have 
\begin{align*}
\E\big[ H_k(g_1; \s_1) H_m(g_2; \s_2)\big] = \dl_{km} k! \big\{\E[ g_1 g_2] \big\}^k.
\end{align*}

\smallskip

\noi
\textup{(ii)}
Let $f, h \in L^2(\T^2)$ such that $\|f\|_{L^2} = \|h\|_{L^2} = 1$.
Then, for $k, m \in \N \cup \{ 0 \}$, we have 
\begin{align*}
\E\big[ H_k(W_f)H_m(W_h)\big]
=  \dl_{km} k! [\jb{f, h}_{L^2}]^k.
\end{align*}

\noi
Here, $\dl_{km}$ denotes the Kronecker's delta function.
\end{lemma}

Part (i) of Lemma \ref{LEM:W1} easily follows from the definition \eqref{H1}
of the generating function: 
\[\E [ F(t, g_1; \s_1) F(s, g_2; \s_2)]
=  \sum_{k, m  = 0}^\infty \frac{t^k}{k!}
 \frac{s^m}{m!}
\E\big[ H_k(g_1; \s_1)H_m(g_2;\s_2)\big], \]

\noi
while Part (ii) is an immediate corollary of Part (i) and \eqref{W0a}.

As in \cite{OTh2}, we also employ the white noise functional adapted to $z(t)$.
In view of \eqref{lin2}, 
we define  the white noise functional 
$W^t_{(\cdot)}: L^2(\T^2) \to L^2(\O)$
with a parameter $t \in \R$ by 
\begin{align}
 W^t_f (\o) = \jb{f, \xi^t(\o)}_{L^2} = \sum_{n \in \Z^2} \ft f(n) \cj{g^t_{0,n}}(\o).
\label{PStr2}
 \end{align}

\noi
Here,  $\xi^t$ denotes (a specific realization of) the white noise on $\T^2$ given by 
\[\xi^t (x; \o) = \sum_{n\in \Z^2} g_{0,n}^t (\o) e^{in \cdot x},\]

\noi
where $g_{0,n}^t$ is defined in \eqref{lin3}.
\noi
Since $\{g_{0,n}^t\}_{n\in \Z^2}$ is a sequence of independent standard Gaussian random variables with
$
g^t_{0,-n} = \cj{g^t_{0,n}}$,
 the white noise functional
$W^t_{(\cdot)}$ defined in~\eqref{PStr2} satisfies
the same properties as the standard white noise functional 
$W_{(\cdot)}$ defined in~\eqref{W0}. 
Moreover, we have the following lemma.

\begin{lemma}\label{LEM:W2}
Let $f, h \in L^2(\T^2)$ such that $\|f\|_{L^2} = \|h\|_{L^2} = 1$.
Then, for $k, m \in \N \cup \{ 0\}$ and $t_1, t_2 \in \R$, we have 
\begin{align}
\E\big[ H_k(W_f^{t_1})H_m(W_h^{t_2})\big]
=  \dl_{km} k! \, ( \IP (f,h)[t_1-t_2] )^k,
\label{W2}
\end{align}

\noi
where
\[
\IP (f,h)[t] = \sum_{n \in \Z^2} \ft f(n) \cj{\ft h(n)} \cos (t \jb{n}).
\]
\end{lemma}

While Lemma~\ref{LEM:W2} follows from a similar argument as in the proof of Lemma 3.4 in \cite{ORT}, for readers' convenience, we provide a proof here.

\begin{proof}
From  \eqref{PStr2} with \eqref{lin3}, we have 
\begin{align*}
W_{f}^{t_1}(\o) +W_{h}^{t_2}(\o)
=& \sum_{n\in\Z^2} \Big\{ \Big( \ft f(n) \cos( t_1\jb{n}) + \ft h(n) \cos (t_2 \jb{n}) \Big) \cj{g_{0,n} (\o)} \\
& \quad  + \Big( \ft f(n) \sin (t_1 \jb{n}) + \ft h(n) \sin (t_2 \jb{n}) \Big) \cj{g_{1,n} (\o)} \Big\} \\
=& \sum_{n\in\Z^2} \Big\{ \Re \Big( \ft f(n) \cos( t_1\jb{n}) + \ft h(n) \cos (t_2 \jb{n}) \Big) \Re g_{0,n} (\o) \\
& \quad  + \Im \Big( \ft f(n) \cos( t_1\jb{n}) + \ft h(n) \cos (t_2 \jb{n}) \Big) \Im g_{0,n} (\o) \\
& \quad  + \Re \Big( \ft f(n) \sin (t_1 \jb{n}) + \ft h(n) \sin (t_2 \jb{n}) \Big) \Re g_{1,n} (\o) \\
& \quad  + \Im \Big( \ft f(n) \sin (t_1 \jb{n}) + \ft h(n) \sin (t_2 \jb{n}) \Big) \Im g_{1,n} (\o) \Big\},
\end{align*}

\noi
where the second equality follows from \eqref{B1}
and the fact that $f$ and $h$ are real-valued.
Since $\Re g_{j,n}$ and $\Im g_{j,n}$ are independent Gaussian random variables with mean $0$ and variance $\frac{1}{2}$ for $n\neq 0$ ($1$ if $n=0$) and $g_{j,-n} = \cj{g_{j,n}}$, we have
\begin{align*}
\int_\O e^{t W_{f}^{t_1}(\o)} e^{sW_{h}^{t_2}(\o)} dP(\o)
= e^{\frac 12 (t^2 \|f \|_{L^2}^2+s^2 \|h \|_{L^2}^2+2 \IP (f,h)[t_1-t_2])}
\end{align*}

\noi
for any $t, s \in \R$, 
where once again we used
the fact that $f$ and $h$ are real-valued.

Let $F$ be as in\;\eqref{H1}.
Then, for any $t, s \in \R$ and $f, h \in L^2(\T^2)$ with $\|f \|_{L^2} = \|h\|_{L^2} = 1$, we have 
\begin{align}
\int_{\O} F(t, W_f^{t_1}(\o)) F(s, W_h^{t_2}(\o)) dP(\o)
& = e^{-\frac{t^2 + s^2}{2}} 
\int_\O e^{tW_{f}^{t_1}(\o) +sW_{h}^{t_2}(\o)} dP(\o) \notag \\
&= e^{ts \IP (f,h)[t_1-t_2]}.
\label{W2a}
\end{align}

\noi
Thus, it follows from\;\eqref{H1} and\;\eqref{W2a} that 
\begin{align*}
e^{ts \IP (f,h)[t_1-t_2]}
 = \sum_{k, m  = 0}^\infty 
 \frac{t^ks^m}{k!m!} 
\int_\O H_k(W_f^{t_1}(\o))H_m(W_h^{t_2}(\o))dP(\o).
\end{align*}
By comparing the coefficients of $t^ks^m$, we obtain \eqref{W2}.
\end{proof}

\subsection{Product estimates}

We recall the following product estimates.
See \cite{GKO} for the proof.

\begin{lemma}\label{LEM:bilin}
 Let $0\le \al \le 1$.

\smallskip

\noi
\textup{(i)} Suppose that 
 $1<p_j,q_j,r < \infty$, $\frac1{p_j} + \frac1{q_j}= \frac1r$, $j = 1, 2$. 
 Then, we have  
\begin{equation*}  
\| \jb{\nb}^\al (fg) \|_{L^r(\T^d)} 
\les \Big( \| f \|_{L^{p_1}(\T^d)} 
\| \jb{\nb}^\al g \|_{L^{q_1}(\T^d)} + \| \jb{\nb}^\al f \|_{L^{p_2}(\T^d)} 
\|  g \|_{L^{q_2}(\T^d)}\Big).
\end{equation*}

\smallskip

\noi
\textup{(ii)} 
Suppose that  
 $1<p,q,r < \infty$ satisfy the scaling condition:
$\frac1p+\frac1q\le \frac1r + \frac{\al}d $.
Then, we have
\begin{align*}
\| \jb{\nb}^{-\al} (fg) \|_{L^r(\T^d)} \les \| \jb{\nb}^{-\al} f \|_{L^p(\T^d) } 
\| \jb{\nb}^\al g \|_{L^q(\T^d)}.  
\end{align*}

\end{lemma}

Note that
while  Lemma \ref{LEM:bilin} (ii) 
was shown only for 
$\frac1p+\frac1q= \frac1r + \frac{\al}d $
in \cite{GKO}, 
the general case
$\frac1p+\frac1q\leq \frac1r + \frac{\al}d $
follows from the inclusion $L^{r_1}(\T^d)\subset L^{r_2}(\T^d)$
for $r_1 \geq r_2$.

\subsection{Tools from stochastic analysis}

We use the short-hand notation $L_T^q L_x^r = L^q([-T,T];L^r(\T^2))$ for $T>0$ and $1 \leq q,r \le \infty$, etc.
Thanks to the randomization of the initial data, the following probabilistic Strichartz estimates hold.

\begin{lemma}\label{LEM:PStr}
Given $(\phi_0,\phi_1) \in \H^0(\T^2)$, let $(\phi_0^0,\phi_1^\o)$ be its randomization defined in \eqref{u4}.
{\rm (i)}
Given $2\leq q<\infty$ and $2 \leq r< \infty$, there exist $C, c>0$ such that
\begin{align*}
P \Big( \| S(t) (\phi_0^\o,\phi_1^\o) \|_{L^q_T L^r_x}> \ld \Big)
&\leq C\exp\bigg(-c\frac{\ld^2}{T^{\frac{2}{q}} \| (\phi_0,\phi_1) \|_{\H^0}^2}\bigg)
\end{align*}
for any $T > 0$ and $\ld > 0$.

\noi
{\rm (ii)}
Let  $s>0$ and $(\phi_0,\phi_1) \in \H^s(\T^2)$.
Then, given $2 \leq r \leq \infty$, there exist $C, c>0$ such that
\begin{align*}
P \Big( \| S(t) (\phi_0^\o,\phi_1^\o) \|_{L^\infty_T L^r_x}> \ld \Big)
&\leq C (1+T) \exp\bigg(-c\frac{\ld^2}{\max (1,T^2) \| (\phi_0,\phi_1) \|_{\H^s}^2}\bigg)
\end{align*}
for any $T > 0$ and $\ld > 0$.
\end{lemma}

The probabilistic Strichartz estimate in (i) of Lemma \ref{LEM:PStr} is proved in \cite{CO, BT3, BOP1}.
See \cite{BTT, OP} for (ii) of Lemma \ref{LEM:PStr}.
While the (deterministic) Strichartz estimate holds only for admissible pairs (see \cite{GV, LS, KT, TAO}), Lemma \ref{LEM:PStr} states that the randomization allows us to take a wide range of exponents.

Next, we recall the Wiener chaos estimates.
Let $\{ g_n \}_{n \in \N}$ be a sequence of independent standard Gaussian random variables defined on a probability space $(\O, \F, P)$, where $\mathcal{F}$ is the $\s$-algebra generated by this sequence. 
Given $k \in \N \cup \{ 0 \}$, 
we define the homogeneous Wiener chaoses $\mathcal{H}_k$ 
to be the closure (under $L^2(\O)$) of the span of  Fourier-Hermite polynomials $\prod_{n = 1}^\infty H_{k_n} (g_n)$, 
where
$H_j$ is the Hermite polynomial of degree $j$ and $k = \sum_{n = 1}^\infty k_n$.\footnote{This implies
that $k_n = 0$ except for finitely many $n$'s.}
Then, we have the following Ito-Wiener decomposition:
\begin{equation*}
L^2(\Omega, \F, P) = \bigoplus_{k = 0}^\infty \mathcal{H}_k.
\end{equation*}

\noi
See Theorem 1.1.1 in \cite{Nu}.
We also set
\begin{align*}
 \H_{\leq k} = \bigoplus_{j = 0}^k \H_j
 \end{align*}

\noi
 for $k \in \N$.

Then, as a consequence
of the  hypercontractivity of the Ornstein-Uhlenbeck
semigroup $U(t) = e^{tL}$ due to Nelson \cite{Nelson2}, 
we have the following Wiener chaos estimate
\cite[Theorem~I.22]{Simon}.
See also \cite[Proposition~2.4]{TTz}.

\begin{lemma}\label{LEM:hyp}
Let $k \in \N$.
Then, we have
\begin{equation*}
\|X \|_{L^p(\O)} \leq (p-1)^\frac{k}{2} \|X\|_{L^2(\O)}
 \end{equation*}
 
 \noi
 for any $p \geq 2$
 and any $X \in \H_{\leq k}$.

\end{lemma}

Note that $:\! z^\l (t)\!:$ defined in \eqref{Wick4} belongs to $\H_{\le \l}$ for $\l \in \N$.
By using the white noise functional defined in \eqref{PStr2} and Lemma \ref{LEM:hyp}, Thomann and the first author \cite{OTh2} proved the following estimate on Wick powers.

\begin{lemma} \label{LEM:Z4}
Let $\l \in \N\cup \{0\}$.
Then, 
given $2 \le q,r < \infty$ and $\eps>  0$, 
there exists $C,c>0$ such that
\begin{align*}
P \Big( \| \jb{\nb}^{-\eps} :\! z^\l \!: \|_{L^q_T L^r_x}> \ld \Big)
&\leq C\exp\bigg(-c\frac{\ld^{\frac 2 \l}}{T^{\frac{2}{q \l}}}\bigg)
\end{align*}
for any $T > 0$ and $\ld > 0$.
\end{lemma}

Note that an analogous estimate also holds even when $q = r = \infty$;
see \cite{GKOT}.

\smallskip

We conclude this section by stating a proposition
useful for studying regularities of stochastic objects.
We say that a stochastic process $X:\R_+ \to \mathcal{D}'(\T^d)$
is spatially homogeneous  if  $\{X(\cdot, t)\}_{t\in \R_+}$
and $\{X(x_0 +\cdot\,, t)\}_{t\in \R_+}$ have the same law for any $x_0 \in \T^d$.
Given $h \in \R$, we define the difference operator $\dl_h$ by setting
\begin{align}
\dl_h X(t) = X(t+h) - X(t).
\label{diff1}
\end{align}

\begin{proposition}\label{PROP:reg}
Let $\{ X_N \}_{N \in \N}$ and $X$ be spatially homogeneous stochastic processes\,$:\R_+ \to \mathcal{D}'(\T^d)$.
Suppose that there exists $k \in \N$ such that 
  $X_N(t)$ and $X(t)$ belong to $\H_{\leq k}$ for each $t \in \R_+$.

\smallskip
\noi\textup{(i)}
Let $t \in \R_+$.
If there exists $s_0 \in \R$ such that 
\begin{align}
\E\big[ |\ft X(n, t)|^2\big]\les \jb{n}^{ - d - 2s_0}
\label{reg1}
\end{align}

\noi
for any $n \in \Z^d$, then  
we have
$X(t) \in W^{s, \infty}(\T^d)$, $s < s_0$, 
almost surely.
Furthermore, if there exists $\g > 0$ such that 
\begin{align}
\E\big[ |\ft X_N(n, t) - \ft X(n, t)|^2\big]\les N^{-\g} \jb{n}^{ - d - 2s_0}
\label{reg2}
\end{align}

\noi
for any $n \in \Z^d$ and $N \geq 1$, 
then 
$X_N(t)$ converges to $X(t)$ in $W^{s, \infty}(\T^d)$, $s < s_0$, 
almost surely.

\noi

\smallskip
\noi\textup{(ii)}
Let $T > 0$ and suppose that \textup{(i)} holds on $[0, T]$.
If there exists $\theta \in (0, 1)$ such that 
\begin{align}
 \E\big[ |\dl_h \ft X(n, t)|^2\big]
 \les \jb{n}^{ - d - 2s_0+ \theta}
|h|^\theta, 
\label{reg3}
\end{align}

\noi
for any  $n \in \Z^d$, $t \in [0, T]$, and $h \in [-1, 1]$,\footnote{We impose $h \geq - t$ such that $t + h \geq 0$.}
then we have 
$X \in C([0, T]; W^{s, \infty}(\T^d))$, 
$s < s_0 - \frac \theta2$,  almost surely.
Furthermore, 
if there exists $\g > 0$ such that 
\begin{align}
 \E\big[ |\dl_h \ft X_N(n, t) - \dl_h \ft X(n, t)|^2\big]
 \les N^{-\g}\jb{n}^{ - d - 2s_0+ \theta}
|h|^\theta, 
\label{reg4}
\end{align}

\noi
for any  $n \in \Z^d$, $t \in [0, T]$,  $h \in [-1, 1]$, and $N \geq 1$, 
then 
$X_N$ converges to $X$ in $C([0, T]; W^{s, \infty}(\T^d))$, $s < s_0 - \frac{\theta}{2}$,
almost surely.

\end{proposition}

Proposition~\ref{PROP:reg} follows
from a straightforward application of the Wiener chaos estimate
(Lemma \ref{LEM:hyp}).
For the proof, see Proposition 3.6 in \cite{MWX}
and  Appendix \ref{SEC:A}.
In particular, for the almost sure convergence claimed
in Proposition \ref{PROP:reg}, 
we need to proceed as in a standard proof of
Kolmogorov'a continuity criterion; see Appendix \ref{SEC:A}
for details.
See also Section~3 in  \cite{OPTz}.

As a corollary, we also have the following (See Remark \ref{REM:cor}).

\begin{corollary} \label{COR:reg1}
Let $\{ X_N \}_{N \in \N}$ be a spatially homogeneous stochastic process\,$:\R_+ \to \mathcal{D}'(\T^d)$.
Suppose that there exists $k \in \N$ such that 
  $X_N(t)$ belongs to $\H_{\leq k}$ for each $t \in \R_+$.

\smallskip
\noi\textup{(i)}
Let $t \in \R_+$.
If there exist $s_0 \in \R$ and $\g > 0$ such that 
\begin{align*}
\E\big[ |\ft X_N(n, t)|^2\big]
&\les \jb{n}^{ - d - 2s_0}, \\
\E\big[ |\ft X_N(n, t) - \ft X_M(n, t)|^2\big]
&\les N^{-\g} \jb{n}^{ - d - 2s_0}
\end{align*}
for any $n \in \Z^d$ and $M \ge N \geq 1$, 
then 
$X_N(t)$ converges in $W^{s, \infty}(\T^d)$, $s < s_0$, 
almost surely.

\noi

\smallskip
\noi\textup{(ii)}
Let $T > 0$ and suppose that \textup{(i)} holds on $[0, T]$.
If there exist $\g > 0$ and $\theta \in (0, 1)$ such that 
\begin{align*}
\E\big[ |\dl_h \ft X_N(n, t)|^2\big]
&\les \jb{n}^{ - d - 2s_0+ \theta}
|h|^\theta, \\
\E\big[ |\dl_h \ft X_N(n, t) - \dl_h \ft X_M(n, t)|^2\big]
&\les N^{-\g}\jb{n}^{ - d - 2s_0+ \theta}
|h|^\theta, 
\end{align*}

\noi
for any  $n \in \Z^d$, $t \in [0, T]$,  $h \in [-1, 1]$, and $M \ge N \geq 1$, 
then 
$X_N$ converges in $C([0, T]; W^{s, \infty}(\T^d))$, $s < s_0 - \frac{\theta}{2}$,
almost surely.

\end{corollary}

Proposition \ref{PROP:reg} and Corollary \ref{COR:reg1}
have been  useful widely in 
the recent study of singular stochastic PDEs;
see for example, \cite{GKO2, OOcomp, GKOT, OOT1, OOT2}.

\section{Global existence of smooth solutions for the renormalized NLW}
\label{SEC:GWP}
In this section, we present the proof of Proposition~\ref{PROP:easy}.
We point out that, thanks to the Cameron-Martin theorem \cite{CM}, we can assume that $r_0=r_1=0$.
See also \cite{OQ}. 
Hence, it suffices to study 
\begin{align} \label{WNLW3}
\dt^2 v + (1- \Dl) v + H_k ( S(t)(\phi_0^\o, \phi_1^\o) +v(t); \sigma(t)) = 0
\end{align}
with the zero initial data, where $\sigma(t)$ is defined by \eqref{sigma_pak}. In particular, it satisfies 
\begin{equation}\label{u5}
\sigma(t) \les \| (\phi_0,\phi_1) \|_{\H^0}^2 \qquad
\text{and}\qquad
|\dt \sigma(t)| \les \| (\phi_0,\phi_1) \|_{\H^\frac 12}^2.
\end{equation}

We first go over local well-posedness of \eqref{WNLW3}.
For this purpose, we consider the following
deterministic perturbed cubic NLW:
\begin{align} \label{smoothv2}
\dt^2 v + (1- \Dl) v + H_k ( f(t) +v(t); \s(t)) = 0,
\end{align}
where $f$ is a given deterministic  function and $\s(t)$ satisfies \eqref{u5}.

\begin{lemma} \label{LEM:LWPreg}
Let $k \ge 3$ be an odd integer,  $(\phi_0,\phi_1) \in \H^0(\T^2)$, 
 $(v_0,v_1) \in \H^1(\T^2)$,  and $f \in L^k([t_0, t_0 + 1]; L^{\infty}(\T^2))$ 
 for some  $t_0\in \R$.
Suppose that there exist $R, \theta> 0$ such that 
\begin{align} \label{LWPcon}
\| (v_0,v_1) \|_{\H^1} \le R \qquad
\text{and} \qquad 
\| f \|_{L^k(I; L^{\infty}(\T^2))} \le |I|^\theta
\end{align}
for any interval $I \subset [t_0, t_0 + 1]$.
Then, there exist $\tau =\tau (R, \theta, \| (\phi_0,\phi_1)\|_{\H^0} )>0$ and a unique solution 
$(v, \dt v)  \in C([t_0,t_0+\tau]; \H^1(\T^2))$ to \eqref{smoothv2} with $(v, \dt v) |_{t = t_0} = (v_0,v_1)$.
\end{lemma}

\begin{remark}\label{REM:LWPreg}\rm
We point out that the second condition in \eqref{LWPcon} can be weakened as follows.
Let $\tau =\tau ( R, \theta, \| (\phi_0,\phi_1)\|_{\H^0} )>0$ be as in Lemma \ref{LEM:LWPreg}.
If we assume
\[
\| f \|_{L^k([t_0,t_0+\tau_\ast]; L^{\infty}(\T^2))} \le \tau_\ast^\theta
\]
for some $0<\tau_\ast \le \tau$ instead of the second condition in \eqref{LWPcon}, 
then the conclusion of Lemma~\ref{LEM:LWPreg} still holds on $[t_0,t_0+\tau_{\ast}]$.
\end{remark}

\begin{proof}[Proof of Lemma \ref{LEM:LWPreg}]
Without loss of generality, we may assume $t_0=0$
and restrict our attention only to positive times.
By writing \eqref{smoothv2} in the Duhamel formulation, we have
\begin{align*}
v(t) & = \Phi(v)(t) \\ &
\hspace{-0.4mm} \stackrel{\text{def}}{=}  S(t) (v_0,v_1) - \int_0^t \frac{\sin ((t-t') \jb{\nabla})}{\jb{\nabla}} H_k ( f(t') +v(t'); \s(t')) dt'.
\end{align*}

\noi
Let $\vec \Phi(v) = (\Phi(v), \dt \Phi(v))$ and $\vec v = ( v, \dt v)$.
Our goal is to show that $\vec \Phi$ is a contraction mapping in a suitable functional framework. 

Let  $0<T\le 1$.
Then, it follows from \eqref{u5}, \eqref{LWPcon}, and Sobolev's inequality that
\begin{align*}
\| H_k ( f+v; \s) \|_{L_T^1 L_x^2}
&\le \sum_{\l=0}^k \begin{pmatrix} k \\ \l \end{pmatrix} \| H_\l (f;\s) v^{k-\l} \|_{L_T^1 L_x^2} \\
&\le \| v^k \|_{L_T^1 L_x^2} + \sum_{\l=1}^k \begin{pmatrix} k \\ \l \end{pmatrix} \| H_\l (f;\s) \|_{L_T^1 L_x^{\infty}} \| v^{k-\l} \|_{L_T^{\infty} L_x^2} \\
&\les T \| v \|_{L_T^{\infty} L_x^{2k}}^k + \sum_{\l=1}^k \big( \| f \|_{L_T^\l L_x^{\infty}}^\l + T \| \s \|_{L_T^{\infty}}^{\frac{\l}{2}} \big) \| v \|_{L_T^{\infty} L_x^{2(k-\l)}}^{k-\l} \\
&\les T \| v \|_{L_T^{\infty} L_x^{2k}}^k + \sum_{\l=1}^k \big( T^{\theta \l} + T \| (\phi_0,\phi_1) \|_{\H^0}^\l \big) \| v \|_{L_T^{\infty} L_x^{2(k-\l)}}^{k-\l} \\
&\les T^{\theta'} \big( 1+\| (\phi_0,\phi_1) \|_{\H^0}^k + \| v \|_{L_T^{\infty} H_x^1}^k \big),
\end{align*}
where $\theta' = \min (\theta,1) >  0$.
Hence, we have
\begin{align*}
\| \vec \Phi (v) \|_{L_T^{\infty} \H_x^1}
& \le \| (v_0,v_1) \|_{\H^1} + \| H_k ( f+v; \s) \|_{L_T^1 L_x^2} \\
&\le R +C T^{\theta'} \big( 1+ \| (\phi_0,\phi_1) \|_{\H^0}^k + \| v \| _{L_T^{\infty} H_x^1}^k \big).
\end{align*}

\noi
A similar computation yields the difference estimate: 
\begin{align*}
\| \vec \Phi&  (v_1) - \vec \Phi(v_2) \|_{L_T^{\infty} \H_x^1} \\
&\le \| H_k ( f+v_1; \s) - H_k ( f+v_2; \s) \|_{L_T^1 L_x^2} \\
& \le C T^{\theta'} \big( 1 + \| (\phi_0,\phi_1) \|_{\H^0}^{k-1} + \| v_1 \| _{L_T^{\infty} H_x^1}^{k-1}+\| v_2 \|_{L_T^{\infty} H_x^1}^{k-1} \big) \| v_1-v_2 \|_{L_T^{\infty} H_x^1}.
\end{align*}
By taking $\tau$ as
\[
\tau \sim \bigg( \frac{\min (1,R)}{1+ \| (\phi_0,\phi_1) \|_{\H^0}^k + R^k}\bigg)^\frac{1}{\ta'},
\]

\noi
we see that 
$\vec \Phi$ is a contraction mapping on the ball 
$B_{2R} = \{ \vec v\in C([0,\tau]; H^1(\T^2)) : \|\vec v\|_{L_\tau^{\infty} \H_x^1} \le 2R \}$.
Therefore, we obtain  a unique\footnote{At this point, 
the uniqueness holds only in $B_{2R}$ but by a standard continuity 
argument, we can extend the uniqueness to the entire
$C([0,\tau]; \H^1(\T^2))$.}
local solution $\vec v  = (v, \dt v) \in C([0,\tau]; \H^1(\T^2))$.
\end{proof}

We now present the proof of Proposition~\ref{PROP:easy}.

\begin{proof}[Proof of Proposition  \ref{PROP:easy}]
As in \cite{CO, BOP2}, it suffices to show the following ``almost'' almost global existence;
given  any $T, \eps>0$, there exists a set $\O_{T,\eps} \subset \O$ such that $P(\O_{T,\eps}^c)<\eps$ and for each $\o \in \O_{T,\eps}$, there exists a solution $\vec v = (v, \dt v)$ to \eqref{WNLW3} on $[-T,T]$.

Let $z(t) = S(t) (\phi_0^\o,\phi_1^\o)$.
Given $T, \eps >0$,  we set 
\begin{align*}
\O_{T,\eps} &=\big \{ \o \in \O : \| z \|_{L_{T, x}^{\infty} } + \| \jb{\nb}^s \wt z \|_{L_{T, x}^{k+1}} \le M \big\}, 
\end{align*}

\noi
where $M$ is given by 
\begin{align}
M &=M(T,\eps,\| (\phi_0,\phi_1) \|_{\H^s}) \sim \jb{T} \| (\phi_0,\phi_1) \|_{\H^s} 
\bigg( \log \frac{\jb{T}}{\eps} \bigg)^{\frac 12} 
\label{C3}
\end{align}

\noi
and $\wt z$ is defined by 
\[ \wt z(t) = - \sin (t\jb{\nb}) \phi_0^\o+ \frac{\cos  (t\jb{\nb}) }{\jb{\nb}}\phi_1^\o.\]

\noi
Note that $\wt z$ also satisfies Lemma \ref{LEM:PStr} and that 
\begin{align}
\dt z = \jb{\nb} \wt z.
\label{C2}
\end{align}

\noi
Then, it follows from Lemma~\ref{LEM:PStr} that
\[
P(\O_{T,\eps}^c) < \eps.
\]

\noi
We point out that the condition $s > 0$ is needed to apply Lemma \ref{LEM:PStr}\,(ii).

As in \cite{BT3, OP}, we use the energy $E(\vec v) = H(v, \dt v)$, 
where $H$ is as in \eqref{Hamil}.
Using the energy $E(\vec v)$,  we show that there exists $R=R(T,\eps, \| (\phi_0,\phi_1) \|_{\H^s})>0$ such that
\begin{align} \label{enebd}
\| (v,\dt v) \|_{L_T^{\infty} \H_x^1} \le R
\end{align}
for any  $\o \in \O_{T,\eps}$.

For now, let us assume \eqref{enebd} and conclude ``almost'' almost sure global existence.
Given $\tau>0$, we write
\[
[-T,T] = \bigcup_{j=-[T/\tau]-1}^{[T\tau]} [ j \tau, (j+1)\tau] \cap [-T,T].
\]
By making $\tau = \tau (M) 
= \tau (T,\eps,\| (\phi_0,\phi_1) \|_{\H^s}) >0$ small, we have
\[
\| z\|_{L^k([ j \tau, (j+1)\tau]; L^{\infty}(\T^2))}
\le \tau^{\frac{1}{k}} M
\le \tau^{\frac{1}{2k}}
\]
for $\o \in \O_{T,\eps}$.
By iteratively applying Lemma \ref{LEM:LWPreg} and Remark \ref{REM:LWPreg}, 
 we can construct a solution $\vec v$ to \eqref{WNLW_PAK} (with $r_0 = r_1 = 1$)
  on $[ j \tau, (j+1)\tau]$, $j=- \big[ \frac{T}{\tau} \big]-1, \dots, \big[ \frac{T}{\tau} \big]$.
This proves  the ``almost'' almost sure global existence.

\medskip

It remains to prove \eqref{enebd}.
We first consider the $k = 3$ case.
In this case, 
it follows from \eqref{WNLW3}, \eqref{H1a}, \eqref{u5}, and H\"older's and Young's inequalities
that 
\begin{align}
E(\vec v(t)) 
&= \int_0^t \int_{\T^2} \dt v \cdot (\dt^2 v +(1- \Dl) v + v^3) dx dt'\notag \\
& 
= \int_0^t \int_{\T^2} \dt v \cdot (-H_3(z+v; \s) + v^3) dx dt' \notag\\
&= \int_0^t\int_{\T^2} \dt v \cdot ( - 3 zv^2 -3 (z^2-\s)v-z^3+3\s z ) dx dt'\label{dEv}\\
&\les \int_0^t \| \dt v (t') \|_{L_x^2} \Big\{ \| z(t') \|_{L_x^{\infty}} \| v(t') \|_{L_x^4}^2 \notag\\
& \quad + (\| z(t') \|_{L_x^8}^2 + \| (\phi_0, \phi_1) \|_{\H^0}^2) \| v(t') \|_{L_x^4} \notag\\ 
& \quad + \| z(t') \|_{L_x^6}^3 + \| (\phi_0, \phi_1) \|_{\H^0}^2 \| z(t') \|_{L_x^2} \Big\} dt'\notag\\
&\les (1 + \| z\|_{L^\infty_T L^\infty_x}) \int_0^t E(\vec v(t')) dt'
+ \|z\|_{L^8_{T, x}}^8 + \| (\phi_0, \phi_1) \|_{\H^0}^8)  \notag\\
&\quad +  \|z\|_{L^6_{T, x}}^6 +  \| (\phi_0, \phi_1) \|_{\H^0}^4\|z\|_{L^2_{T, x}} \notag\\
&\les (1 + M) \int_0^t E(\vec v(t')) dt'
+ C(T, M,  \| (\phi_0, \phi_1) \|_{\H^0})\notag 
\end{align}

\noi
for $\o \in \O_{T,\eps}$.
Hence, from  Gronwall's inequality, we obtain \eqref{enebd} for $k=3$ and $s>0$.

Next, we consider the case $k \ge 5$.
From  \eqref{H2}, we have
\begin{align} 
\begin{split}
\dt H_\l (z(x,t); \s(t))
& = \l H_{\l-1} (z(x,t); \s(t)) \dt z(x,t) \\
& \quad -\ind_{\l \geq 2} \cdot  \frac{\l (\l-1)}{2} H_{\l-2}  (z(x,t);\s(t)) \dt \s(t).
\end{split}
\label{C1}
\end{align}

\noi
Then, from \eqref{WNLW3} and 
integration by parts with \eqref{C1}, we have 
\begin{align}
 E(\vec v(t))
&= \int_0^t \int_{\T^2} \dt v \cdot (\dt^2 v +(1- \Dl) v + v^k) dx dt'\notag \\
& = \int_0^t  \int_{\T^2} \dt v \cdot (-H_k (z+v; \s) + v^k) dx dt'\notag\\
&= - \sum_{\l=1}^k \begin{pmatrix} k \\ \l \end{pmatrix} 
\int_0^t \int_{\T^2} \dt v \cdot H_\l(z;\s) v^{k-\l} dx dt'\notag\\
&= - \sum_{\l=1}^k \begin{pmatrix} k \\ \l \end{pmatrix} \frac{1}{k-\l+1} 
\bigg\{ \int_{\T^2} H_\l(z;\s) v^{k-\l+1} dx\bigg|_0^t  \notag\\
&\quad - \l \int_0^t \int_{\T^2} H_{\l-1}(z;\s) \dt z \cdot v^{k-\l+1} dx dt' \notag\\
&\quad + \ind_{\l \geq 2}\cdot  \frac{\l (\l-1)}{2} \int_0^t \int_{\T^2} H_{\l-2}(z;\s) \dt \s \cdot v^{k-\l+1} dx dt'
\bigg\}. \label{derEv}
\end{align}

From Young's inequality and \eqref{u5}, we have 
\begin{align}
\begin{split}
\bigg| \int_{\T^2} H_\l & (z;\s) v^{k-\l+1} (t) dx \bigg|
\le C(\dl) \| H_\l(z(t);\s (t)) \|_{L_x^{\frac{k+1}{\l}}}^{\frac{k+1}{\l}} + \dl \| v(t) \|_{L_x^{k+1}}^{k+1} \\
&\le C(\dl) \big( \| z(t) \|_{L_x^{k+1}}^\l + \| (\phi_0,\phi_1) \|_{\H^0}^\frac \l2  \big)^{\frac{k+1}{\l}} + \dl E(\vec v(t)) \\
&\le C(\dl) \big( M^{k+1} + \| (\phi_0,\phi_1) \|_{\H^0}^{\frac{k+1}{2}} \big) + \dl E(\vec v(t)) 
\end{split}
 \label{derEv1}
\end{align}
for $\o \in \O_{T,\eps}$ and $1\le \l \le k$,
where $\dl>0$ is a small constant to be chosen later.
From \eqref{C2} and  Young's and H\"older's inequalities with \eqref{C3}, 
we have
\begin{align}
\begin{split}
\bigg| & \int_0^t \int_{\T^2} H_{\l-1}(z;\s) \dt z \cdot v^{k-\l+1} dx dt'\bigg|
= \bigg| \int_0^t \int_{\T^2} H_{\l-1}(z;\s) \jb{\nb} \wt z \cdot v^{k-\l+1} dx dt'\bigg| \\
&\les \int_0^t  \| H_{\l-1}(z(t');\s (t')) \jb{\nb} \wt z(t')\|_{L_x^{\frac{k+1}{\l}}}^{\frac{k+1}{\l}} + \| v(t') \|_{L_x^{k+1}}^{k+1} dt' \\
&\les \Big( \big( \| z(t) \|_{L_x^{k+1}}^{\l-1} + \| (\phi_0,\phi_1) \|_{\H^0}^\frac{\l-1}2  \big) \| \jb{\nb}\wt z(t) \|_{L_x^{k+1}} \Big)^{\frac{k+1}{\l}} + \int_0^t E(\vec v(t')) dt' \\
&\les C(T, M,  \| (\phi_0,\phi_1) \|_{\H^0})   + \int_0^t E(\vec v(t'))dt' 
\end{split}
\label{derEv2}
\end{align}
for $\o \in \O_{T,\eps}$ and $1\le \l \le k$.
Lastly, from  Young's inequality and \eqref{u5}, we have
\begin{align}
\begin{split}
&\bigg| \int_0^t \int_{\T^2} H_{\l-2}(z;\s) \dt \s \cdot v^{k-\l+1} dx dt' \bigg| \\
&\les \| (\phi_0,\phi_1) \|_{\H^{\frac 12}}^2 
\int_0^t \| H_{\l-2}(z(t');\s (t')) \|_{L_x^{\frac{k+1}{\l}}}^{\frac{k+1}{\l}} + \| v(t') \|_{L_x^{k+1}}^{k+1}dt' \\
&\les \| (\phi_0,\phi_1) \|_{\H^{\frac 12}}^2
\int_0^t  \big( \| z(t') \|_{L_x^{k+1}}^{\l-2} + \| (\phi_0,\phi_1)\|_{\H^0}^\frac{\l-2}2 \big)^{\frac{k+1}{\l}} + E(\vec v(t')) dt' \\
&\les C(T, M,  \| (\phi_0,\phi_1) \|_{\H^\frac 12} )+ \| (\phi_0,\phi_1) \|_{\H^{\frac 12}}^2\int_0^t  E(\vec v(t')) dt'
\end{split}
\label{derEv3}
\end{align}
for $\o \in \O_{T,\eps}$ and $2\le \l \le k$.
Hence,  by taking $\dl >0$ small, it follows from \eqref{derEv}, \eqref{derEv1}, \eqref{derEv2}, and \eqref{derEv3} that
\begin{align*}
E(\vec v(t)) &= \int_0^t \frac{d}{dt} E(\vec v(t')) dt' \\
&\le \frac 12 E(\vec v(t)) + C(T, M,  \| (\phi_0,\phi_1) \|_{\H^{\frac 12}}) 
 + \|(\phi_0,\phi_1) \|_{\H^{\frac 12}}^2  \int_0^t E(v(t')) dt',
\end{align*}
which implies that
\[
E(\vec v(t))
\le  C(T, M,  \| (\phi_0,\phi_1) \|_{\H^{\frac 12}}) 
 + \|(\phi_0,\phi_1) \|_{\H^{\frac 12}}^2  \int_0^t E(v(t')) dt',
\]
for $\o \in \O_{T,\eps}$.
Therefore, from Gronwall's inequality, we obtain \eqref{enebd} for $k \ge 5$ and $s \ge 1$.
This concludes the proof of Proposition \ref{PROP:easy}.
\end{proof}

\begin{remark} \label{REM:lwp}\rm
Noting that Lemma \ref{LEM:PStr}\.(i) holds
for $s \geq 0$, 
we see that we can handle all the terms in \eqref{dEv}
for $s = 0$,  except for $\int_0^t \int_{\T^2} \dt v \cdot z v^2 dxdt'$.
As for this term, we can use Yudovich's argument as in  \cite{BT3}
and hence Proposition~\ref{PROP:easy} with $k=3$ indeed holds for $s=0$.

For $k \ge 5$, we used the assumption $s \ge 1$ to control $\| \jb{\nb}\wt  z \|_{L_{T, x}^{k+1}}$ in \eqref{derEv2}.
By proceeding as in \cite{OP} via the Littlewood-Paley decomposition, 
we may extend the result to some $s < 1$.
However, since the main purpose of Proposition \ref{PROP:easy} is to give a remark on the almost sure global existence with smooth random initial data, we do not pursue this issue further.

\end{remark}
\section{Unique limit of smooth solutions with mollified data}
\label{SEC:uniq}

In this section, we present the proof of Theorem \ref{THM:uniq}.
We first prove the almost sure convergence of the Wick powers for the Gaussian initial data
\eqref{Gauss1}  in Subsection \ref{SUBSEC:sto1}.
We then show convergence in probability of the Wick powers for smooth Gaussian initial data in Subsection \ref{SUBSEC:sto2}.
Moreover, we prove that the limit is independent of mollification kernels.
In Subsection \ref{SUBSEC:LWPdpNLW}, we go over local well-posedness of the perturbed NLW
with deterministic perturbations (Lemma \ref{LEM:LWPx}).
Finally, in Subsection \ref{SUBSEC:4.3}, we iteratively apply Lemma~\ref{LEM:LWPx} for short time intervals to prove Theorem~\ref{THM:uniq}.

\subsection{Convergence of  the Wick powers}
\label{SUBSEC:sto1}

In this subsection, we present a proof of Proposition~\ref{PROP:Z1}.
We first estimate the variance of the Fourier coefficients of 
the truncated Wick powers $:\! z_{N}^\l (t) \!:$ defined in \eqref{Wick1}.

\begin{lemma} \label{LEM:Z1p}
Let $\l \in \N\cup \{0\}$.
For any $\eps>0$, $\g >0$, $n \in \Z^2$, $t \in \R$, and $M \ge N \ge 1$,
we have
\begin{align}
\E \big[ |\jb{\ \! :\! z_{N}^\l (t) \!: , e_n}_{L^2} |^2 \big]
&\les
\jb{n}^{-2+\eps}, \label{PZ1} \\
\E \big[ |\jb{\ \! :\! z_{N}^\l (t) \!: - \, :\! z_{M}^\l(t)\!: , e_n}_{L^2} |^2 \big]
&\les
N^{-\g} \jb{n}^{-2+\eps+\g}, \label{PZ2}
\end{align}
where $e_n(x) = e^{in \cdot x}$.
In addition,
for any $\eps>0$, $\g>0$, $\theta \in (0,1)$, $n \in \Z^2$, $t \in \R$, $h \in [-1,1]$, and $M \ge N \ge 1$,
we have
\begin{align}
\E \big[ |\jb{\dl_h \! :\! z_{N}^\l (t) \!: , e_n}_{L^2} |^2 \big]
&\les
\jb{n}^{-2+\eps+\theta} |h|^{\theta}, \label{PZ3} \\
\E \big[ |\jb{\dl_h \! :\! z_N^\l (t) \!: - \, \dl_h :\! z_M^\l(t)\!: , e_n}_{L^2} |^2 \big]
&\les
N^{-\g} \jb{n}^{-2+\eps+\g+\theta} |h|^{\theta}, \label{PZ4}
\end{align}
where $\dl_h$ is as in \eqref{diff1}.
\end{lemma}

Once we prove Lemma~\ref{LEM:Z1p}, by choosing  $\g$ and $\theta$ 
sufficiently small such that $\g+\theta<\eps$, Proposition~\ref{PROP:Z1} follows from Corollary~\ref{COR:reg1}.
\\

For the proof of Lemma \ref{LEM:Z1p},
we employ the argument used in the proofs of \cite[Lemma 2.5]{OTh1}
and  \cite[Proposition 2.3]{OTh2}.
Let us first introduce some notations.
For {\it fixed} $x \in \T^2$, we define
\begin{align} \label{etaN}
\eta_{N} (x) (\cdot) \stackrel{\text{def}}{=} \frac{1}{\s_{N}^\frac{1}{2}}
\sum_{|n| \le N} \frac{\cj{e_n(x)}}{\jb{n}}e_n(\cdot),
\end{align}
	
\noi
where $\s_N$ is as in \eqref{G6}.
Note that $\eta_{N}(x)(\cdot)$ is real-valued
with 
$ \| \eta_{N}(x)\|_{L^2(\T^2)} = 1$
for any  $x \in \T^2$ and $N \in \N$.
Moreover, we have 
\begin{align}
\jb{\eta_{N}(x), \eta_{M}(y)}_{L^2}
= \frac{1}{\s_{N}^\frac{1}{2}\s_{M}^\frac{1}{2}} \sum_{|n| \le N} \frac{1}{\jb{n}^2} e_n(y-x)
= \frac{1}{\s_{N}^\frac{1}{2}\s_{M}^\frac{1}{2}} \sum_{|n| \le N} \frac{1}{\jb{n}^2} e_n(x-y)
\label{W4}
\end{align}

\noi
for any  $x, y\in \T^2$
and  $M \ge N \ge 1$.

\begin{proof}[Proof of Lemma \ref{LEM:Z1p}]
We only consider \eqref{PZ2} and \eqref{PZ4}, since \eqref{PZ1} and \eqref{PZ3} follow 
from an analogous (but simpler) argument.

By \eqref{PStr2} and \eqref{etaN} (see also \eqref{lin2}), we note that
\[
z_{N} (x,t) = \s_{N}^{\frac{1}{2}} \frac{z_{N}(x,t)}{\s_{N}^{\frac{1}{2}}}
 = \s_{N}^{\frac{1}{2}} W_{\eta_{N}(x)}^t.
\]
Then, from \eqref{Wick1}, we have
\begin{align} \label{W4z}
\! :\! z_{N}^\l (t) \!: \, = H_\l (z_{N}(x,t); \s_{N}) = \s_N^{\frac{\l}{2}} H_\l (W_{\eta_N(x)}^t).
\end{align}
Given $n \in \Z^2$, define $\G_\l (n)$ by
\begin{align} \notag
\G_\l(n) \stackrel{\text{def}}{=} \{ (n_1, \dots, n_\l) \in (\Z^2)^\l : n_1+ \dots + n_\l =n \}.
\end{align}

\noi
For $(n_1, \dots, n_\l) \in \G_\l(n)$, we have $\max_j |n_j| \ges |n|$.
It follows from \eqref{W4z}, Lemma \ref{LEM:W1}, and~\eqref{W4} that 
\begin{align}
& \E \big[ |\jb{\ \! :\! z_{N}^\l (t) \!: - \, :\! z_{M}^\l(t)\!: , e_n}_{L^2} |^2 \big]
\notag\\
&= \int_{\T^2_x \times \T^2_y}\cj{e_n (x)} e_n(y)\notag  \\
& \quad 
 \int_\O \Big[
\s_N^\l H_\l\big(W^t_{\eta_N(x)} \big) \cj{H_\l\big(W^t_{\eta_{N}(y)} \big)}
+\s_{M}^{\l} H_\l\big(W^t_{\eta_{M}(x)} \big) \cj{H_\l\big(W^t_{\eta_{M}(y)} \big)} \notag \\
& \quad - \s_N^{\frac{\l}{2}} \s_{M}^{\frac{\l}{2}} \Big\{ H_\l\big(W^t_{\eta_N(x)} \big) \cj{H_\l\big(W^t_{\eta_{M}(y)} \big)} + H_\l\big(W^t_{\eta_{M}(x)} \big) \cj{H_\l\big(W^t_{\eta_{N}(y)} \big)} \Big\} \Big] dP dxdy \notag\\
& = \l ! 
\Bigg\{ \sum_{\substack{\G_\l(n) \\ |n_j|\le M} }\prod_{j= 1}^\l \frac{1}{\jb{n_j}^2}
- \sum_{\substack{\G_\l(n) \\ |n_j|\le N}} \prod_{j= 1}^\l \frac{1}{\jb{n_j}^2} \Bigg\} \notag\\
&\les N^{-\g} \jb{n}^{-2+\eps+\g}.\label{PStr6a} 
\end{align}

\noi
for any $M \ge N \ge 1$.  This prove \eqref{PZ2}.

Next, we consider \eqref{PZ4}.
From \eqref{W4z},  Lemmas \ref{LEM:W1},  and \ref{LEM:W2} with \eqref{W4}, we have
\begin{align}
&\E \big[ |\jb{\dl_h \! :\! z_N^\l (t) \!: - \, \dl_h :\! z_M^\l(t)\!: , e_n}_{L^2} |^2 \big] \notag\\
&= \int_{\T^2_x \times \T^2_y}\cj{e_n (x)} e_n(y) \notag\\
& \quad 
 \int_\O \Big[
\s_{N}^{\l} \Big\{ H_\l\big(W^{t+h}_{\eta_{N}(x)} \big) \cj{H_\l\big(W^{t+h}_{\eta_{N}(y)} \big)} - H_\l\big(W^{t+h}_{\eta_{N}(x)} \big) \cj{H_\l\big(W^t_{\eta_{N}(y)} \big)} \notag\\
& \hspace*{50pt} - H_\l\big(W^t_{\eta_{N}(x)} \big) \cj{H_\l\big(W^{t+h}_{\eta_{N}(y)} \big)} + H_\l\big(W^t_{\eta_{N}(x)} \big) \cj{H_\l\big(W^t_{\eta_{N}(y)} \big)} \Big\} \notag\\
& \quad +\s_{M}^{\l} \Big\{ H_\l\big(W^{t+h}_{\eta_{M}(x)} \big) \cj{H_\l\big(W^{t+h}_{\eta_{M}(y)} \big)} - H_\l\big(W^{t+h}_{\eta_{M}(x)} \big) \cj{H_\l\big(W^t_{\eta_{M}(y)} \big)} \notag\\
& \hspace*{50pt} - H_\l\big(W^t_{\eta_{M}(x)} \big) \cj{H_\l\big(W^{t+h}_{\eta_{M}(y)} \big)} + H_\l\big(W^t_{\eta_{M}(x)} \big) \cj{H_\l\big(W^t_{\eta_{M}(y)} \big)} \Big\} \notag\\
& \quad - \s_{N}^{\frac{\l}{2}} \s_{M}^{\frac{\l}{2}} \Big\{ H_\l\big(W^{t+h}_{\eta_{N}(x)} \big) \cj{H_\l\big(W^{t+h}_{\eta_{M}(y)} \big)} - H_\l\big(W^{t+h}_{\eta_{N}(x)} \big) \cj{H_\l\big(W^t_{\eta_{M}(y)} \big)} \notag\\
& \hspace*{50pt} - H_\l\big(W^{t}_{\eta_{N}(x)} \big) \cj{H_\l\big(W^{t+h}_{\eta_{M}(y)} \big)} + H_\l\big(W^{t}_{\eta_{N}(x)} \big) \cj{H_\l\big(W^t_{\eta_{M}(y)} \big)} \notag \\
& \hspace*{50pt} + H_\l\big(W^{t+h}_{\eta_{M}(x)} \big) \cj{H_\l\big(W^{t+h}_{\eta_{N}(y)} \big)} - H_\l\big(W^{t+h}_{\eta_{M}(x)} \big) \cj{H_\l\big(W^t_{\eta_{N}(y)} \big)} \notag \\
& \hspace*{50pt} - H_\l\big(W^{t}_{\eta_{M}(x)} \big) \cj{H_\l\big(W^{t+h}_{\eta_{N}(y)} \big)} + H_\l\big(W^{t}_{\eta_{M}(x)} \big) \cj{H_\l\big(W^t_{\eta_{N}(y)} \big)} \Big\} \Big] dP dxdy \notag\\
& = 2\l ! \sum_{\substack{\G_\l(n) \\ N<  \max_j |n_j| \le M}}
\bigg\{ \prod_{j= 1}^\l \frac{1}{\jb{n_j}^2}
-  \prod_{j= 1}^\l \frac{ \cos (h \jb{n_j})}{\jb{n_j}^2}
\bigg\}.\label{PStr6b}
\end{align}

\noi
By writing the last expression in a telescoping sum
and applying the mean-value theorem, we have
\begin{align*}
\text{RHS of } \eqref{PStr6b}
& \les
\sum_{\substack{\G_\l(n) \\ N< \max_j |n_j|\le M}}
\sum_{k = 1}^\l  |h|^\ta \jb{n_k}^\ta  \prod_{j= 1}^\l \frac 1{\jb{n_j}^2} 
\notag\\
&\les N^{-\g} \jb{n}^{-2+\eps+\g+\theta} |h|^{\theta}. 
\end{align*}

\noi
This proves \eqref{PZ4}.
\end{proof}

\subsection{Uniqueness of the Wick powers}
\label{SUBSEC:sto2}

In this subsection, we study  the Wick powers for smooth Gaussian initial data
$(u_{0,\dl}^\o, u_{1,\dl}^\o)$  in \eqref{uniq1}
and show that 
 they converge in probability
to the Wick powers
$:\! z^\l \!:$ constructed in the previous subsection, 
which in particular implies that limit is  independent of mollification kernels.
In order to signify  the dependence on a mollification kernel $\rho,$ we write
\begin{align*}
z_{\rho,\dl} &= S(t) (u_{0,\dl}^\o, u_{1,\dl}^\o), \\ 
\s_{\rho,\dl} &=
\text{Var}(z_{\rho,\dl}(x, t))
= \E [  z_{\rho,\dl}^2(x, t)]
= \sum_{n \in \Z^2}\frac{|\ft \rho (\dl n)|^2}{\jb{n}^2}, \notag\\
:\! z_{\rho,\dl}^\l(x, t) \!:  \,
&= H_\l(z_{\rho,\dl}(x, t); \s_{\rho,\dl}), \notag
\end{align*}
where $(u_{0,\dl}^\o, u_{1,\dl}^\o)$ is defined in \eqref{uniq1}.
Our main goal in this subsection is to prove
the following proposition.

\begin{proposition}\label{PROP:Z1rho}
Let $\l \in \N\cup \{0\}$.
Then, 
for any  $T > 0$  and $\eps>  0$, 
the mollified Wick powers $:\! z_{\rho,\dl}^\l \!:$ converges in probability
to $:\! z^\l \!:$
in $C([-T, T]; W^{-\eps, \infty}(\T^2))$ as $\dl \to 0$,
where $:\! z^\l \!:$ is defined in \eqref{Wick4}.
\end{proposition}

We point out that Proposition \ref{PROP:Z1rho}
establishes convergence in probability, {\it not} almost sure convergence.
This is due to the fact that we take a limit  along a continuous
parameter $\dl \to 0$.
Indeed, in the second part of the proof of 
Proposition \ref{PROP:Z1rho}, 
by restricting our attention to a discrete sequence tending to $0$ (i.e.~$\dl = \frac 1N$, $N \in \N$), 
we show that 
the sequence $ \{ \,:\! z_{\rho,\frac 1N}^\l\!:\,\}_{N \in \N}$ converges almost surely.\footnote{It seems possible to adapt the argument in the proof of Proposition 2.3 in \cite{TzV}
 to prove almost sure convergence of $:\! z_{\rho,\dl}^\l \!:$
along a continuous parameter $\dl \to 0$. We, however, do not pursue this issue here.}

As in the proof of Proposition \ref{PROP:Z1}, we first estimate the variance of the Fourier coefficients of 
the mollified Wick powers $:\! z_{\rho,\dl}^\l (t) \!:$.

\begin{lemma} \label{LEM:Z1pr}
Let $\l \in \N\cup \{0\}$.
For any $\eps>0$, $\g \in (0,1)$, $n \in \Z^2$, $t \in \R$, and $\dl, \dl' \in (0,1]$
we have
\begin{align}
\E \big[ |\jb{\ \! :\! z_{\rho,\dl}^\l (t) \!: , e_n}_{L^2} |^2 \big]
&\les
\jb{n}^{-2+\eps}, \label{PZ2r0}\\
\E \big[ |\jb{\ \! :\! z_{\rho,\dl}^\l (t) \!: - \, :\! z_{\rho,\dl'}^\l(t)\!: , e_n}_{L^2} |^2 \big]
&\les
|\dl-\dl'|^{\g} \jb{n}^{-2+\eps+\g}. \label{PZ2r}
\end{align}

\noi
In addition, 
for any $\eps>0$, $\g,\theta \in (0,1)$, $n \in \Z^2$, $t \in \R$, $h \in [-1,1]$, and $\dl,\dl' \in (0,1]$,
we have
\begin{align}
\E \big[ |\jb{\dl_h \! :\! z_{\rho,\dl}^\l (t) \!: , e_n}_{L^2} |^2 \big]
&\les
\jb{n}^{-2+\eps+\theta} |h|^{\theta}, \notag \\
\E \big[ |\jb{\dl_h \! :\! z_{\rho,\dl}^\l (t) \!: - \, \dl_h :\! z_{\rho,\dl'}^\l(t)\!: , e_n}_{L^2} |^2 \big]
&\les
|\dl-\dl'|^{\g} \jb{n}^{-2+\eps+\g+\theta} |h|^{\theta}. \label{PZ4r}
\end{align}
\end{lemma}

\begin{proof}
Since these estimates follow from a slight modification of the proof of Lemma \ref{LEM:Z1p}, we give a brief explanation of the proof of \eqref{PZ2r} and \eqref{PZ4r}.
Proceeding as in  \eqref{PStr6a}, we have
\begin{align}
&\E \big[ |\jb{\ \! :\! z_{\rho,\dl}^\l (t) \!: - \, :\! z_{\rho,\dl'}^\l(t)\!: , e_n}_{L^2} |^2 \big] \notag\\
& = \l ! \sum_{\G_\l(n)}
\Bigg\{ \prod_{j= 1}^\l \frac{|\ft \rho (\dl n_j)|^2}{\jb{n_j}^2}
+ \prod_{j= 1}^\l \frac{|\ft \rho (\dl' n_j)|^2}{\jb{n_j}^2}\notag \\
& \hphantom{XXX}- \prod_{j= 1}^\l \frac{\ft \rho (\dl n_j) \cj{\ft \rho (\dl' n_j)}}{\jb{n_j}^2}
- \prod_{j= 1}^\l \frac{\ft \rho (\dl' n_j) \cj{\ft \rho (\dl n_j)}}{\jb{n_j}^2} \Bigg\} \notag\\
& = \l ! \sum_{\G_\l(n)}
 \prod_{j= 1}^\l \frac{|\ft \rho (\dl n_j)- \ft \rho (\dl' n_j)|^2}{\jb{n_j}^2}.
\label{PZ2ra}
\end{align}
Since $\rho \in L^1(\T^2)$,  it follows from the mean value theorem that 
\begin{align}
|\ft \rho (\dl n) -\ft \rho (\dl' n)|
\le \int_{\T^2} \big| 1-e^{i (\dl-\dl') n \cdot x} \big| |\rho(x)| dx
\les \min(1,|\dl-\dl'| |n|).
\label{rhomv}
\end{align}
Hence, \eqref{PZ2r} follows from \eqref{PZ2ra} and \eqref{rhomv}.

Similarly, proceeding as in  \eqref{PStr6b} with \eqref{rhomv}
and the mean value theorem,  we have
\begin{align}
&\E \big[ |\jb{\dl_h \! :\! z_{\rho,\dl}^\l (t) \!: - \, \dl_h :\! z_{\rho,\dl'}^\l(t)\!: , e_n}_{L^2} |^2 \big] \notag\\
& = 2\l ! \sum_{\G_\l(n)}
\Bigg\{ \prod_{j= 1}^\l \frac{|\ft \rho (\dl n_j)|^2}{\jb{n_j}^2} - \prod_{j=1}^\l \frac{|\ft \rho (\dl n_j)|^2 \cos (h \jb{n_j})}{\jb{n_j}^2} \notag \\
& \hspace*{50pt} + \prod_{j= 1}^\l \frac{|\ft \rho (\dl' n_j)|^2}{\jb{n_j}^2} - \prod_{j=1}^\l \frac{|\ft \rho (\dl' n_j)|^2 \cos (h \jb{n_j})}{\jb{n_j}^2} \notag\\
& \hspace*{50pt} - \bigg( \prod_{j= 1}^\l \frac{\ft \rho (\dl n_j) \cj{\ft \rho (\dl' n_j)}}{\jb{n_j}^2} - \prod_{j= 1}^\l \frac{\ft \rho (\dl n_j) \cj{\ft \rho (\dl' n_j)} \cos (h \jb{n_j})}{\jb{n_j}^2} \notag\\
& \hspace*{70pt} + \prod_{j= 1}^\l \frac{\ft \rho (\dl' n_j) \cj{\ft \rho (\dl n_j)}}{\jb{n_j}^2} - \prod_{j= 1}^\l \frac{\ft \rho (\dl' n_j) \cj{\ft \rho (\dl n_j)} \cos (h \jb{n_j})}{\jb{n_j}^2}  \bigg) \Bigg\}\notag\\
& = \l ! \sum_{\G_\l(n)}\Bigg\{
 \prod_{j= 1}^\l \frac{|\ft \rho (\dl n_j)- \ft \rho (\dl' n_j)|^2}{\jb{n_j}^2}
 -  \prod_{j= 1}^\l \frac{|\ft \rho (\dl n_j)- \ft \rho (\dl' n_j)|^2 \cos( h \jb{n_j})}{\jb{n_j}^2}
\Bigg\}\notag\\
&
\les |\dl - \dl'|^{\g} \jb{n}^{-2+\eps+\g+\theta} |h|^{\theta}, 
 \label{PStre6br}
\end{align}

\noi
yielding 
 \eqref{PZ4r}.
\end{proof}

We are now ready to prove Proposition \ref{PROP:Z1rho}.

\begin{proof}[Proof of Proposition \ref{PROP:Z1rho}]
Fix small  $\g, \theta > 0$ such that $\g+\theta<\eps$.
Fix $t \in \R$.
Then, it follows from  \eqref{PZ2r0}, \eqref{PZ2r},  and Lemma \ref{LEM:hyp} (see also Remark \ref{REM:cor})
that, as $\dl \to 0$, 
$ :\! z_{\rho,\dl}^\l (t) \!:$ converges to some limit
$ :\! z_\rho^\l (t) \!: $
 in $L^p(\O; W^{-\eps,\infty}(\T^2))$ for any finite $p \ge 1$.

Let $\{ \dl_j \}_{j \in \N}$ be a sequence satisfying $\dl_j \to 0$ as $j \to \infty$.
There exists a subsequence $\{ \dl_{j(m)} \}_{m \in \N} \subset \{ \dl_j \}_{j \in \N}$ such that $\dl_{j(m)} < m^{-1}$ for $m \in \N$.
It follows from Corollary~\ref{COR:reg1} with Lemma \ref{LEM:Z1pr} that
the subsequence $:\! z_{\rho,\dl_{j(m)}}^\l \!:$ converges almost surely 
(and hence in measure) to $:\! z_\rho^\l \!:$ in $C([-T,T]; W^{-\eps,\infty}(\T^2))$,  as $m \to \infty$.
Since the limit $\! :\! z_\rho^\l \!:$ is independent of the choice of a sequence
$\{ \dl_j \}_{j \in \N}$,  we  deduce that $:\! z_{\rho,\dl}^\l \!:$ converges in probability to $:\! z_\rho^\l \!:$ in $C([-T,T]; W^{-\eps,\infty}(\T^2))$.

Next, we prove that the limit is independent of mollification kernels.
Since $\rho \in L^1(\T^2)$ and $\ft \rho (0)=1$, 
it follows from the mean value theorem that 
\begin{align*}
\Big| 1-\ft \rho \Big( \frac nN \Big) \Big|
\le \int_{\T^2} \big| 1-e^{- i \frac nN \cdot x} \big| |\rho(x)| dx
\les \min \Big( 1, \frac{|n|}{N} \Big).
\end{align*}

\noi
Given $h \in [-1,1]$, proceeding as in 
 \eqref{PStre6br}, we have 
\begin{align}
& \E \big[ | \jb{ \dl_h \! :\! z_N^\l (t) \!: - \, \dl_h :\! z_{\rho,\frac 1N}^\l(t)\!: , e_n}_{L^2} |^2 \big] \notag\\
& = 2\l ! \sum_{\G_\l(n) }
\Bigg\{ \prod_{j= 1}^\l \frac{\ind_{|n_j|\le N}}{\jb{n_j}^2}
 - \prod_{j=1}^\l \frac{\ind_{|n_j|\le N}\cos (h \jb{n_j})}{\jb{n_j}^2} \notag \\
 & \hphantom{XXXXX}
 + \prod_{j= 1}^\l \frac{|\ft \rho (\frac{n_j}{N})|^2}{\jb{n_j}^2} - \prod_{j=1}^\l \frac{|\ft \rho (\frac{n_j}{N})|^2 \cos (h \jb{n_j})}{\jb{n_j}^2}\notag \\
 & \hphantom{XXXXX}
- \bigg(\prod_{j= 1}^\l \frac{\ind_{|n_j|\le N}\cj{\ft \rho (\frac{n_j}{N})}}{\jb{n_j}^2} - \prod_{j= 1}^\l \frac{\ind_{|n_j|\le N}\cj{\ft \rho (\frac{n_j}{N})} \cos (h \jb{n_j})}{\jb{n_j}^2}\notag\\
 & \hphantom{XXXXXXX}
+ \prod_{j= 1}^\l \frac{\ind_{|n_j|\le N}\ft \rho (\frac{n_j}{N})}{\jb{n_j}^2} - \prod_{j= 1}^\l
 \frac{\ind_{|n_j|\le N} \ft \rho (\frac{n_j}{N}) \cos (h \jb{n_j})}{\jb{n_j}^2}  \Bigg) \Bigg\} \notag\\
&=  2\l ! \sum_{\substack{\G_\l(n)}}
\Bigg\{ 
\prod_{j= 1}^\l \frac{|\ind_{|n_j|\le N}- \ft \rho (\frac{n_j}{N})|^2}{\jb{n_j}^2} 
- \prod_{j= 1}^\l
 \frac{|\ind_{|n_j|\le N} - \ft \rho (\frac{n_j}{N})|^2 \cos (h \jb{n_j})}{\jb{n_j}^2}  
 \Bigg\}\label{PStr7}.
%
\end{align}

\noi
By writing the summand in a telescoping sum and applying the mean value theorem 
(to $1 - \cos (h \jb{n_j})$) and \eqref{PStr7}, we have
\begin{align*}
\text{RHS of } \eqref{PStr7}
& \les
 \sum_{\substack{\G_\l(n)}}
\sum_{k = 1}^\l  |h|^\ta \jb{n_k}^\ta
\prod_{j= 1}^\l \frac{|\ind_{|n_j|\le N}- \ft \rho (\frac{n_j}{N})|^2}{\jb{n_j}^2} \notag\\
& \les
 \sum_{\substack{\G_\l(n)}}
\sum_{k = 1}^\l  |h|^\ta \jb{n_k}^\ta
\prod_{j= 1}^\l \frac{\ind_{|n_j|>  N}}{\jb{n_j}^2} \notag\\
& \quad + 
 \sum_{\substack{\G_\l(n)}}
\sum_{k = 1}^\l  |h|^\ta \jb{n_k}^\ta
\prod_{j= 1}^\l \frac{|\ind_{|n_j|\le N}- \ft \rho (\frac{n_j}{N})|^2}{\jb{n_j}^2} \notag\\
& \les N^{-\g} \jb{n}^{-2+\eps+\g+\theta} |h|^{\theta}.
\end{align*}

\noi
A similar estimate holds for  the difference:
\[ \E \big[ | \jb{   :\! z_N^\l (t) \!: - \, :\! z_{\rho,\frac 1N}^\l(t)\!: , e_n}_{L^2} |^2 \big].
\]

\noi
Therefore, 
from the above computation with Lemma \ref{LEM:Z1p}
and  Proposition \ref{PROP:reg}, we see that, as $N \to \infty$, 
$ :\! z_{\rho,\frac 1N}^\l\!:$ converges almost surely
to  $:\! z^\l\!: $ (constructed in Proposition \ref{PROP:Z1})
in  $C([-T, T]; W^{-\eps, \infty}(\T^2)) $.
Together with the convergence in probability of $\{ :\! z_{\rho,\dl}^\l \!: \}_{\dl \in (0,1]}$ to 
${ :\! z_{\rho}^\l\!:}$, 
we conclude 
 that $:\! z^\l\!: \, =\, :\! z_{\rho}^\l\!:\
$ almost surely.
This completes the proof of Proposition~\ref{PROP:Z1rho}.
\end{proof}

\subsection{Local well-posedness of the  perturbed NLW
with deterministic perturbation}
\label{SUBSEC:LWPdpNLW}

In this subsection, we consider the local well-posedness of the following Cauchy problem:
\begin{align}
\begin{cases}
\dt^2 v  + (1 -  \Dl)  v  +  v^3 + 3f_1 v^2 + 3f_2 v + f_3 = 0\\ 
(v, \dt v) |_{t = 0} = (v_0, v_1), 
\end{cases}
\label{fNLW1}
\end{align}
where $f_1,f_2,f_3$ are given (deterministic) functions.
We define the function space
\begin{align*} 
X^s(I) = C(I;H^s(\T^2)) \cap C^1(I;H^{s-1}(\T^2))
\end{align*}
for $s \in \R$ and an interval $I \subset \R$.
If $I=[-T,T]$, we write $X^s_T = X^s([-T,T])$.

\begin{lemma}\label{LEM:LWPx}
Let $\frac 12 <s<1$.
There exists $\eps = \eps (s)>0$ such that if $f_1,f_2,f_3 \in L_{\rm loc}^{\frac 2\eps} (\R; W^{-\eps, \frac 2\eps}(\T^2))$, then the Cauchy problem \eqref{fNLW1} is locally well-posed in $\H^s(\T^2)$.
More precisely, given $(v_0,v_1) \in \H^s(\T^2)$, there exist $T>0$ 
and  a unique solution $v \in X^s_T$ to~\eqref{fNLW1}, depending continuously on the enhanced data set
\begin{align}
\Xi = (v_0,v_1, f_1, f_2, f_3)
\label{enh1}
\end{align}
in the class:
\[
\mathcal{X}_T^{s,\eps} = \H^s (\T^2) \times L^{\frac 2\eps} ([-T,T]; W^{-\eps, \frac 2\eps}(\T^2))^3.
\]
\end{lemma}

By using the Strichartz estimates as in \cite{GKO, OTh2}, 
we can indeed prove
local well-posedness of \eqref{fNLW1}  for $\frac 14<s<1$.
Note that $s = \frac 14$ is the critical regularity as in \eqref{scrit}.
For simplicity, however, we only consider the case $\frac 12 <s<1$, 
 where the local well-posedness follows from a fixed point argument 
 with Sobolev's inequality and the product estimates (Lemma \ref{LEM:bilin}).

\begin{proof} The proof is essentially contained in Proposition 4.1 in \cite{GKOT}
and thus we will be brief here.
By writing  \eqref{fNLW1} in the Duhamel formulation, we have 
\begin{align*}
v(t)
& = \Phi_{\Xi} (v) (t) \\
& \hspace{-0.5mm}\stackrel{\text{def}}{=} S(t) (v_0,v_1) - \int_0^t \frac{\sin ((t-t') \jb{\nabla})}{\jb{\nabla}} \big( v^3 + 3f_1 v^2 + 3f_2 v + f_3 \big) (t') dt'.
\end{align*}

\noi
We will show that $\Phi_\Ld$ is a contraction mapping on a ball in $X_T^s$.

By  Sobolev's inequality, we have
\begin{align}
\begin{split}
\bigg\| \int_0^t \frac{\sin ((t-t') \jb{\nabla})}{\jb{\nabla}} v^3(t') dt' \bigg\|_{X^s_T}
&\les T \| v^3 \|_{L_T^\infty H_x^{s-1}}
\les T \| v^3 \|_{L_T^\infty L_x^{\frac{2}{2-s}}} \\
&\les T \| v \|_{L_T^\infty L_x^{\frac{6}{2-s}}}^3
\les T \| v \|_{L_T^\infty H_x^{\frac{1+s}{3}}}^3 \\
&\les T \| v \|_{X_T^s}^3
\end{split}
\label{dpLWP1}
\end{align}
for $\frac 12 \le s \le 1$.
From Lemma \ref{LEM:bilin} and Sobolev's inequality, we have
\begin{align}
\begin{split}
\bigg\| \int_0^t & \frac{\sin ((t-t') \jb{\nabla})}{\jb{\nabla}} (f_1v^2)(t') dt' \bigg\|_{X_T^s}\\
&\les \| \jb{\nb}^{-\eps} (f_1 v^2) \|_{L_T^1 L_x^2} \\
&\les T^{1-\frac \eps2} \| \jb{\nb}^{-\eps} f_1 \|_{L_T^{\frac 2\eps} L_x^{\frac 2\eps}} \| \jb{\nb}^\eps (v^2) \|_{L_T^\infty L_x^2} \\
&\les T^{1-\frac \eps2}  \| \jb{\nb}^{-\eps} f_1 \|_{L_T^{\frac 2\eps} L_x^{\frac 2\eps}} \| \jb{\nb}^\eps v \|_{L_T^\infty L_x^4} \| v \|_{L_T^\infty L_x^4} \\
&\les T^{1-\frac \eps2}  \| \jb{\nb}^{-\eps} f_1 \|_{L_T^{\frac 2\eps} L_x^{\frac 2\eps}} \| v \|_{X_T^s}^2, 
\end{split}
\label{dpLWP2}
\end{align}
provided that $\frac 12< s <1$ and $\eps =\eps(s)>0$ is sufficiently small.
Similarly, we have
\begin{align}
\bigg\| \int_0^t \frac{\sin ((t-t') \jb{\nabla})}{\jb{\nabla}} (f_2v) dt' \bigg\|_{X_T^s}
&\les T^{1-\frac \eps2}  \| \jb{\nb}^{-\eps} f_2 \|_{L_T^{\frac 2\eps} L_x^{\frac 2\eps}} \| \jb{\nb}^\eps v \|_{X_T^s}, \label{dpLWP3} \\
\bigg\| \int_0^t \frac{\sin ((t-t') \jb{\nabla})}{\jb{\nabla}} f_3(t') dt' \bigg\|_{X^s_T}
&\les T^{1-\frac \eps2}  \| \jb{\nb}^{-\eps} f_3 \|_{L_T^{\frac 2\eps} L_x^{\frac 2\eps}}.
\label{dpLWP4}
\end{align}

\noi
A standard argument with  \eqref{dpLWP1} - \eqref{dpLWP4} 
then shows that 
$\Phi_\Xi$ is a contraction on a small ball in $X^s_T$
by choosing $T = T(\|\Xi\|_{\mathcal{X}^{s, \eps}_1})> 0$
sufficiently small.
Moreover, a slight modification of the argument allows us to show continuous dependence 
of the solution on the enhanced data set $\Xi$ in \eqref{enh1}.
Since the argument is standard,
we omit details.
\end{proof}

\subsection{Proof of Theorem \ref{THM:uniq}}
\label{SUBSEC:4.3}

We conclude this section by presenting the proof of Theorem~\ref{THM:uniq}.
Set $v=u-z$, where $u$ is the solution constructed in Theorem~\ref{THM:OTh} and $z
= S(t) (u_0^\o,  u_1^\o)$ is as in   \eqref{lin1a}.
Let $v_{\rho,\dl}$ be the solution of \eqref{WNLW_PAK} with 
the mollified initial data  $(u_{0,\dl}^\o, u_{1,\dl}^\o)$  defined in \eqref{uniq1}
with a mollification kernel $\rho$.
Let $T>0$ and $\frac 12<s_0<1$.
In view of  Proposition~\ref{PROP:Z1rho}, it suffices to show that $v_{\rho,\dl}$ converges in probability to $v$ in $C([-T,T]; \H^{s_0} (\T^2))$.

\medskip

By Theorem~\ref{THM:OTh}, the global solution $u \in C(\R; H^s(\T^2))$ to \eqref{WNLW0} satisfies $v=u-z \in C(\R; H^{s_0}(\T^2))$,  almost surely.
In particular, from the construction of the global solution, for any $\eta>0$, there exists $R=R(T,\eta)\ge 1$ such that $\O_1 = \{ \o \in \O : \| (v,\dt v) \|_{L_T^{\infty} \H_x^{s_0}} \le R \}$ satisfies
\begin{align} \label{pO1}
P(\O_1^c)< \frac \eta4.
\end{align}

We divide the interval $[-T,T]$ into finitely many subintervals:
\[
[-T,T] = \bigcup_{j=-[ T/\tau]-1}^{[ T/\tau]} I_j,
\quad I_j = [j\tau, (j+1)\tau] \cap [-T,T],
\]
where $\tau>0$ is to be chosen later.
Let $\eps =\eps (s_0) >0$ be as in Lemma \ref{LEM:LWPx}.
We set
\[
\O_2 = \big\{ \o \in \O : \| :\! z^\l\!: \|_{L^{\frac 2\eps} (I_j;W^{-\eps, \frac 2\eps}(\T^2))} \le 1, \ \l =1,2,3, \ j= - \big[ \tfrac T\tau \big]-1, \dots, \big[ \tfrac T\tau \big] \big\}.
\]
By Lemma \ref{LEM:Z4} and taking $\tau = \tau (T,\eta)> 0$ small, we have
\begin{align}
P(\O_2^c)
&\le \sum_{\l=1}^3 \sum_{j=-[\frac{T}{\tau}]-1}^{[\frac{T}{\tau}]} P \Big( \| :\! z^\l\!: \|_{L^{\frac 2\eps} (I_j; W^{-\eps, \frac 2\eps}(\T^2))} > 1 \Big) \notag\\
&\les \sum_{\l=1}^3 \frac{T}{\tau} \exp \big( -c \tau^{-\frac{\eps}{\l}} \big) \notag\\
&\les \frac{T}{\tau} \tau \exp \Big( - \frac c2 \tau^{-\frac{\eps}{3}} \Big) \notag\\
&= T \exp \Big( - \frac c2 \tau^{-\frac{\eps}{3}} \Big)
<\frac \eta4. \label{pO2}
\end{align}

Moreover, we set
\[
\O_{3,\dl} = \Big\{ \o \in \O : \| :\! z^\l\!: - :\! z_{\rho,\dl}^\l\!: \|_{L_T^{\frac 2\eps} W_x^{-\eps, \frac 2\eps}} \le 8^{-\frac{T}{\tau}-5}, \ \l=1,2,3 \Big\}.
\]
From Proposition \ref{PROP:Z1rho}, there exists $\dl_0>0$ such that for any $0<\dl<\dl_0$, we have
\begin{align} \label{pO3}
P(\O_{3,\dl}^c) < \frac \eta4.
\end{align}
Then, we define $\O_{T,\eta,\dl} = \O_1 \cap \O_2 \cap \O_{3,\dl}$.
It follows from \eqref{pO1}, \eqref{pO2}, and \eqref{pO3} that
\begin{align} \label{pO4}
P(\O_{T,\eta,\dl}^c) <\frac 34 \eta.
\end{align}

Let  $w_{\rho,\dl} = v-v_{\rho,\dl}$. Then, $w_{\rho, \dl}$ satisfies
\begin{align*}
\begin{cases}
\dt^2 w_{\rho,\dl} + (1-  \Dl) w_{\rho,\dl} + \NN^3_{(u_0^\o, u_1^\o)} (v) - \NN^3_{(u_{0,\dl}^\o, u_{1,\dl}^\o)} (v_{\rho,\dl}) = 0 \\
(w_{\rho,\dl}, \dt w_{\rho,\dl}) |_{t = 0} = (0, 0),
\end{cases}
\end{align*}

\noi
where $\NN^3_{(u_0^\o, u_1^\o)} (v)$ is well defined thanks to Theorem~\ref{THM:OTh}. 
From \eqref{Herm4}, we have
\begin{align*}
\NN^3_{(u_0^\o, u_1^\o)} & (v) - \NN^3_{(u_{0,\dl}^\o, u_{1,\dl}^\o)} (v_{\rho,\dl})
\\
&= v^3-v_{\rho,\dl}^3 + 3 (v^2-v_{\rho,\dl}^2) z + 3 v_{\rho,\dl}^2 (z-z_{\rho,\dl}) + 3 w_{\rho,\dl} :\! z^2 \!:\, \\
&\quad + 3 v_{\rho,\dl} (:\! z^2 \!:\, - :\! z_{\rho,\dl}^2 \!:\,) + :\! z^3 \!:\, - :\! z_{\rho,\dl}^3 \!:\\
&= - 3v^2 w_{\rho,\dl} + 3v (2v - w_{\rho, \dl} ) w_{\rho,\dl} + w_{\rho,\dl}^3 +  3 (2v-w_{\rho,\dl})w_{\rho,\dl} z \\
&\quad + 3 (v^2-2vw_{\rho,\dl}+w_{\rho,\dl}^2) (z-z_{\rho,\dl}) + 3 w_{\rho,\dl} :\! z^2 \!:\, \\
&\quad + 3 (v-w_{\rho,\dl}) (:\! z^2 \!:\, - :\! z_{\rho,\dl}^2 \!:\,) + :\! z^3 \!:\, - :\! z_{\rho,\dl}^3 \!:\,.
\end{align*}

\noi
By taking $\tau = \tau (R) > 0$ sufficiently small, 
the local well-posedness argument in 
 the proof of Lemma~\ref{LEM:LWPx} yields 
\begin{align} \label{Xvj}
\| v \|_{X^{s_0} (I_j)} \le 2R
\end{align}
for $\o \in \O_{T,\eta,\dl}$ and $j=- \big[ \frac{T}{\tau} \big]-1, \dots, \big[ \frac{T}{\tau} \big]$.

In the following, we restrict our attention to positive times, 
i.e.~we work on $I_j$ for $j =0, \dots, \big[ \frac T\tau \big]$.
By applying the estimates \eqref{dpLWP1}, \eqref{dpLWP2}, \eqref{dpLWP3}, and \eqref{dpLWP4}
with \eqref{Xvj}
and taking $\tau = \tau (R) > 0$ sufficiently small, we have 
\begin{align}
\| w_{\rho,\dl} \|_{X^{s_0}(I_j)}
&\le  \| (w_{\rho,\dl}(j\tau), \dt w_{\rho,\dl}(j\tau)) \|_{\H^{s_0}} \notag\\
&\quad + C \tau^{1-\frac \eps2} 
\Big( ( R^2 +  \| w_{\rho,\dl} \|_{X^{s_0}(I_j)}^2)  \| w_{\rho,\dl} \|_{X^{s_0}(I_j)}\notag \\
& \qquad + ( R +  \| w_{\rho,\dl} \|_{X^{s_0}(I_j)})  \| w_{\rho,\dl} \|_{X^{s_0}(I_j)}
 \notag\\
&\qquad + (R^2 + \| w_{\rho,\dl} \|_{X^{s_0}(I_j)}^2) \| z - z_{\rho,\dl} \|_{L_T^{\frac 2\eps} W_x^{-\eps,\frac 2\eps}} \notag\\
&\qquad +  (R+\| w_{\rho,\dl} \|_{X^{s_0}(I_j)}) \| :\! z^2 \!:\, - :\! z_{\rho,\dl}^2 \!: \|_{L_T^{\frac 2\eps} W_x^{-\eps,\frac 2\eps}} \notag\\
&\qquad+ \| :\! z^3 \!:\, - :\! z_{\rho,\dl}^3 \!: \|_{L_T^{\frac 2\eps} W_x^{-\eps,\frac 2\eps}}
\Big) \notag\\
&\le  \| (w_{\rho,\dl}(j\tau), \dt w_{\rho,\dl}(j\tau)) \|_{\H^{s_0}} \notag\\
&\quad+ \frac 12 \sum_{\l=1}^3 \Big( \| w_{\rho,\dl} \|_{X^{s_0}(I_j)}^\l + \| :\! z^\l\!: - :\! z_{\rho,\dl}^\l\!: \|_{L_T^{\frac 2\eps} W_x^{-\eps,\frac 2\eps}} \Big)
\notag
\end{align}
for any $\o \in \O_{T,\eta,\dl}$ and $j=0, \dots, \big[ \frac{T}{\tau} \big]$.
By setting
\[
A =\sum_{\l=1}^3\| :\! z^\l\!: - :\! z_{\rho,\dl}^\l\!: \|_{L_T^{\frac 2\eps} W_x^{-\eps,\frac 2\eps}},
\]
we have
\begin{align}
\begin{split}
\| w_{\rho,\dl} \|_{X^{s_0}(I_j)}
& \le 2 \| (w_{\rho,\dl}(j\tau), \dt w_{\rho,\dl}(j\tau)) \|_{\H^{s_0}}\\
& \quad + \| w_{\rho,\dl} \|_{X^{s_0}(I_j)}^2 + \| w_{\rho,\dl} \|_{X^{s_0}(I_j)}^3 + A.
\end{split}
\label{wIj}
\end{align}
When $j=0$, since $(w_{\rho,\dl}(0), \dt w_{\rho,\dl}(0)) =(0,0)$ and $A< 8^{-3}$, 
a continuity argument yields
\[
\| w_{\rho,\dl} \|_{X^{s_0}(I_0)}
\le 2 A.
\]
In particular, we have $\| (w_{\rho,\dl}(\tau), \dt w_{\rho,\dl}(\tau)) \|_{\H^{s_0}} \le 2A$.
For $j=1,\dots, \big[ \frac T\tau \big]$,
since $A< 8^{-j-3}$ for $\o \in \O_{T,\eta,\dl}$,
we can repeatedly apply \eqref{wIj} and the continuity argument to obtain
\[
\| w_{\rho,\dl} \|_{X^{s_0}(I_j)}
\le 2 \cdot 8^j A.
\]
Hence, we have
\begin{align}
\| w_{\rho,\dl} \|_{L_T^{\infty} H_x^{s_0}}
\le 2 \cdot 8^{[\frac T \tau]+1} \sum_{\l=1}^3\| :\! z^\l\!: - :\! z_{\rho,\dl}^\l\!: \|_{L_T^{\frac 2\eps} W_x^{-\eps,\frac 2\eps}}.
\label{wLT}
\end{align}

Finally,  from \eqref{wLT} and Proposition \ref{PROP:Z1rho}, 
we see that for any $\ld>0$, there exists $\dl_1 \in (0,\dl_0)$ such that
\[
P( \{ \o \in \O_{T,\eta,\dl} : \| w_{\rho,\dl} \|_{L_T^{\infty} H_x^{s_0}}>\ld \})
< \frac \eta4
\]
for $0<\dl<\dl_1$.
Together with  \eqref{pO4}, 
we conclude that $w_{\rho,\dl}$ converges in probability to $0$ in $C([-T,T]; H^{s_0}(\T^2))$.
Recalling that  $w_{\rho,\dl} = v-v_{\rho,\dl}$, this concludes the proof
of Theorem~\ref{THM:uniq}.

\begin{remark}\label{REM:uniq2} \rm
Since $(\phi_{0,\eps}, \phi_{1,\eps})$ is smooth, Theorem \ref{THM:OTh} with 
the Cameron-Martin theorem~\cite{CM} implies
almost sure global existence of the solution $v_{\eps}$ to \eqref{WNLW4b};
see \cite{OQ}.
Moreover, for any $T>0$ and $\eta>0$, there exists $\wt{R}=\wt{R}(T,\eta,\phi_{0,\eps}, \phi_{1,\eps})$ such that
\[
P \big( \| (v_\eps, \dt v_\eps) \|_{L_T^{\infty} \H_x^{s_0}} > \wt{R} \big) < \frac \eta4\,.
\]
Then, we can use this bound instead of \eqref{pO1}
and repeat the argument presented above to conclude that 
 the solution $v_{\dl,\eps}$ to \eqref{WNLW4a} converges in probability to the solution $v_{\eps}$ to~\eqref{WNLW4b}.
\end{remark}

\section{Norm inflation 
for  the (unrenormalized) NLW in negative Sobolev spaces}
\label{SEC:NI}


In this section, we present the proof of Theorem~\ref{THM:illposed}, 
norm inflation for the cubic NLW~\eqref{NLW3} with $k = 3$.
In the remaining part of the paper, when we refer to \eqref{NLW3} (and \eqref{NLW0}), it is understood that $k = 3$.
Furthermore, 
for simplicity of the presentation, 
we set $m = 1$, 
where $m$ denotes 
the mass $m \geq 0$ in \eqref{NLW3}.
 Namely, we consider \eqref{NLW0} with $k = 3$.

We first state the following norm inflation result for smooth initial data.

\begin{proposition}\label{PROP:NI3}
Let $d\in \N$.
Suppose that  $s \in \R$ satisfies
either 
\textup{(i)} $s \leq  - \frac 12$ 
when $d = 1$
or 
\textup{(ii)} $s <  0$ 
when $d \geq 2$.
Fix $\vec u_0 = (u_0, u_1) \in \S(\M)\times \S(\M)$.
Then, 
given any $n \in \N$, 
there exist a  solution $u_n$ to the cubic NLW \eqref{NLW0} with $k = 3$
and $t_n  \in \big(0, \frac 1n\big) $ such that 
\begin{align}
 \big\| (u_n(0), \dt u_n(0)) - (u_0, u_1) \big\|_{\H^s(\M)} < \tfrac 1n \qquad \text{ and } 
\qquad \| u_n(t_n)\|_{H^s(\M)} > n.
\label{main1}
\end{align}

\end{proposition}

Once we prove
Proposition \ref{PROP:NI3}, 
Theorem \ref{THM:illposed} follows 
from the density 
of $\S(\M)\times \S(\M)$ in $\H^s(\M)$
and a diagonal argument.
See   \cite{Xia, Oh1, Ok}.
While the basic structure of the argument is the same as that  presented in \cite{Oh1}, 
we establish different multilinear estimates by exploiting 
one degree of smoothing in the Duhamel integral operator $\I$ in \eqref{Duhamel1} below.
In the following, we fix $\vec u_0 \in \S(\M)\times \S(\M)$
and may suppress 
the dependence of various constants on $\vec u_0$.

Before proceeding further, we introduce some notations.
Given $\M = \R^d$ or $\T^d$, 
let   $\ft \M$ denote the Pontryagin dual of $\M$, 
i.e.~
\begin{align}
\ft \M = \begin{cases}
 \R^d & \text{if }\M = \R^d, \\
 \Z^d & \text{if } \M = \T^d.
\end{cases}
\label{dual}
\end{align}

\noi
When $\ft \M = \Z^d$, 
we endow it with the counting measure.
We then define
the Fourier-Lebesgue space $\F L^{s,p}(\M)$  by the norm:
\begin{align*}
\|f \|_{\F L^{s,p}(\M)} = \big\| \jb{\xi}^s \ft f \big\|_{L^{p}(\ft \M)}.
\end{align*}

\noi
In particular, 
$\F L^1(\M) \stackrel{\text{def}}{=} \F L^{0,1}(\M)$ corresponds to 
the Wiener algebra.
We also define
\begin{align}
\overrightarrow{\F L}^{s, p}(\M) \stackrel{\text{\rm def}}{=} \F L^{s,p}(\M) \times \F L^{s-1,p}(\M).
\label{FL}
\end{align}

In Subsection \ref{SUBSEC:5.1}, we first go over local well-posedness 
of \eqref{NLW0} in the Wiener algebra $\overrightarrow{\F L}^{0,1}(\M)$.
Then, we express solutions in a  power series expansion in terms of initial data, 
where the summation ranges over all finite ternary trees.
We then establish basic nonlinear estimates
on the multilinear terms arising in the power series expansion
in Subsection \ref{SUBSEC:5.2}.
In Subsection \ref{SUBSEC:NI}, we present the proof of Proposition \ref{PROP:NI3}.

\subsection{Power series expansion indexed by trees}
\label{SUBSEC:5.1}

We define the Duhamel integral  operator
$\I$ by 
\begin{align}
\I[u_1, u_2, u_3](t)
\stackrel{\text{def}}{=}  
-\int_0^t \frac{\sin ((t - t') \jb{\nb})}{\jb{\nb}} [u_1 u_2 u_3](t')dt'.
\label{Duhamel1}
\end{align}

\noi
When all the three arguments $u_1, u_2$, and $u_3$ are identical, 
we use the following shorthand notation:
\begin{align}
\I^3[u] \stackrel{\text{def}}{=} \I[u, u, u].
\label{Duhamel2}
\end{align}

\noi
We say that $u$ is a solution to \eqref{NLW0}
with $(u, \dt u)|_{t = 0} = (u_0, u_1)$
if $u$ satisfies the following Duhamel formulation:
\begin{align}
u(t) = S(t) (u_0, u_1) + \I^3[u](t).
\label{Duhamel3}
\end{align}

\noi
We first state the local well-posedness of \eqref{NLW0} 
in $\overrightarrow{\F L}^{0, 1}(\M)$.

\begin{lemma}\label{LEM:LWP}
The cubic NLW \eqref{NLW0} with $k = 3$ is locally well-posed
in  $\overrightarrow{\F L}^{0, 1}(\M)$.
More precisely, given $\vec u_0 = (u_0, u_1) \in \overrightarrow{\F L}^{0, 1}(\M)$, 
there exist $T \sim \| \vec u_0\|_{\overrightarrow{\F L}^{0, 1}}^{-1}>0 $ 
and a unconditionally unique solution $ u \in C([-T, T]; \F L^1(\M))$, 
satisfying \eqref{Duhamel3}.

\end{lemma}

The unconditional uniqueness refers to the uniqueness of solutions
in the entire $ C([-T, T]; \F L^1(\M))$.
Unconditional uniqueness is a concept of uniqueness which does not depend on how solutions are constructed.

In view of the boundedness of $S(t)$ in $\overrightarrow{\F L}^{0, 1}(\M)$ and 
the algebra property of $\F L^1(\M)$
together with the bound:
\begin{align}
\int_0^t \frac{|\sin ((t - t') \jb{\xi})|}{\jb{\xi}} dt'
\le C t^2
\label{Hs21b}
\end{align}

\noi
uniformly in $\xi \in \ft \M$ (also see \eqref{X1} below)
Lemma \ref{LEM:LWP} follows
from a standard fixed point argument.
We omit details.

Let $\vec u_0\in \overrightarrow{\F L}^{0, 1}(\M)$.
Then, (the proof of) Lemma \ref{LEM:LWP}   guarantees the convergence of the following Picard iteration
scheme: 
\begin{align}
P_0(\vec \phi) = S(t) \vec u_0
\qquad \text{and}\qquad
P_{j}(\vec u_0) = S(t) \vec u_0 + \I^3[P_{j-1}(\vec u_0)], \ \  j \in \N, 
\label{power2}
\end{align}

\noi
at least for short times. 
It follows from  \eqref{Duhamel2} and \eqref{power2}  that $P_j$ consists of multilinear terms
of degrees at most $3^j$ (in $\vec u_0$).
In the following, we discuss a more general recursive scheme
and express a solution in a power
series indexed by trees
as in \cite{CH2, Oh1}.
We introduce the following notion of (ternary) trees.
Our trees 
refer to a particular subclass of usual trees with the following properties:

\begin{definition} \label{DEF:tree} \rm
(i) Given a partially ordered set $\TT$ with partial order $\leq$, 
we say that $b \in \TT$ 
with $b \leq a$ and $b \ne a$
is a child of $a \in \TT$,
if  $b\leq c \leq a$ implies
either $c = a$ or $c = b$.
If the latter condition holds, we also say that $a$ is the parent of $b$.

\smallskip 

\noi
(ii) 
A tree $\TT$ is a finite partially ordered set,  satisfying
the following properties:
\begin{itemize}
\item Let $a_1, a_2, a_3, a_4 \in \TT$.
If $a_4 \leq a_2 \leq a_1$ and  
$a_4 \leq a_3 \leq a_1$, then we have $a_2\leq a_3$ or $a_3 \leq a_2$.

\item
A node $a\in \TT$ is called terminal, if it has no child.
A non-terminal node $a\in \TT$ is a node 
with  exactly three children.

\item There exists a maximal element $r \in \TT$ (called the root node) such that $a \leq r$ for all $a \in \TT$.

\item $\TT$ consists of the disjoint union of $\TT^0$ and $\TT^\infty$,
where $\TT^0$ and $\TT^\infty$
denote  the collections of non-terminal nodes and terminal nodes, respectively.
\end{itemize}

%
%
%
%

\end{definition}

Note that the number $|\TT|$ of nodes in a tree $\TT$ is $3j+1$ for some $j \in \mathbb{N}\cup\{0\}$,
where $|\TT^0| = j$ and $|\TT^\infty| = 2j + 1$.
Let us denote  the collection of trees of  $j$ generations (i.e.~with $j$ parental nodes) by $\BT(j)$, i.e.
\begin{equation*}
\BT(j) \stackrel{\text{def}}{=} \big\{ \TT : \TT \text{ is a tree with } |\TT| = 3j+1 \big\}.
\end{equation*}

Recall 
the following exponential bound
on the number 
$\# \BT(j)$ of 
 trees of  $j$ generations.
See \cite{Oh1} for a proof.

\begin{lemma}\label{LEM:tree}
Let $\BT(j)$ be as above.
Then, there exists $C >0$ such that 
\begin{align}
\# \BT(j) \leq C^j
\notag
\end{align}

\noi
for all $j \in \mathbb{N}\cup\{0\}$.
\end{lemma}

Next, we express the solution $u$ constructed in Lemma \ref{LEM:LWP}
in a power series indexed by trees.
Fix $\vec u_0 \in \overrightarrow{\F L}^{0, 1}(\M)$.
Given a tree $\TT \in \BT(j)$,
$j \in \N \cup\{0\}$, 
we associate a multilinear
operator (in $\vec u_0$) by the following rules:
\begin{itemize}
\item Replace a non-terminal node ``\,$\<1'>$\,'' 
by the Duhamel integral operator $\I$ defined in \eqref{Duhamel1}
with its three children as arguments $u_1, u_2$, and $u_3$, 

\item Replace a terminal node ``\,$\<1>$\,'' 
by the linear solution $S(t) \vec u_0$. 
\end{itemize}

\noi
In the following, we denote this mapping
from $ \bigcup_{j = 0}^\infty \BT(j)$ to  $\mathcal{D}'(\M\times [-T, T])$
by $\Psi_{\vec u_0}$.

For example, 
 $\Psi_{\vec u_0}$ maps
the trivial  tree ``\,$\<1>$\,'',  consisting only of the root node
to the linear solution $S(t) \vec u_0$.
Namely, we have $\Psi_{\vec u_0}(\,\<1>\,) = S(t) \vec u_0$.
Similarly, we have 
\begin{align}
\Psi_{\vec u_0}( \<3>) & =
\I^3[S(t) \vec u_0],
\notag
\\
\Psi_{\vec u_0}\big( \<31>\big) & =
\I[\I^3[S(t) \vec u_0], S(t) \vec u_0, S(t) \vec u_0], 
\notag
\end{align}

\noi
where $\I^3$ is as in \eqref{Duhamel2}.
In view of the algebra property of $\F L^1(\M)$ 
along with the continuity and boundedness of $S(t)$, 
we have
$\Psi_{\vec u_0}(\TT) \in C([-T, T]; \F L^1(\M))$
for any tree $\TT$,  provided $\vec u_0 \in \overrightarrow{\F L}^{0, 1}(\M)$.
Note that, if  $\TT \in \BT(j)$, 
then $\Psi_{\vec u_0}(\TT) $ is  $(2j+1)$-linear in $\vec u_0$.

Lastly, we define $\Xi_j$ by 
\begin{align}
\Xi_j (\vec u_0)
\stackrel{\text{def}}{=}  \sum_{\TT \in \BT(j)} \Psi_{\vec u_0} (\TT).
\label{tree1}
 \end{align}

\noi
When $j = 0$ and $1$, we have 
\begin{align}
\Xi_0 (\vec u_0) = S(t) \vec u_0
\qquad \text{and}
\qquad
\Xi_1 (\vec u_0) = \I^3\big[S(t) \vec u_0\big].
\label{tree2}
\end{align}

\noi
Then, from Lemma \ref{LEM:tree}, \eqref{tree1}, 
the definition of $\Psi_{\vec u_0} (\TT)$, 
and Young's inequality
together with \eqref{Hs21b}, we obtain the following lemma.
See \cite{Oh1}.

\begin{lemma}\label{LEM:nonlin1}
There exists $C >0$ such that 
\begin{align*}
\|  \Xi_j (\vec u_0)(t) \|_{\F L^1}
& \leq C^j t^{2j} 
\| \vec u_0\|_{\overrightarrow{\F L}^{0, 1}}^{2j+1}.
\end{align*}
	
\noi
for 
all $\vec u_0 \in \overrightarrow{\F L}^{0, 1}(\M)$
and 
all $j \in \N$.
In particular, 
there exist $T \sim \| \vec u_0  \|_{\overrightarrow{\F L}^{0, 1}}^{-1}>0 $ 
such that the power series expansion:
\begin{align}
u
 = \sum_{j = 0}^\infty \Xi_j (\vec u_0)
 = \sum_{j = 0}^\infty \sum_{\TT \in \BT(j)} \Psi_{\vec u_0}(\TT)
\label{R3}
 \end{align}

\noi
converges in 
$C([-T, T]; \F L^1(\M))$.

\end{lemma}

It is easy to check that $u$ defined by the power series \eqref{R3}
is indeed a solution to the cubic NLW \eqref{NLW0}.
Then, thanks to the unconditional uniqueness of the solution 
constructed in Lemma \ref{LEM:LWP},
we conclude that the power series expansion \eqref{R3}
must agree with the solution 
constructed in Lemma \ref{LEM:LWP}.
Note that the time of local existence in 
Lemma \ref{LEM:LWP} and the time of convergence in Lemma \ref{LEM:nonlin1}
are of the same order
 $ \sim \| \vec u_0  \|_{\overrightarrow{\F L}^{0, 1}}^{-1}>0 $.

\subsection{Multilinear estimates}
\label{SUBSEC:5.2}

We first go over  our choice of initial data for proving Proposition \ref{PROP:NI3}.
Given $n \in \N$, 
fix $N = N(n) \gg 1$ (to be chosen later).
We define $\vec \phi_n = (\phi_{0, n}, \phi_{1, n})$ by setting
\begin{align}
\ft \phi_{0, n} (\xi)= R \sum_{j \in \{-2, -1, 1, 2\}}\ind_{j Ne_1+ Q_A} (\xi)
\qquad \text{and}\qquad 
 \phi_{1, n}  = N  \phi_{0, n},  
\label{phi1}
\end{align}
	
\noi
where $Q_A = \big[-\frac A2, \frac A2\big)^d$, 
$e_1 = (1, 0, \dots, 0)$, 
$R = R(N) \geq 1 $,  and $A = A(N)\gg 1$, satisfying
\begin{align}
 RA^{\frac{d}{2}}   \gg \|\vec u_0\|_{\overrightarrow{\F L}^{0, 1}}, 
\qquad \text{and}
\qquad A\ll N, 
\label{phi1a}
\end{align}

\noi
 are to be chosen later.
 Note that we have
\begin{align}
\| \vec \phi_n\|_{\H^s} \sim R A^\frac{d}{2} N^s
\qquad \text{and} \qquad 
\| \vec \phi_n\|_{\overrightarrow{\F L}^{0, 1}} \sim R A^d, 
\label{phi2}
\end{align}

\noi
for any $s \in \R$.
Lastly, given $\vec u_0 \in \S(\M) \times \S(\M)$, 
set $\vec{u}_{0, n} = (u_{0, n}, u_{1, n})$ by 
\begin{align}
\vec{u}_{0, n} = (u_{0, n}, u_{1, n}) 
& = (u_0, u_1) + (\phi_{0, n}, \phi_{1, n})\notag \\
& = \vec u_0 + \vec \phi_{n}.
\label{phi3}
\end{align}

Let $u_n$ be 
 the corresponding solution to \eqref{NLW0} with $(u_n, \dt u_n)|_{t = 0} = \vec u_{0,n}$.
Lemmas \ref{LEM:tree} and \ref{LEM:nonlin1} with \eqref{phi2} guarantee the convergence
of the following power series expansion:
\begin{align}
 u_n & 
 = \sum_{j = 0}^\infty \Xi_j (\vec u_{0, n}) 
  = \sum_{j = 0}^\infty \Xi_j (\vec u_{0} + \vec \phi_n), 
\label{tree3}
 \end{align}

\noi
on $[-T, T]$, 
as long as 
\begin{align}
T \les \big(\| \vec u_0\|_{\overrightarrow{\F L}^{0,1}} + R A^d\big)^{-1}
\sim (R A^d)^{-1},
\label{phi4}
\end{align}

\noi
where the last equivalence follows from \eqref{phi1a}.
Our main goal is to show that $u_n$ 
satisfies~\eqref{main1}
by estimating each of $\Xi_j (\vec u_{0, n})$ in 
the power series expansion \eqref{tree3}.

We now state the basic multilinear estimates.
Keep in mind that implicit constants in  Lemma~\ref{LEM:Hs1}
depend on (various norms of) $\vec u_0$.

\begin{lemma}\label{LEM:Hs1}

Let  $\vec u_{0, n}=  (u_{0, n}, u_{1, n})$  and 
$\vec \phi_n= (\phi_{0, n}, \phi_{1, n})$
be as in \eqref{phi3} and \eqref{phi1}.
Let $s < 0$.
Then, 
there exists $C>0$ such that 
\begin{align}
\| \vec u_{0, n} - \vec u_0\|_{\H^s} & \le C R A^{\frac {d}{2}}N^s, 
\label{Hs11}\\
\| \Xi_0(\vec u_{0,n})(t)\|_{H^s} & \le C( 1 + R A^{\frac {d}{2}}N^s),  \label{Hs12}\\
\| \Xi_1 (\vec u_{0, n})(t) - \Xi_1 ( \vec \phi_n)(t)\|_{H^s}
& \le C  t^2 \| \vec u_0\|_{\H^0} R^2A^{2d},
\label{Hs13}\\
\|  \Xi_1 ( \vec \phi_n)(t)\|_{H^s}
& \le C  t^2 R^3A^{2d}\cdot f(A) , 
\label{Hs14}\\
\| \Xi_j (\vec u_{0, n})(t)\|_{H^s}
& \leq 
C^j t^{2j} R^{2j+1} 
A^{2dj}
\cdot f(A) , 
\label{Hs21}
\end{align}
	
\noi
for any integer $j \geq 2$, 
where $f(A)$ is given by 
\begin{align}
f(A) = 
\begin{cases}
1, & \text{if } s <  -\frac d 2, \\
(\log A)^\frac 12, & \text{if } s =  -\frac d 2, \\
A^{\frac{d}{2}  +s},  & \text{if } s > -\frac d 2.
\end{cases}
\label{Hs21a}
\end{align}

\end{lemma}

This lemma in particular shows that the power series \eqref{tree3}
is convergent in $C([-T, T]; H^s(\M))$, provided
that $T^2 R^2 A^{2d} \ll 1$, 
which is consistent with \eqref{phi4}.

\begin{proof}
Recalling that $\vec \phi_n = \vec u_{0, n} - \vec u_0$, 
the first two estimates \eqref{Hs11} and \eqref{Hs12}
follow from~\eqref{phi2}
and the boundedness of $S(t)$ on $\H^s (\M)$.

Next, we prove \eqref{Hs13}.
In this case, we use the multilinearity of $\Xi_1$.
See \eqref{Duhamel1} and \eqref{tree2}.
By the Cauchy-Schwarz inequality, \eqref{Hs21b}, and Young's inequality
with~\eqref{phi1a}, we have
\begin{align*}
\| \Xi_1 (\vec u_{0, n})(t)  - \Xi_1 ( \vec \phi_n)(t)\|_{H^s}
&  \le \| \Xi_1 (\vec u_{0, n})(t) - \Xi_1 ( \vec \phi_n)(t)\|_{L^2}\notag\\
& \les  t^2 \|\vec u_0\|_{\H^0}\big(\|\vec u_0\|_{\overrightarrow{\F L}^{0, 1}}^2+ \| \vec \phi_n\|_{\overrightarrow{\F L}^{0, 1}}^2\big)\notag\\
& \les t^2 \|\vec u_0\|_{\H^0}(1 + R^2A^{2d})\notag\\
& \les t^2 \|\vec u_0\|_{\H^0} R^2A^{2d}.
\end{align*}

Lastly, we consider \eqref{Hs14} and \eqref{Hs21}.
It follows from the definition \eqref{phi1}
that 
 $\supp \F [S(t) \vec \phi_n]$
consists of four disjoint cubes of volume $\sim A^d$.
Given $\TT \in \BT(j)$, 
$\Psi_{\vec \phi_n}(\TT)$ is basically 
a $(2j+1)$-fold product of 
$S(t) \vec \phi_n$ under
iterated time integrations and spatial smoothing.
Hence, the
spatial support of $\F[ \Psi_{\vec \phi_n}(\TT)]$
consists of (at most) $4^{2j+1}$
cubes of volume $\sim A^d$.
Namely, we have 
\begin{align*}
\big|\supp \F[ \Psi_{\vec \phi_n}(\TT)]\big| \leq  C^j A^d = |  C_0^j Q_A|
\end{align*}

\noi
for some $C, C_0>0$.
Noting that, for $s < 0$, 
$\jb{\xi}^s$ is a decreasing function in  $|\xi|$, 
 we obtain
\begin{align}
\| \jb{\xi}^{s}\|_{L^2_\xi(\supp \F[\Psi_{\vec \phi_n} (\TT)])}
& \leq \| \jb{\xi}^{s}\|_{L^2_\xi( C_0^j Q_A)}
\les C^j f(A).
\label{Hs22}
\end{align}

\noi
 By \eqref{Hs21b} and Young's inequality, we have 
\begin{align}
\big\|   \I[u_1, u_2, u_3] (t) \big\|_{\F L^p} 
& \leq C  t^2\prod_{j = 1}^3 \|u_j\|_{\F L^{p_j}} , 
 \label{Hs22a}
\end{align}

\noi
for $1 \le p, p_1, p_2, p_3 \le \infty$, satisfying 
\begin{align*}
\frac 1p + 2 =  \frac{1}{p_1}+ \frac{1}{p_2}+ \frac{1}{p_3}.
\end{align*}

\noi
Then, by first applying \eqref{Hs22} and then
iteratively applying
\eqref{Hs22a}, 
we have 
\begin{align}
\big\| \Psi_{\vec \phi_n}(\TT)(t) \big\|_{H^s}
& \leq \| \jb{\xi}^{s}\|_{L^2_\xi(\supp \F[\Psi_{\vec \phi_n } (\TT)])}
\big\| \Psi_{\vec \phi_n }(\TT)(t) \big\|_{\F L^\infty} \notag\\
& \leq C^j t^{2j} f(A) 
\|\vec \phi_n \|_{\overrightarrow{\F L}^{0, 1}}
 \|\vec \phi_n\|_{\overrightarrow{\F L}^{0, q}}^{2j}\notag \\
& \leq C^j t^{2j}  R A^d\cdot 
R^{2j}  A^{\frac{2j}{q}d}f(A) \notag \\
&   =  
C^j t^{2j} 
R^{2j+1} A^{2dj}f(A), 
\label{Hs22c}
\end{align}

\noi
where $q$ satisfies
$
 2 j  =    \frac{2j}{q} + 1.
$
Hence,  it follows from \eqref{Hs22c} with \eqref{tree1} and Lemma \ref{LEM:tree}
that 
\begin{align}
\| \Xi_j (\vec  \phi_n)(t)\|_{H^s}
& \leq C^j t^{2j} 
R^{2j+1} A^{2d j}f(A) .
\label{Hs22d}
\end{align}

\noi
In particular, this proves \eqref{Hs14}.

Next, we estimate the difference 
$\Xi_j (\vec u_0 + \vec \phi_n) - \Xi_j ( \vec \phi_n)$.
Since  we do not know anything about 
the Fourier support of $\vec u_0$, 
we simply proceed with  a loss in the first step:
\begin{align}
\| \Xi_j (\vec u_0 + \vec \phi_n)(t)
 - \Xi_j (\vec  \phi_n)(t)\|_{H^s}
& \leq \| \Xi_j (\vec u_0 + \vec \phi_n)(t)
- \Xi_j ( \vec \phi_n)(t)\|_{L^2}.
\label{Hs24}
\end{align}

\noi
Then, with \eqref{tree1}, Lemma \ref{LEM:tree},
the multilinearity of $ \Psi_{\vec \phi} (\TT)$, 
and \eqref{Hs22a}, 
we have 
\begin{align}
\| \Xi_j (\vec u_0 + \vec \phi_n)(t)
 - \Xi_j ( \vec \phi_n)(t)\|_{L^2}
& \leq C^j t^{2j} 
\|\vec u_0 \|_{\overrightarrow{\F L}^{0, 1}}
\big(
\| \vec u_0\|_{\overrightarrow{\F L}^{0, r}}^{2j}
+ \| \vec \phi_n\|_{\overrightarrow{\F L}^{0, r}}^{2j}\big), 
\label{Hs24a}
\end{align}

\noi
where $r$ satisfies $2j=\frac{2j}{r}+\frac 12$.
Hence, from \eqref{Hs24} and \eqref{Hs24a}, 
we obtain 
\begin{align}
\| \Xi_j (\vec u_0 + \vec \phi_n)(t)
 - \Xi_j ( \vec \phi_n)(t)\|_{H^s}
& \leq C^j t^{2j} 
\|\vec u_0 \|_{\overrightarrow{\F L}^{0, 1}}
\big(
\| \vec u_0\|_{\overrightarrow{\F L}^{0, \wt{q}}}^{2j}
+ R^{2j} A^{2d  j -\frac{d}{2}}\big) \notag\\
& \leq C^j t^{2j} 
\|\vec u_0 \|_{\overrightarrow{\F L}^{0, 1}}
 R^{2j} A^{2d  j -\frac{d}{2}}, 
\label{Hs24c}
\end{align}

\noi
where the last step follows from  \eqref{phi1a}.
Therefore, 
the desired estimate \eqref{Hs21}
follows from~\eqref{Hs22d} and \eqref{Hs24c} with \eqref{phi1a}.
\end{proof}

Next, we state a crucial lemma,
establishing a lower bound on $\Xi_1(\vec \phi_n)$.
As in \cite{Kishimoto, Oh1, CP}, the argument exploits 
the high-to-low energy transfer mechanism in $\Xi_1(\vec \phi_n)$.

\begin{lemma}\label{LEM:Hs3}
Let
$\vec \phi_n = (\phi_{0, n}, \phi_{1, n})$  be as in \eqref{phi1} and $s < 0$.
Then, for $0 < t \ll N^{-1}$, we have
\begin{align}
\|  \Xi_1 ( \phi_n)(t)\|_{H^s}
\ges t^2  R^3 A^{2d} \cdot f(A),
\label{Hs31}
\end{align}

\noi
where $f(A)$ is as in  \eqref{Hs21a}.

\end{lemma}

\begin{proof}
From \eqref{Duhamel1}, we have 
\begin{align}
\begin{split}
\F & \big[  \Xi_1(\vec \phi_n)(t)\big](\xi)\\
& = - 
\intt_{\xi = \xi_1 + \xi_2 + \xi_3}
\int_0^t \frac{\sin ((t - t') \jb{\xi})}{\jb{\xi}} 
 \bigg(\prod_{j = 1}^3 \F\big[ S(t')\vec \phi_n\big](\xi_j)\bigg)
d\xi_1 d\xi_2 d\xi_3dt'.
\end{split}
\label{Hs32}
\end{align}

\noi
From the definition \eqref{phi1}, 
we have $|\xi_j| \les N$ for $\xi_j \in \supp \ft \phi_{k, n}$, $k = 0, 1$.
Then, for $0 < t \ll N^{-1} \ll1$, we have
\begin{align}
\begin{split}
\cos (t \jb{\xi_j}) & =  1 + O(t^2\jb{\xi_j}^2) > \frac{1}{2},\\
\frac t 2 <  \frac{\sin (t \jb{\xi_j})}{\jb{\xi_j}}  
&  = t + O (t^3 \jb{\xi_j}^2) \ll N^{-1}. 
\end{split}
\label{nonlin1}
\end{align}

\noi
Moreover, in view of 
$\xi = \xi_1 + \xi_2 + \xi_3$, we have 
\begin{align}
 \frac{\sin ((t-t') \jb{\xi})}{\jb{\xi}}  
&  = t - t' + O ((t-t')^3 \jb{\xi}^2) >\frac 12 (t- t')
\label{nonlin2}
\end{align}

\noi
for  $0  < t'< t \ll N^{-1} \ll1$.

Recalling that 
\begin{align*}
\ind_{a + Q_A}* \ind_{b + Q_A} (\xi)\ges A^d \ind_{a+b+Q_A}(\xi)
\end{align*}
		
\noi
for all $a, b, \xi \in \ft \M$ and $A \geq 1 $,  
where $\ft \M$ is as in \eqref{dual}, 
 it follows from \eqref{Hs32}, \eqref{nonlin1}, and \eqref{nonlin2}
 with \eqref{phi1}
that 
\begin{align*}
\big|\F\big[\Xi_1(\vec \phi_n)(t)\big](\xi)\big|
\ges t^2 R^3 A^{2d} \cdot \ind_{Q_A}(\xi).
\end{align*}

\noi
Lastly, 
 noting that
$\| \jb{\xi}^s\|_{L^2_\xi( Q_A)}
\sim f(A)$,
we obtain 
\eqref{Hs31}.
\end{proof}

\subsection{Proof of Proposition~\ref{PROP:NI3}}
\label{SUBSEC:NI}

We conclude this section by 
briefly discussing the proof of  Proposition \ref{PROP:NI3}.
As in \cite{Oh1}, it suffices to show that, 
given $n \in \N$, the following properties hold:
\begin{align*}
 \textup{(i)} & \quad R A^\frac{d}{2} N^s \ll \tfrac{1}{n}, \\ 
 \textup{(ii)} & \quad T^2 R^2 A^{2d} \ll 1,  \\ 
 \textup{(iii)} & \quad T^2 R^3 A^{2d} \cdot f(A) \gg n,\\ 
 \textup{(iv)} & \quad T^2 R^3 A^{2d} \cdot f(A)\gg T^4 R^5 A^{4d} \cdot f(A), \\
 \textup{(v)} & \quad T \ll N^{-1}, \\
 \textup{(vi)} & \quad R A^\frac{d}{2} \gg 1 
 \end{align*}

\noi
for some $A, R, T$, and $N$, depending on $n$.
Here, $f(A)$ is as in \eqref{Hs21a}.
As mentioned before, implicit constants depend on
(fixed) $\vec u_0 \in \S(\M)\times \S(\M)$.

We first show how the conditions (i) - (vi) 
imply Proposition \ref{PROP:NI3}.
This argument is essentially contained in \cite{Oh1}\footnote{Simply replace $T \ll N^{-2}$ in \cite{Oh1} by $T^2 \ll N^{-2}$ in our setting.}
but we include it for readers'  convenience.
The first condition (i) 
together with 
\eqref{Hs11} in Lemma \ref{LEM:Hs1} 
verifies the first estimate in \eqref{main1}.
The second condition~(ii) with \eqref{phi4}
guarantees local existence of the solution $u_n$
on $[-T, T]$ with $(u_n, \dt u_n)|_{t = 0} = \vec u_{0, n}$
and the convergence of the power series expansion \eqref{tree3}
in $C([-T, T]; \F L^1(\M))$. 
Moreover, assuming the conditions (ii) - (vi), 
it follows from Lemmas \ref{LEM:Hs1} and \ref{LEM:Hs3}
with the power series expansion \eqref{tree3}
that 
\begin{align*}
\| u_n(T)\|_{H^s} & \geq 
\|  \Xi_1 ( \vec \phi_n)(T)\|_{H^s}
- \| \Xi_0(\vec u_{0,n})\|_{H^s} \notag \\
& \hphantom{X} - \| \Xi_1 (\vec u_{0, n})(T) - \Xi_1 (\vec  \phi_n)(T)\|_{H^s}
-\bigg\|  \sum_{j = 2}^\infty \Xi_j (\vec u_{0, n}) (T)\bigg\|_{H^s} \notag\\
& \ges
T^2 R^3 A^{2d}\cdot f(A)
- (1 + R A^\frac{d}{2} N^s) \notag\\
&\hphantom{XXXX}- 
T^2R^2 A^{2d} \| \vec u_0\|_{\H^0}- T^4 R^5 A^{4d} \cdot f(A)\notag\\
& \sim 
T^2 R^3 A^{2d}\cdot f(A)
\gg n.
\end{align*}

\noi
This verifies the second estimate in \eqref{main1}
at time $t_n = T$.
Lastly, by choosing $N = N(n)$ sufficiently large, 
the  condition (v) guarantees that $t_n \in (0, \frac 1n)$.
This completes the proof of Proposition \ref{PROP:NI3}.	
	
\medskip

Therefore, it remains to 
verify the conditions (i) - (vi).
Note that the conditions (i) - (iv) are identical to those in
the Schr\"odinger case studied in  \cite{Oh1}
with $T^2 (\ll N^{-2})$ replaced $T\ll N^{-2}$.
Namely, we simply  use the same choices for $A$ and $R$ 
and the square root for the choice of $T$ from \cite{Oh1}.

\medskip

\noi
$\bullet$
{\bf Case 1:} $s < - \frac  d 2$.
\quad 
In this case, we set
\begin{align}
A = N^{\frac 1 d (1 - \dl)}, 
\quad 
R = N^{2 \dl}, 
\quad \text{and} 
\quad
T = N^{ - 1 - \frac 32 \dl}, 
\label{AR1}
\end{align}

\noi
where $\dl>0 $ is sufficiently small such that $s < - \frac 12 - \frac 32 \dl$.

\medskip

\noi
$\bullet$
{\bf Case 2:} $s  =  - \frac  d 2$.
\quad In this case, we set
\begin{align}
A = \frac{N^{\frac 1 d}}{(\log N )^{\frac 1{16d}}}, 
\quad R = 1, 
\quad \text{and} 
\quad
T = \frac 1{N (\log N)^{\frac 1{16}}} .
\notag
\end{align}

\medskip

\noi
$\bullet$
{\bf Case 3:} $ - \frac  d 2 < s < 0$.
\quad 
Recall that this case is relevant only for $d \geq 2$.
We set
\begin{align}
A = N^{\frac 2 d -\dl}, 
\quad 
R = N^{-1-s + \frac d 2\dl - \theta}, 
\quad \text{and} 
\quad
T= N^{ - 1 +s + \frac{1}{2}d \dl + \frac{1}{2} \theta}, 
\label{AR3}
\end{align}

\noi
where $\dl \gg  \theta  > 0$ are sufficiently small
such that 
\begin{align*} 
-2 s >d \dl + \theta
\qquad 
\text{and}
\qquad   -s \dl > 2 \theta.
\end{align*}

\noi
Then, by repeating the argument in \cite{Oh1}, 
we see that the conditions (i) - (vi) are satisfied in each case.

\begin{remark} \label{REM:NI0} \rm
It is easy to check that the choices in \eqref{AR1} 
of Case 1 is also valid for $s<-\frac 12$.
Namely,  Cases 1 and 3 are sufficient to prove Proposition \ref{PROP:NI3} for $d \ge 2$.
In particular, a logarithmic divergence as in Case 2 appears only when $d=1$, since $s=-\frac 12$ is the scaling critical regularity.
\end{remark}

\section{Almost sure norm inflation 
for the  Wick ordered  cubic NLW} 
\label{SEC:NI2}

In this section, we present the proof of Proposition \ref{PROP:ill3}
on almost sure norm inflation for the  Wick ordered  cubic NLW
on $\T^2$.
While the discussion in Section \ref{SEC:NI}
was for a general dimension $d \ge 1$, 
we restrict our attention to the two-dimensional case in this section.

\subsection{Local well-posedness of the Wick ordered NLW}
In this subsection, we briefly go over local well-posedness  of 
the perturbed NLW \eqref{WNLW5} on $\T^2$. More precisely, we consider 
\begin{align}
\begin{cases}
\dt^2 v  + (1 -  \Dl)  v  +  v^3 + \RR(v, z) = 0\\ 
(v, \dt v) |_{t = 0} = (\phi_{0}, \phi_{1}), 
\end{cases}
\label{WNLW5_pak}
\end{align}
\noi
where $\RR(v, z)$ is given by 
$$
\RR(v, z) = \, :\!  (z+v)^{3}\!\!:  \,  -\,  v^3  = 3 z  v^2 + 3 :\! z^2\!:   v+ :\! z^3\!:.
$$

\noi
In \cite{OTh2}, Thomann and the first author
proved almost sure local well-posedness
of \eqref{WNLW5_pak} via the Strichartz estimates and Lemma \ref{LEM:Z4}.
Note that, while only the zero initial data\footnote{This corresponds to the Wick ordered NLW \eqref{WNLW0}
with the random initial data \eqref{Gauss1}.} 
 for~\eqref{WNLW5_pak}
is considered in \cite{OTh2},
the same proof applies to any $(\phi_0, \phi_1) \in \H^s(\T^2)$, $s > s_\text{crit} = \frac 14$.
See also \cite{GKO}. On the other hand,  in proving Proposition~\ref{PROP:ill3}, 
we need to {\it maximize} the local existence time. In this respect, the Strichartz estimates are not very efficient.
In order to simultaneously handle the Wick powers and
 make the local existence time longer, we prove local well-posedness of \eqref{WNLW5_pak} 
in the Fourier-Lebesgue space $\overrightarrow{\F L}^{\al,\frac{1}{1-\al}}(\T^2)$ for sufficiently small $\al > 0$, where $\overrightarrow{\F L}^{\al,\frac{1}{1-\al}}(\T^2)$ is as in \eqref{FL}.
\begin{lemma}\label{LEM:LWP1}
Let $\al > 0$ be sufficiently small.
Then, the perturbed NLW \eqref{WNLW5_pak} is almost surely locally well-posed
in $\overrightarrow{\F L}^{\al,\frac{1}{1-\al}}(\T^2)$ on a time interval $[-T, T]$,
where
\begin{align}
T \ges \bigg\{\max\Big(  \|(\phi_0, \phi_1)\|_{\overrightarrow{\F L}^{\al,\frac{1}{1-\al}}}^{\frac{1}{1-\al}},  
\, K_\o \big(1 + \|  (\phi_0, \phi_1) \|_{\overrightarrow{\F L}^{\al,\frac{1}{1-\al}}} \big) \Big)\bigg\}^{-1}
\label{LWP1}
\end{align}

\noi
for some almost surely finite constant $K_\o >0$.
Moreover, we have
\begin{align*}
\sup_{t \in [-T, T]} 
  \|v(t) \|_{{\F L}^{\al,\frac{1}{1-\al}}}
  \les 
  \|(\phi_0, \phi_1)\|_{\overrightarrow{\F L}^{\al,\frac{1}{1-\al}}}.
\end{align*}

\end{lemma}

\begin{proof} 
Let $\I$ be the Duhamel integral operator defined in \eqref{Duhamel1}.
Then, by the algebra property of $\F L^1(\T^2)$
with \eqref{Hs21b}, 
we have
\begin{align}
\big\|\I [v_1, v_2, v_3] \big\|_{L_T^{\infty} \F L_x^1}
\les T^2 \prod_{j = 1}^3 \|v_j\|_{L_T^{\infty} \F L_x^1}.
\label{LWP2}
\end{align}

\noi
On the other hand, by Sobolev's inequality, we have
\begin{align}
\big\|\I [v_1, v_2, v_3] \big\|_{L_T^{\infty} H_x^{\frac 12}}
& \leq T \big\|\jb{\nb}^{-\frac{1}{2}} (u_1 u_2 u_3)\big\|_{L_T^{\infty}L_x^2}
\les T \big\| u_1 u_2 u_3 \big\|_{L_T^{\infty}L_x^{\frac{4}{3}}} \notag\\
& \le T \prod_{j = 1}^3 \|v_j\|_{L_T^{\infty} L_x^4}
\les T \prod_{j = 1}^3 \|v_j\|_{L_T^{\infty} H_x^\frac{1}{2}}.
\label{LWP3}
\end{align}

\noi
By the  interpolation of weighted $\l^p$-spaces applied to  \eqref{LWP2} and \eqref{LWP3}
with $\al = \theta \cdot \frac 12 + (1-\theta) \cdot 0 = \frac{\theta}{2}$, we obtain
\begin{align}
\big\|\I [v_1, v_2, v_3] \big\|_{L_T^{\infty} \F L_x^{\al,\frac{1}{1-\al}}}
\les T^{2(1-\al)} \prod_{j = 1}^3 \|v_j\|_{L_T^{\infty} \F L_x^{\al,\frac{1}{1-\al}}}
\label{LWP4}
\end{align}
for $0<\al<\frac 12$.

Next, we consider the terms
\begin{align}
\begin{split}
\1 & =  \int_0^t \frac{\sin ((t - t') \jb{\nb})}{\jb{\nb}}[ v_1v_2 z](t')dt', 
\\
\II & =  \int_0^t \frac{\sin ((t - t') \jb{\nb})}{\jb{\nb}}[ v :\! z^2\!: \,](t')dt',
\\
\III & =  \int_0^t \frac{\sin ((t - t') \jb{\nb})}{\jb{\nb}} :\! z^3 (t')\!:  dt'.
\end{split}
\label{LWP4a}
\end{align}

\noi
By Proposition \ref{PROP:Z1}, 
there exists an almost surely finite constant $K_\o >0$ such that 
\begin{align}
\|  :\! z^\l\!:  \|_{L^{\infty}([-1, 1]; W_x^{- \frac \al2, \infty})} \leq K_\o
\label{LWP5}
\end{align}

\noi
for $\l = 1, 2, 3$.
Let $0 < T \le 1$ in the following.
Note that Hausdorff-Young's inequality yields that 
$\F L^{\frac{1}{1-\al}}(\T^2) \embeds L^{\frac 1\al}(\T^2)$,
in particular, $\F L^{\frac{1}{1-\al}}(\T^2) \embeds L^4(\T^2)$ holds if $0<\al\le \frac 14$.
Then, it follows from
H\"older's inequality, Lemma \ref{LEM:bilin}, and \eqref{LWP5} with $\al < \frac 12$ that
\begin{align}
\| \1 \|_{L_T^{\infty} \F L_x^{\al,\frac{1}{1-\al}}}
&\leq T \big\| \jb{\nb}^{-1+\al}[ v_1v_2  z]\big\|_{L_T^{\infty} \F L_x^{0,\frac{1}{1-\al}}}
\les T \big\| \jb{\nb}^{-\frac \al2} [v_1v_2 z ] \big\|_{L_T^\infty L_x^2} \notag \\
&\les T \big\| \jb{\nb}^{\frac \al2} [v_1v_2] \|_{L_T^\infty L_x^2} \big\| \jb{\nb}^{-\frac \al2} z \big\|_{L_T^\infty L_x^{\frac 4\al}} \notag\\
&\les T K_\o \| v_1 \|_{L_T^{\infty} \F L_x^{\al,\frac{1}{1-\al}}} \| v_2 \|_{L_T^{\infty} \F L_x^{\al,\frac{1}{1-\al}}}.
\label{LWP6}
\end{align}
\noi
Similarly, 
by H\"older's inequality, Lemma \ref{LEM:bilin}, and \eqref{LWP5}, we have
\begin{align}
\| \II\|_{L_T^{\infty} \F L_x^{\al,\frac{1}{1-\al}}}
& \les T \big\| \jb{\nb}^{-\frac \al2} [v :\! z^2\!: \,]\big\|_{L_T^{\infty} L_x^2} \notag\\
&\les T \big\| \jb{\nb}^{\frac \al2} v \big\|_{L_T^\infty L_x^2} \big\| \jb{\nb}^{- \frac \al2} :\! z^2\!: \big\|_{L_T^{\infty} L_x^{\frac 4\al}} \notag \\
& \les T K_\o \|  v\|_{L_T^{\infty} \F L_x^{\al,\frac{1}{1-\al}}}
\label{LWP7}
\end{align}

\noi
and 
\begin{align}
\| \III\|_{L_T^{\infty} \F L_x^{\al,\frac{1}{1-\al}}}
&\les T \| \jb{\nb}^{-\frac \al2}  :\! z^3\!:\|_{L_T^{\infty} L_x^2}
\le T K_\o .
\label{LWP8}
\end{align}

By putting \eqref{LWP4}, \eqref{LWP6}, \eqref{LWP7}, and \eqref{LWP8}
together, 
a standard fixed point argument establishes almost sure local well-posedness of \eqref{WNLW5_pak},
provided that $T = T(\o)$ such sufficiently small such that 
\begin{align*}
 T^{2(1-\al)} \|(\phi_0, \phi_1)\|_{\overrightarrow{\F L}^{\al,\frac{1}{1-\al}}}^2 & \les 1, \\
 T K_\o \big(1 + \|  (\phi_0, \phi_1) \|_{\overrightarrow{\F L}^{\al,\frac{1}{1-\al}}} \big)& \les 1,
\end{align*}

\noi
yielding the condition \eqref{LWP1}.
\end{proof}
\subsection{Proof of Proposition \ref{PROP:ill3}}
In this subsection, we present the proof of 
Proposition~\ref{PROP:ill3}.
We prove this almost sure norm inflation result
by viewing \eqref{WNLW5}
as the (unrenormalized) NLW \eqref{NLW0}
with a random perturbation
and invoking the norm inflation result (Proposition \ref{PROP:NI3}) for 
the cubic NLW \eqref{NLW0}.

Let $s < 0$.
Given $n \in \N$,  fix $N = N(n) \gg1 $ 
to be chosen later.
Let $\vec \phi_n = (\phi_{0, n}, \phi_{1, n})$ be as in \eqref{phi1}
with  $A = A(N)$, $R = R(N)$, and $T = T(N) >0$
as in Subsection~\ref{SUBSEC:NI}.
Then, by taking some small $\al > 0$, we have
\begin{align}
T & \les\| \vec \phi_n\|_{\overrightarrow{\F L}^{\al,\frac{1}{1-\al}}}^{-\frac{1}{1-\al}}
= (RA^{2(1-\al)} N^{\al})^{-\frac{1}{1-\al}}
= (RN^\al)^{-\frac{1}{1-\al}} A^{-2}.
\label{Y1}
\end{align}

In fact, when $s < - \frac 12$, it follows from \eqref{AR1} and Remark \ref{REM:NI0} that
\begin{align} \label{etime0}
T(RN^\al)^{\frac{1}{1-\al}} A^2
= N^{-1-\frac 32\dl} N^{\frac{2\dl+\al}{1-\al}} N^{1-\dl}
= N^{-\frac{\dl-(2+5\dl)\al}{2(1-\al)}}.
\end{align}
Hence,
\eqref{Y1} holds, provided that $\al<\frac{\dl}{2+5\dl}$.
When $-\frac 12 \le s<0$,  \eqref{AR3} yields
\begin{align} \label{etime1}
T(RN^\al)^{\frac{1}{1-\al}} A^2
= N^{-1+s+\dl+\frac{\theta}{2}} N^{\frac{-1-s+\dl-\theta+\al}{1-\al}} N^{2-2\dl}
= N^{-\frac{\theta+(2s-2\dl+\theta)\al}{2(1-\al)}},
\end{align}

\noi
and hence 
\eqref{Y1} holds, provided that $\al<\frac{\theta}{-2s+2\dl-\theta}$.

Let $u = u(n)$ and $v = v(n)$ be the solutions
to 
the unrenormalized NLW \eqref{NLW0} and 
 the perturbed NLW \eqref{WNLW5} with the  initial data $\vec \phi_n$, 
 respectively.
Then, the above observation guarantees that
$u$ and $v$ exist on $[-T, T]$.
Moreover, in view of \eqref{phi4}, the power series expansion \eqref{tree3}
for $u_n$ (with $\vec u_0 = 0$)
converges uniformly on $[-T, T]$.
Then, it follows from Proposition \ref{PROP:NI3} that
\begin{align}
\| u(T) \|_{H^s} \gg  n
\label{Y4}
\end{align}

\noi
for suitably chosen $N = N(n, \o) \gg1 $.
Therefore, Proposition \ref{PROP:ill3} 
 follows from \eqref{Y4} 
 once we prove  the following approximation lemma.

\begin{lemma}\label{LEM:approx}
Given $n \in \N$,
let $u = u(n)$, $v = v(n)$, and $T=T(n)$ be as above.
Namely, they are
the solutions
to the unrenormalized NLW \eqref{NLW0} and 
 the perturbed NLW \eqref{WNLW5} with the  initial data $\vec \phi_n$, 
 respectively.
Then, there exists $C_\o > 0$, almost surely  tending to $0$ as $n \to 0$
 \textup{(}and hence $N \to \infty$\textup{)}, 
such that 
\begin{align}
\sup_{t \in [-T, T]} \| u(t) - v(t) \|_{L^2_x} \leq C_\o.
\label{X-1}
\end{align}

\end{lemma}

\begin{proof}
By our choice of $T$ in \eqref{AR1} and \eqref{AR3}, \eqref{etime0}, \eqref{etime1},
and the local well-posedness
of \eqref{NLW0} and \eqref{WNLW5}
in $\overrightarrow{\F L}^{0, 1}(\T^2)$ and $\overrightarrow{\F L}^{\al, \frac{1}{1-\al}}(\T^2)$, respectively, 
there exists $\eps>0$ such that
\begin{align}
T^{2}\|u\|^2_{L_T^{\infty} \F L^1_x} & +  T^{2(1-\al)} \| v \|^2_{L_T^{\infty} \F L_x^{\al,\frac{1}{1-\al}}}\notag\\
& \les T^{2}\|\vec\phi_n\|^2_{\overrightarrow{\F L}^{0, 1}}+  T^{2(1-\al)} \| \vec \phi_n \|^2_{\overrightarrow{\F L}^{\al, \frac{1}{1-\al}}}
\ll T^{2\eps}.
\label{X0}
\end{align}

By  Young's inequality with  \eqref{Hs21b}, we have
\begin{align}
\big\|\I[u_1, u_2, u_3]\big\|_{L_T^{\infty} L^2_x}
& \le T^{2} \|u_1\|_{L_T^{\infty} L^2_x} \prod_{j = 2}^3 \| u_j \|_{L_T^{\infty}\F L^1_x}.
\label{X1}
\end{align}

\noi
By H\"older's and Young's inequalities with a variant of \eqref{Hs21b}, we have
\begin{align}
\big\|\I[u_1, u_2, u_3]\big\|_{L_T^{\infty} L^2_x}
& \les \bigg\|\int_0^t \frac{\sin ((t - t') \jb{\nb})}{\jb{\nb}^{1-2\al-2\eps}} [u_1u_2u_3](t')dt'
\bigg\|_{L_T^{\infty} \F L^{\frac{2}{1-2\al-\eps}}_x}\notag \\
& \les T^{2- 2\al-2\eps} \|u_1\|_{L_T^{\infty} L^2_x}\prod_{j = 2}^3  \|u_j\|_{L_T^{\infty} \F L^\frac{4}{4-2\al-\eps}_x} \notag\\
& \les T^{2- 2\al-2\eps} \|u_1\|_{L_T^{\infty} L^2_x}\prod_{j = 2}^3  \|u_j\|_{L_T^{\infty} \F L^{\al,\frac{1}{1-\al}}_x}.
\label{X2}
\end{align}

\noi
Similarly, we have 
\begin{align}
\big\|\I[u_1, u_2, u_3]\big\|_{L_T^{\infty} L^2_x}
& \les \bigg\|\int_0^t \frac{\sin ((t - t') \jb{\nb})}{\jb{\nb}^{1-\al-\eps}} [u_1u_2u_3](t')dt'
\bigg\|_{L_T^{\infty} \F L^{\frac{4}{2-2\al-\eps}}_x}\notag \\
& \les T^{2- \al-\eps} \|u_1\|_{L_T^{\infty} L^2_x} 
\| u_2 \|_{L_T^{\infty}\F L^1_x} \| u_3 \|_{L_T^{\infty}\F L^\frac{4}{4-2\al-\eps}_x} \notag \\
& \les T^{2- \al-\eps} \|u_1\|_{L_T^{\infty} L^2_x} 
\| u_2 \|_{L_T^{\infty}\F L^1_x} \| u_3 \|_{L_T^{\infty}\F L^{\al, \frac{1}{1-\al}}_x}.
\label{X3}
\end{align}

\noi
Hence, it follows from \eqref{X1}, \eqref{X2},  and \eqref{X3} that 
\begin{align}
\big\|\I^3[u] - \I^3 [v]\big\|_{L_T^{\infty} L^2_x}
& \les
T^{-2\eps} \Big\{T^{2}\|u\|^2_{L_T^{\infty} \F L^1_x}+  T^{2(1-\al)} \| v \|^2_{L_T^{\infty} \F L^{\al, \frac{1}{1-\al}}_x}\Big\}
\|u - v\|_{L_T^{\infty} L^2_x} \notag\\
& \ll \|u - v\|_{L_T^{\infty} L^2_x} ,
\label{X4}
\end{align}
	
\noi
where the last inequality follows from \eqref{X0}.

Let $\1$, $\II$, and $\III$ be as in \eqref{LWP4a}.
Then, proceeding as in the proof of Lemma \ref{LEM:LWP1} with a variant of \eqref{Hs21b}, we obtain
\begin{align*}
\| \1\|_{L_T^{\infty}L^2_x}
& \les T^{2-\frac \al2} \big\| \jb{\nb}^{-\frac \al2} [v^2 z ] \big\|_{L_T^\infty L_x^2}
\les T^{2-\frac \al2} K_\o  \|  v\|^2_{L_T^{\infty} \F L^{\al, \frac{1}{1-\al}}_x}, \\ \| \II\|_{L_T^{\infty}L^2_x}
& \leq T^{2-\frac \al2} \big\| \jb{\nb}^{-\frac \al2} [ v :\! z^2\!:] \big\|_{L_T^\infty L_x^2}
\les T^{2- \frac \al2} K_\o \|  v\|_{L_T^{\infty}\F L^{\al, \frac{1}{1-\al}}_x}, \\%
\| \III\|_{L_T^{\infty}L^2_x}
& \leq T^{2-\frac \al 2} \big\| \jb{\nb}^{-\frac \al 2}  :\! z^3\!: \big\|_{L_T^{\infty}L^2_x} 
\les T^{2-\frac \al2 }  K_\o .\notag
\end{align*}

\noi
Hence,  in view of  \eqref{X0}, 
we can choose $n \gg1 $ (and hence $N \gg 1$ and $T \ll 1$) depending on $\o$
such that 
\begin{align}
\| \1\|_{L_T^{\infty}L^2_x} + \| \II\|_{L_T^{\infty}L^2_x}
+ \| \III\|_{L_T^{\infty}L^2_x}
& \ll 1.
\label{X5}
\end{align}

\noi
Finally,  noting that $u = \I^3[u]$ and $v = \I^3[v] + \1 + \II + \III$, 
the desired bound \eqref{X-1} follows
from \eqref{X4} and \eqref{X5}.
\end{proof}

\appendix

\section{On almost sure convergence of stochastic objects}
\label{SEC:A}

We present the proof of Proposition \ref{PROP:reg}.
First, we show the following lemma
which relates the decay in the hypothesis of Proposition \ref{PROP:reg}
to the boundedness of the relevant norms.

\begin{lemma} \label{LEM:A1}
Let $\{X_N\}$ and $X$ satisfy the assumption in Proposition \ref{PROP:reg}.

\smallskip
\noi\textup{(i)}
For $p\ge 1$, $s<s_0$, $t \in [0,T]$, and $N \ge 1$, we have
\begin{align}
\E \big[ \| X(t) \|_{W^{s,\infty}}^p \big]
&\les p^{\frac{kp}{2}},
\label{reg21} \\
\E \big[ \| X_N(t) - X(t) \|_{W^{s,\infty}}^p \big]
&\les p^{\frac{kp}{2}} N^{-\g p}. \label{reg22}
\end{align}

\noi

\smallskip
\noi\textup{(ii)}
For $p \ge 1$, $s<s_0-\frac \theta2$, $t \in [0, T]$,  $h \in [-1, 1]$, and $N \geq 1$, we have
\begin{align}
\E \big[ \| \dl_h X(t) \|_{W^{s,\infty}}^p \big]
&\les |h|^{\theta p}, \label{reg23} \\
\E \big[ \| \dl_h X_N(t) - \dl_h X(t) \|_{W^{s,\infty}}^p \big]
&\les N^{-\g p} |h|^{\theta p}.
\label{reg24}
\end{align}

\end{lemma}

\begin{proof}
We only consider the proof of \eqref{reg21} since the remaining estimates follow from the same argument with \eqref{reg2}, \eqref{reg3}, and \eqref{reg4}.

From $s<s_0$ and Sobolev's inequality, there exists finite $r>1$ such that $W^{\frac{s_0-s}{2},r}(\T^d) \hookrightarrow L^{\infty}(\T^d)$.
Then, it follows from Lemma \ref{LEM:hyp} that
\begin{align}
\big\| \| X(t) \|_{W^{s,\infty}} \big\|_{L^p(\O)}
&\les \Big\| \big\| \jb{\nb}^{s+\frac{s_0-s}{2}} X(t) \big\|_{L^p(\O)} \Big\|_{L^r} \notag \\
&\les p^{\frac k2} \Big\| \big\| \jb{\nb}^{s+\frac{s_0-s}{2}} X(t) \big\|_{L^2(\O)} \Big\|_{L^r}
 \label{A-1}
\end{align}
for $p \ge r$.

Now, note that the spatially homogeneity yields that
\begin{align} \label{A0}
\E \big[ \ft X(n_1, t) \ft X(n_2, t) \big] =0
\end{align}
if $n_1+n_2 \neq 0$.
Indeed,
we have
\begin{align*}
\E \big[ \ft X(n_1, t) \ft X(n_2, t) \big]
&= \int_{\T^d} \int_{\T^d} \E [ X(x_1,t) X(x_2,t) ] e^{-i(n_1 \cdot x_1 + n_2 \cdot x_2)} dx_1 dx_2 \notag\\
&= \int_{\T^d} \int_{\T^d} \E [ X(x_1,t) X(x_2,t) ] e^{-i (n_1+n_2) \cdot x_1 + i n_2 \cdot (x_1-x_2)} 
dx_1 dx_2.
\end{align*}
The spatially homogeneity implies that $\E [ X(x_1,t) X(x_2,t) ]$ is a function of $x_1-x_2$.
Then, by a change of variables $y_2 = x_1 - x_2$, we have, for some function $F$ on $\T^d$, 
\begin{align*}
\E \big[ \ft X(n_1, t) \ft X(n_2, t) \big]
&= \int_{\T^d} \ft F(n_2)  e^{-i (n_1+n_2) \cdot x_1 } dx_1 ,
\end{align*}

\noi
which vanishes unless $n_1+ n_2 = 0$.
We thus obtain \eqref{A0}.
Therefore, from  \eqref{A0} and 
\eqref{reg1}, we have 
\[
\big\| \jb{\nb}^{s+\frac{s_0-s}{2}} X(x,t) \big\|_{L^2(\O)}^2
\les \sum_{n \in \Z^d} \jb{n}^{s+s_0} \E \big[ |\ft X(n, t)|^2 \big]
\les \sum_{n \in \Z^d} \jb{n}^{-d+s-s_0}
\les 1.
\]
By combining this with \eqref{A-1}, we obtain \eqref{reg21}.
\end{proof}

We now present the proof of Proposition \ref{PROP:reg}.

\begin{proof}[Proof of Proposition \ref{PROP:reg}]

(i)
From \eqref{reg21}, we have $X(t) \in W^{s,\infty}(\T^d)$ almost surely.
Given $j \in \N$, 
it follows from Chebyshev's inequality and \eqref{reg22} that 
\[
\sum_{N = 1}^\infty
P\bigg( \| X_N(t) - X (t) \|_{W^{s,\infty}} > \frac{1}{j} \bigg)
\les 
\sum_{N = 1}^\infty e^{-c N^\frac{2\g}{k} j^{-\frac{2}{k}}}< \infty.
\]

\noi
Therefore, we conclude from the Borel-Cantelli lemma that 
there exists $\O_j$ with $P(\O_j) = 1$
such that 
for each $\o \in \O_j$, 
there exists $M = M(\o) \in \N$ such that 
$\| X_N(t; \o) - X(t;\o) \|_{W^{s,\infty}} <\frac{1}{j}$ for any $N \geq M$.
By setting $\Sigma = \bigcap_{j = 1}^\infty \O_j$, 
we have $P(\Si) = 1$.
Hence, we conclude that $X_N(t)$ converges almost surely to $X(t)$ in $W^{s,\infty}(\T^d)$.
Note that the set of almost sure convergence depends on $t\in [0, T]$
at this point.

\smallskip

\noi
(ii)
Next, we prove the second part of Proposition \ref{PROP:reg}.
By \eqref{reg23}, Kolmogorov's continuity criterion implies that 
$X \in C([0, T]; W^{s,\infty}(\T^d))$
almost surely.
We now modify the proof of Kolmogorov's continuity criterion
to prove  almost sure convergence of $\{X_N\}_{N \in \N}$
in $C([0,T]; W^{s,\infty}(\T^d))$.

In the following, fix $t \in [0, T]$ and  $h \in [-1, 1]$ (such that $t + h \in [0, T]$). 
We choose $p\gg 1$ such that 
\begin{align}
\theta p \ge 1+ \eps > 1\qquad \text{and} \quad \g p>2.
\label{K6}
\end{align}

\noi
Let $Y_N = X_N - X$.
Then, for any $\al > 0$, it follows from Chebyshev's inequality, \eqref{reg24},  and~\eqref{K6} that 
\begin{align*}
P\Big( & \sup_{N \in \N} \max_{j = 1, \dots, 2^\l}
N^{\frac{\g}{2}}
\big\|Y_N\big(\tfrac{j}{2^\l}\big) - Y_N\big(\tfrac{j-1}{2^\l}\big)\big\|_{W^{s,\infty}}
\geq 2^{-\al\l}
\Big)\notag\\
& = P\Big(\bigcup_{N \in \N} \bigcup_{j = 1}^{2^\l}
\big\|Y_N\big(\tfrac{j}{2^\l}\big) - Y_N\big(\tfrac{j-1}{2^\l}\big)\big\|_{W^{s,\infty}}
\geq N^{-\frac \g2} 2^{-\al\l}
\Big)\notag\\
& \leq  \sum_{N=1}^{\infty} \sum_{j = 1}^{2^\l}
P\Big(\big\|Y_N\big(\tfrac{j}{2^\l}\big) - Y_N\big(\tfrac{j-1}{2^\l}\big)\big\|_{W^{s,\infty}}
\geq N^{-\frac \g2} 2^{-\al\l}
\Big)\notag\\
& \leq  \sum_{N=1}^{\infty} \sum_{j = 1}^{2^\l}
 N^{\frac{\g p}{2}} 2^{\al p\l}
\, \E\Big[\big\|Y_N\big(\tfrac{j}{2^\l}\big) - Y_N\big(\tfrac{j-1}{2^\l}\big)\big\|_{W^{s,\infty}}^p\Big]\notag\\
& \les 2^{(\al p -\eps)\l} \sum_{N=1}^{\infty} N^{-\frac{\g p}{2}}
\les 2^{(\al p -\eps)\l}.
\end{align*}

\noi
Now, 
let  $\al \in (0, \frac{\eps}{p})$, i.e.~$\al p - \eps < 0$. 
Then, summing over $\l \in \N$, we obtain 	
\begin{align*}
\sum_{\l = 0}^\infty P\Big( & \sup_{N \in \N} \max_{j = 1, \dots, 2^\l}
N^{\frac \g2}
\big\|Y_N\big(\tfrac{j}{2^\l}\big) - Y_N\big(\tfrac{j-1}{2^\l}\big)\big\|_{W^{s,\infty}}
\geq 2^{-\al\l}
\Big) < \infty.
\end{align*}

\noi
Hence, by the Borel-Cantelli lemma, 
there exists a set $\wt \Sigma \subset \O$ with $P(\wt \Si) = 1$
such that, for each $\o \in \wt \Si$, we have 
\begin{align*}
\sup_{N \in \N} \max_{j = 1, \dots, 2^\l}
N^{\frac \g2}
\big\|Y_N\big(\tfrac{j}{2^\l};\o\big) - Y_N\big(\tfrac{j-1}{2^\l};\o\big)\big\|_{W^{s,\infty}}
\leq 2^{-\al\l}
\end{align*}

\noi
for all $\l \geq L = L(\o)$.
This in particular implies that there exists $C = C(\o)>0$ such that 
\begin{align}
\max_{j = 1, \dots, 2^\l}
\big\|Y_N\big(\tfrac{j}{2^\l};\o\big) - Y_N\big(\tfrac{j-1}{2^\l};\o\big)\big\|_{W^{s,\infty}}
\leq C(\o) N^{- \frac \g2} 2^{-\al\l}
\label{K7}
\end{align}

\noi
for any $\l \geq 0$,
uniformly in $N \in \N$.

For simplicity, let $T = 1$
and $t \in [0, 1]$.
Express $t$ in the following binary expansion:
\begin{align}
 t = \sum_{j = 1}^\infty \frac{b_j}{2^j}
\label{K7a}
 \end{align}

\noi
where $b_j \in \{0, 1\}$.
Let $t_\l = \sum_{j = 1}^\l \frac{b_j}{2^j}$
and $t_0 = 0$.
Then, from \eqref{K7}, we have
\begin{align}
\|Y_N(t; \o)\|_{W^{s,\infty}}
& \leq \sum_{\l = 1}^\infty  \|Y_N(t_\l;\o)- Y_N(t_{\l-1};\o)\|_{W^{s,\infty}} + \|Y_N(0;\o)\|_{W^{s,\infty}}\notag\\
& \leq C(\o) N^{- \frac \g2}
\sum_{\l = 1}^\infty
2^{-\al\l} + \|Y_N(0;\o)\|_{W^{s,\infty}}\notag\\
& \leq C'(\o) N^{- \frac \g2}
 + \|Y_N(0;\o)\|_{W^{s,\infty}},\label{K8}
\end{align}

\noi
for $\o \in \wt \Si$.
Note that the right-hand side of \eqref{K8} is independent of $t\in [0, 1]$.
Hence, by taking a supremum in $t \in [0, 1]$, we obtain 
\begin{align*}
\|X_N(\o) - X(\o)\|_{C([0, 1];W^{s,\infty}(\T^d))}
& \leq C'(\o) N^{- \frac \g2}
 + \|Y_N(0;\o)\|_{W^{s,\infty}}\notag\\
 & \too 0, 
\end{align*}

\noi
as $N \to \infty$.
Here, we used Part (i) of Proposition \ref{PROP:reg};
$Y_N(0) = X_N(0) - X(0)$ converges to $0$ in $W^{s,\infty}(\T^d)$, almost surely.
This yields almost sure convergence of $\{X_N\}_{N \in \N}$
in $C([0, 1]; W^{s,\infty}(\T^d))$, which completes the proof of Proposition \ref{PROP:reg}.
\end{proof}

\begin{remark}\rm
By slightly modifying the argument, 
we can also prove 
that  $X_N$ converges almost surely to $X$
in $C^\al([0, 1]; W^{s,\infty}(\T^d))$ for $\al < \frac{\eps}{p}$
(and hence $\al < \theta$ by taking $p \to \infty$ in view of \eqref{K6}).

Let $t, \tau \in [0, 1]$ such that $\frac{1}{2^{j-1}} \leq |t - \tau| \leq \frac{1}{2^{j}}$.
Express $t$ and $\tau$ in the binary expansions
\eqref{K7a} and
\begin{align*}
 \tau  = \sum_{j = 1}^\infty \frac{c_j}{2^j}
 \end{align*}

\noi
where $c_j \in \{0, 1\}$, 
and set
$\tau_\l = \sum_{j = 1}^\l \frac{c_j}{2^j}$.
Then, from \eqref{K7}, 
we have
\begin{align*}
\|Y_N(t; \o) - Y_N(\tau; \o)\|_{W^{s,\infty}}
& \leq \sum_{\l = j+1}^\infty  \|Y_N(t_\l;\o)- Y_N(t_{\l-1};\o)\|_{W^{s,\infty}} \notag\\
& \hphantom{X} + 
     \|Y_N(t_j;\o)- Y_N(\tau_{j};\o)\|_{W^{s,\infty}} \notag\\
& \hphantom{X} + 
 \sum_{\l = j+1}^\infty  \|Y_N(\tau_\l;\o)- Y_N(\tau_{\l-1};\o)\|_{W^{s,\infty}}\notag\\
& \leq C(\o) N^{- \frac \g2}
\sum_{\l = j}^\infty
2^{-\al\l} \notag\\
& \leq C'(\o) N^{- \frac \g2} 2^{-\al j}
\end{align*}

\noi
for $\o \in \wt \Si$.
Then, dividing both sides by  $2^{-\al j}$ and taking a supremum in $t \ne \tau$, we obtain
\[ \|X_N - X\|_{C^\al([0, 1]; W^{s,\infty}(\T^d))}  \leq C''(\o) N^{- \frac \g2},\]

\noi
which tends to 0 as $N \to \infty$.

\end{remark}

\begin{remark} \label{REM:cor} \rm
If $\{ X_N \}$ satisfies the assumption of Corollary \ref{COR:reg1}, then 
by proceeding as in the proof of Lemma \ref{LEM:A1}, 
we have

\smallskip
\noi\textup{(i)}
For $p \ge 1$, $s<s_0$, $t \in [0,T]$, and $M \ge N \ge 1$, we have
\begin{align}
\E \big[ \| X_N(t) \|_{W^{s,\infty}}^p \big]
&\les p^{\frac{kp}{2}},
\label{reg21c} \\
\E \big[ \| X_N(t) - X_M(t) \|_{W^{s,\infty}}^p \big]
&\les p^{\frac{kp}{2}} N^{-\g p}. \label{reg22c}
\end{align}

\smallskip
\noi\textup{(ii)}
For $p \ge 1$, $s<s_0-\frac \theta2$, $t \in [0, T]$,  $h \in [-1, 1]$, and $M \ge N \geq 1$, we have
\begin{align}
\E \big[ \| \dl_h X_N(t) \|_{W^{s,\infty}}^p \big]
&\les_p |h|^{\theta p}, \label{reg23c} \\
\E \big[ \| \dl_h X_N(t) - \dl_h X_M(t) \|_{W^{s,\infty}}^p \big]
&\les_p N^{-\g p} |h|^{\theta p}.
\label{reg24c}
\end{align}

It follows from  \eqref{reg21c} and \eqref{reg22c}
that $X_N(t)$ converges to some $X(t)$
in $L^p(\O;W^{s,\infty}(\T^d))$ and also in $W^{s, \infty}(\T^d)$, almost surely.
Moreover, \eqref{reg21} and \eqref{reg22} hold.
Then, by applying  Fatou's lemma applied to \eqref{reg23c} and \eqref{reg24c} 
(in taking $N \to \infty$), we obtain \eqref{reg23} and~\eqref{reg24}, 
which allows us to repeat the proof of Proposition \ref{PROP:reg}.

\end{remark}

\begin{ackno}\rm
T.O.~was supported by the European Research Council (grant no.~637995 ``ProbDynDispEq''
and grant no.~864138 ``SingStochDispDyn").
M.O.~was supported by JSPS KAKENHI Grant numbers
JP16K17624 and JP20K14342.
N.T.~was supported by the ANR grant ODA (ANR-18-CE40-0020-01).
T.O.~would like to thank 
the  Centre de Recherches Math\'ematiques, Montr\'eal, Canada, for its hospitality.
M.O.~would like to thank the School of Mathematics at the University of Edinburgh
for its hospitality during his visit in 2019, when this manuscript was prepared.
Lastly, the authors would like to 
thank the anonymous referees for their helpful comments
that have improved the presentation of the paper.
\end{ackno}


\begin{thebibliography}{99}






\bibitem{BT}
I.~Bejenaru, T.~Tao, 
{\it Sharp well-posedness and ill-posedness results for a quadratic non-linear Schr\"odinger equation,}
 J. Funct. Anal. 233 (2006), no. 1, 228--259.



\bibitem{BOP1}
\'A.~B\'enyi, T.~Oh, O.~Pocovnicu,
{\it  Wiener randomization on unbounded domains and an application to almost sure well-posedness of NLS,}  Excursions in harmonic analysis. Vol. 4, 3--25, Appl. Numer. Harmon. Anal., Birkh\"auser/Springer, Cham, 2015. 

\bibitem{BOP2}
\'A.~B\'enyi, T.~Oh, O.~Pocovnicu,
{\it On the probabilistic Cauchy theory of the cubic nonlinear Schr\"odinger equation on 
$\R^d$, $d \geq 3$}, Trans. Amer. Math. Soc. Ser. B 2 (2015), 1--50.


\bibitem{BonaT}
J.~Bona, N.~Tzvetkov, Nikolay, 
{\it Sharp well-posedness results for the BBM equation,}
Discrete Contin. Dyn. Syst. 23 (2009), no. 4, 1241--1252. 



\bibitem{BO94}
J.~Bourgain, 
{\it Periodic nonlinear Schr\"odinger equation and invariant measures}, 
Comm. Math. Phys. 166 (1994), no. 1, 1--26.



\bibitem{BO96}
J.~Bourgain, {\it Invariant measures for the 2$D$-defocusing nonlinear Schr\"odinger equation,} 
Comm. Math. Phys. 176 (1996), no. 2, 421--445.






\bibitem{Bring}
B.~Bringmann,
{\it Invariant Gibbs measures for the three-dimensional wave equation with a Hartree nonlinearity II: dynamics}, 
to appear in  J. Eur. Math. Soc. 


\bibitem{BDNY}

B.~Bringmann,  Y.~Deng, A.~R.~Nahmod, H.~Yue,
{\it Invariant Gibbs measures for the three dimensional 
cubic nonlinear wave equation,}
arXiv:2205.03893v1 [math.AP] 


\bibitem{BTT}
N.~Burq, L.~Thomann, N.~Tzvetkov, 
{\it Global infinite energy solutions for the cubic wave equation,}
 {Bull. Soc. Math. France.} 143 (2015), no. 2,  301--313.



\bibitem{BTT2}
N.~Burq, L.~Thomann, N.~Tzvetkov, 
{\it Remarks on the Gibbs measures for nonlinear dispersive equations,}
Ann. Fac. Sci. Toulouse Math.  27 (2018), no. 3, 527--597.

\bibitem{BT1}
N.~Burq, N.~Tzvetkov, 
{\it Random data Cauchy theory for supercritical wave equations. I. Local theory,}
 Invent. Math. 173 (2008), no. 3, 449--475.


\bibitem{BT2}
N.~Burq, N.~Tzvetkov, 
{\it Random data Cauchy theory for supercritical wave equations. II. A global existence result,} Invent. Math. 173 (2008), no. 3, 477--496.




\bibitem{BT3}
 N.~Burq, N.~Tzvetkov, 
 {\it Probabilistic well-posedness for the cubic wave equation,} J. Eur. Math. Soc.
(JEMS) 16 (2014), no. 1, 1--30.


\bibitem{CM}
R.~Cameron, W.~Martin, 
{\it Transformations of Wiener integrals under translations},  Ann. of  Math.  45 (1944) 386--396.




\bibitem{CC}
R.~Catellier, K.~Chouk, 
{\it Paracontrolled distributions and the 3-dimensional stochastic quantization equation},
  Ann. Probab. 46 (2018), no. 5, 2621--2679.



\bibitem{CP}
A.~Choffrut, 
O.~Pocovnicu, {\it 
Ill-posedness of the cubic half-wave equation 
and other fractional NLS
on the real line}, 
 Int. Math. Res. Not. IMRN 2018, no. 3, 699--738. 




\bibitem{CH2}
M.~Christ, {\it Power series solution of a nonlinear Schr\"odinger equation,} Mathematical aspects of nonlinear
dispersive equations, 131--155, Ann. of Math. Stud., 163, Princeton Univ. Press, Princeton, NJ,
2007.





\bibitem{CCT2b}
M.~Christ, J.~Colliander, T.~Tao, 
{\it Ill-posedness for nonlinear Schr\"odinger and wave equations}, 
arXiv:math/0311048 [math.AP].


\bibitem{CO}
J.~Colliander, T.~Oh, 
{\it  Almost sure well-posedness of the cubic nonlinear Schr\"odinger equation below $L^2(\T)$},
 Duke Math. J. 161 (2012), no. 3, 367--414. 

\bibitem{DPD}
G.~Da Prato, A.~Debussche,  {\it Strong solutions to the stochastic quantization equations},  Ann. Probab.  31 (2003), 1900--1916.

\bibitem{DPT1}
G.~Da Prato, L.~Tubaro, 
{\it Wick powers in stochastic PDEs: an introduction,} 
Technical Report UTM, 2006, 39 pp.


\bibitem{F1}
J.~Forlano, 
{\it Almost sure global well posedness for the BBM equation with infinite $L^2$ initial data}, 
 Discrete Contin. Dyn. Syst. 40 (2020), no. 1, 267--318.

\bibitem{F2}
J.~Forlano, M.~Okamoto, {\it A remark on norm inflation for nonlinear wave equations}, 
 Dyn. Partial Differ. Equ.  17 (2020), no. 4,  361--381. 

\bibitem{GV}
J.~Ginibre, G.~Velo,
{\it Generalized Strichartz inequalities for the wave equation},
J. Funct. Anal. 133 (1995), no. 1, 50--68.


\bibitem{GJ}
J.~Glimm, A.~Jaffe, 
{\it Quantum physics. A functional integral point of view,} Second edition. Springer-Verlag, New York, 1987. xxii+535 pp.



\bibitem{GKO}
M.~Gubinelli, H.~Koch, T.~Oh,
{\it  Renormalization of the two-dimensional stochastic nonlinear wave equations,}
Trans. Amer. Math. Soc. 370 (2018), no 10, 7335--7359. 


\bibitem{GKO2}
M.~Gubinelli, H.~Koch, T.~Oh,
{\it Paracontrolled approach to the three-dimensional stochastic nonlinear wave equation with quadratic nonlinearity,}
J. Eur. Math. Soc.
(2023). doi: 10.4171/JEMS/1294 

\bibitem{GKOT}
M.~Gubinelli, H.~Koch, T.~Oh, L.~Tolomeo,
{\it Global dynamics for the two-dimensional stochastic nonlinear wave equations,}
Int. Math. Res. Not. 
 2022, no. 21, 16954--16999. 


\bibitem{GO}
Z.~Guo, T.~Oh, 
{\it Non-existence of solutions for the periodic cubic nonlinear Schr\"odinger equation below $L^2$}, 
Internat. Math. Res. Not. 2018, no.6, 1656--1729. 





\bibitem{Hairer14}

M.~Hairer,
{\it  A theory of regularity structures},
 Invent. math.  198  (2014), 269--504



\bibitem{HZZ}
M.~Hofmanov\'a, R.~Zhu, X.~Zhu,
{\it 
Non-uniqueness in law of stochastic 3D Navier--Stokes equations},
arXiv:1912.11841 [math.PR].

\bibitem{IO}
T.~Iwabuchi, T.~Ogawa, 
{\it Ill-posedness for the nonlinear Schr\"odinger equation with quadratic non-linearity in low dimensions,} Trans. Amer. Math. Soc. 367 (2015), no. 4, 2613--2630. 


\bibitem{Kap}
L.~Kapitanski, 
{\it Weak and yet weaker solutions of semilinear wave equations}, 
Comm. Part. Diff. Eq. 19 (1994), 1629--1676.

\bibitem{KT}
M.~Keel, T.~Tao,
{\it Endpoint Strichartz estimates},
Amer. J. Math. 120 (1998), no. 5, 955--980.



\bibitem{Kishimoto}
N.~Kishimoto, {\it A remark on norm inflation for nonlinear Schr\"odinger equations}, 
 Commun. Pure Appl. Anal. 18 (2019), no. 3, 1375--1402.


\bibitem{Kuo}
H. Kuo, {\it Introduction to stochastic integration,} Universitext. Springer, New York, 2006. xiv+278 pp.


\bibitem{LS}
H.~Lindblad, C.~Sogge, {\it On existence and scattering with minimal regularity for semilinear wave equations,} J. Funct. Anal. 130 (1995), 357--426.









\bibitem{MW1}
J.-C.~Mourrat, H.~Weber,
{\it The dynamic $\Phi^4_3$ model comes down from infinity}, 
Comm. Math. Phys. 356 (2017), no. 3, 673--753.


\bibitem{MWX}
J.-C.~Mourrat, H.~Weber, W.~Xu,
{\it Construction of $\Phi^4_3$ diagrams for pedestrians},
 From particle systems to partial differential equations, 1--46, Springer Proc. Math. Stat., 209, Springer, Cham, 2017. 

\bibitem{Nelson2}
E.~Nelson, 
{\it A quartic interaction in two dimensions}, 
 1966 Mathematical Theory of Elementary Particles (Proc. Conf., Dedham, Mass., 1965) pp. 69--73 M.I.T. Press, Cambridge, Mass.



\bibitem{Nu}
D.~Nualart, 
{\it The Malliavin calculus and related topics,} Second edition. Probability and its Applications (New York). Springer-Verlag, Berlin, 2006. xiv+382 pp.



\bibitem{Oh1}
T.~Oh, {\it A remark on norm inflation with general initial data for the cubic nonlinear Schr\"odinger equations in negative Sobolev spaces},
 Funkcial. Ekvac. 60 (2017), 259--277. 
	
	
	
\bibitem{OOcomp}
T.~Oh, M.~Okamoto,
{\it  Comparing the stochastic nonlinear wave and heat equations: a case study},
Electron. J. Probab. 26 (2021), paper no. 9, 44 pp. 



\bibitem{OOR}
T.~Oh, M.~Okamoto, T. ~Robert,
{\it  A remark on triviality for the two-dimensional stochastic nonlinear wave equation}, Stochastic Process. Appl. 130 (2020), no. 9, 5838--5864. 	
	
	
	

\bibitem{OOT1}
T.~Oh, M.~Okamoto, L.~Tolomeo,
{\it 	
Focusing $\Phi^4_3$-model with a Hartree-type nonlinearity}, 
to appear in Mem. Amer. Math. Soc.

\bibitem{OOT2}
T.~Oh, M.~Okamoto, L.~Tolomeo,
{\it 	
Stochastic quantization of the $\Phi^3_3$-model,}
arXiv:2108.06777 [math.PR].



\bibitem{OP}
T.~Oh, O.~Pocovnicu, 
{\it Probabilistic global well-posedness of the energy-critical defocusing quintic nonlinear wave equation 
on $\R^3$}, J. Math. Pures Appl. 105 (2016), 342--366. 


\bibitem{OPTz}
T.~Oh, O.~Pocovnicu, N.~Tzvetkov,
{\it  Probabilistic local Cauchy theory of the cubic nonlinear wave equation in negative Sobolev spaces}, 
 Ann. Inst. Fourier (Grenoble)
72 (2022) no. 2, 771--830. 

\bibitem{OQ}
T.~Oh, J.~Quastel,
{\it  On Cameron-Martin theorem and almost sure global existence,} Proc. Edinb. Math. Soc. 59 (2016), 483--501. 
	

\bibitem{ORT}
T.~Oh, G.~Richards, L.~Thomann,
{\it
On invariant Gibbs measures for the generalized KdV equations},
Dyn. Partial Differ. Equ. 13 (2016), no. 2, 133--153. 




\bibitem{OS}
T.~Oh, C.~Sulem,
{\it  On the one-dimensional cubic nonlinear Schr\"odinger equation below $L^2$,} Kyoto J. Math. 52 (2012), no.1, 99--115. 



\bibitem{OTh1}
T.~Oh, L.~Thomann, 
{\it 
 A pedestrian approach to the invariant Gibbs measure for 
the 2-$d$ defocusing nonlinear Schr\"odinger equations}, 
Stoch. Partial Differ. Equ. Anal. Comput. 6 (2018), 397--445.


\bibitem{OTh2}
T.~Oh, L.~Thomann, 
{\it Invariant Gibbs measures for the 2-$d$ defocusing nonlinear wave equations,} 
Ann. Fac. Sci. Toulouse Math.
 29 (2020), no. 1, 1--26. 
 
 \bibitem{OTzW}
T.~Oh,  N.~Tzvetkov, Y.~Wang, 
 {\it Solving the 4NLS with white noise initial data}, 
 Forum Math. Sigma. 8 (2020), e48, 63 pp.

 

\bibitem{OW}
T.~Oh, Y.~Wang,
{\it  Global well-posedness of the periodic cubic fourth order NLS in negative Sobolev spaces},
Forum Math. Sigma 6 (2018), e5, 80 pp. 


\bibitem{OW2}
T.~Oh, Y.~Wang,
{\it Normal form approach to the one-dimensional periodic cubic nonlinear Schr\"odinger equation in almost critical Fourier-Lebesgue spaces},  
J. Anal. Math.
143 (2021), no. 2, 723--762. 

\bibitem{Ok}
M.~Okamoto, {\it Norm inflation for the generalized Boussinesq and Kawahara equations,}
Nonlinear Anal. 157 (2017), 44--61. 


\bibitem{Poc}
O.~Pocovnicu, 
{\it  Almost sure global well-posedness for the energy-critical defocusing nonlinear wave equation 
on $\R^d$,  $d = 4$ and $5$}, J. Eur. Math. Soc. 19 (2017), 2321--2375. 

\bibitem{Simon}
B.~Simon, 
{\it  The $P(\varphi)_2$ Euclidean (quantum) field theory,} Princeton Series in Physics. Princeton University Press, Princeton, N.J., 1974. xx+392 pp.



\bibitem{STz}
C.~Sun, N.~Tzvetkov,
{\it 
Concerning the pathological set in the context of probabilistic well-posedness}, 
 C. R. Math. Acad. Sci. Paris 358 (2020), no. 9-10, 989--999. 

\bibitem{Tao1}
T.~Tao, 
{\it Low regularity semi-linear wave equations}, 
Comm. Partial Differential Equations 24 (1999), no. 3-4, 599--629.


\bibitem{TAO}
 T.~Tao, {\it Nonlinear dispersive equations. Local and global analysis,} CBMS Regional Conference Series in Mathematics, 106. Published for the Conference Board of the Mathematical Sciences, Washington, DC; by the American Mathematical Society, Providence, RI, 2006. xvi+373 pp.






\bibitem{TTz}
L.~Thomann, N.~Tzvetkov, 
{\it Gibbs measure for the periodic derivative nonlinear Schr\"odinger equation},
Nonlinearity 23 (2010), no. 11, 2771--2791.





\bibitem{Tz1}
N.~Tzvetkov, {\it Random data wave equations}, 
Singular random dynamics, 221--313, Lecture Notes in Math., 2253, Fond. CIME/CIME Found. Subser., Springer, Cham, 2019.



\bibitem{TzV}
N.~Tzvetkov, N.~Visciglia, 
{\it Two dimensional nonlinear Schr\"odinger equation with spatial white noise potential and fourth order nonlinearity}, Stoch. Partial Differ. Equ. Anal. Comput. (2022). 
doi: https://doi.org/10.1007/s40072-022-00251-z


\bibitem{Xia}
B.~Xia, 
{\it Generic illposedness for wave equation of power type on three-dimensional torus},
 Int. Math. Res. Not. IMRN 2021, no. 20, 15533--15554.	








\end{thebibliography}
\end{document}